\documentclass[reqno,11pt]{amsart}
\usepackage{hyperref}
\usepackage{amsmath}
\usepackage{amsfonts}
\usepackage{amsthm}
\usepackage{amssymb}
\usepackage{amscd}
\usepackage{epic}
\usepackage{eepic}
\usepackage{relsize}
\usepackage{bigints}
\usepackage{enumitem,linegoal}
\usepackage{calc}
\usepackage[ansinew]{inputenc}
\usepackage{amscd,amssymb,verbatim}
\usepackage{marginnote}\usepackage{color}
\usepackage[mathscr]{euscript}

\newcommand{\printname}[1]
  {\smash{\makebox[0pt]{pace{-2.0in}\raisebox{8pt}{\tiny #1}}}}
\newtheorem {Lemma} {Lemma} [section]
\newtheorem {lemma} {Lemma} [section]

\newtheorem {Theorem}[Lemma]{Theorem}
\newtheorem {theorem}[Lemma]{Theorem}
\newtheorem {Proposition}[Lemma]{Proposition}
\newtheorem {Corollary}[Lemma]{Corollary}

\theoremstyle{definition}
\newtheorem{Definition}[Lemma]{Definition}
\theoremstyle{remark}
\newtheorem{Remark}[Lemma]{Remark}
\newtheorem{remark}[Lemma]{Remark}

\newcommand {\cal}[1]  {{\mathcal{#1}}}

\newcommand{\nab}[1][]{\ensuremath{\mathrm{\nabla}{#1}}}

\newcommand{\0}{\operatorname{\mathrm{\underline{0}}}}

\newcommand{\be}{\begin{equation}}
\newcommand{\ee}{\end{equation}}

\newcommand*\xbar[1]{%
  \hbox{%
    \vbox{%
      \hrule height 0.5pt 
      \kern0.5ex
      \hbox{%
        \kern-0.1em
        \ensuremath{#1}%
        \kern-0.1em
      }%
    }%
  }%
}

\newcommand{\del}{\partial}
\newcommand{\lb}[1]{\label{#1}}

\newcommand{\al}{\alpha}
\newcommand{\om}{\omega}

\newcommand{\we}{\wedge}

\newcommand{\lam}{\lambda}

\newcommand{\Ref}[1]{(\ref{#1})}
\newcommand{\fr}{\frac}

\newcommand{\ri}{\mathrm{Ric}}

\def\t{\tau}

\newcommand{\RR}{\mathbb{R}}

\newcommand{\EE}{\cal{E}}
\newcommand{\Ss}{\calS}
\newcommand{\Es}{\EE/\Ss}

\newcommand{\VV}{\cal{V}}
\newcommand{\HH}{\cal{H}}
\newcommand{\lan}{\langle}
\newcommand{\ran}{\rangle}
\newcommand{\aA}{\mathbf{a}}
\newcommand{\fF}{\mathbf{b}}

\newcommand{\sS}{\mathbf{s}}
\newcommand{\tT}{\mathbf{t}}
\newcommand{\pP}{\mathbf{p}}
\newcommand{\qQ}{\mathbf{q}}
\newcommand{\rR}{\mathbf{r}}

\newcommand{\indl}{\operatorname{ind_\ell}}
\newcommand{\ind}{\operatorname{ind_{\,e^{-X}}}}
\newcommand{\p}{\partial}
\newcommand{\pM}{{\partial{}M}}

\newcommand{\calL}   {{\mathcal L}}

\newcommand{\ppM}{{\partial{}M}}
\renewcommand{\b}{\bullet}
\newcommand{\n}{\nabla}
\newcommand{\tiln}{\widetilde{\nabla}}
\newcommand{\CC}{\mathbb{C}}
\newcommand{\ZZ}{\mathbb{Z}}
\newcommand{\gM}{g^M}
\newcommand{\gpM}{g^{\pM}}

\newcommand{\End}{\operatorname{End}}
\newcommand{\E}{{\mathcal{E}}}

\newcommand{\W}{{\mathcal{W}}}
\newcommand{\calS}{{\mathscr{S}}}
\newcommand{\calR}{{\mathscr{R}}}
\newcommand{\calV}{{\mathcal{V}}}
\newcommand{\coker}{\operatorname{coker}}
\newcommand{\Tr}{\operatorname{Tr}}
\newcommand{\oc}{\bar{c}}

\newcommand{\grg}{\mathfrak{g}}
\newcommand{\Str}{\operatorname{Str}}
\newcommand{\ch}{\operatorname{ch}}
\newcommand{\Tch}{\operatorname{\emph{T}\,ch}}
\newcommand{\sign}{\operatorname{sign}}
\newcommand{\even}{\mathrm{even}}
\newcommand{\odd}{\mathrm{odd}}
\newcommand{\tilF}{\widetilde{F}}
\newcommand{\tilD}{\widetilde{D}}

\begin{document}
\title[Equivariant APS index]{Equivariant APS index for Dirac operators of non-product type near the boundary}
\author[Maxim Braverman and Gideon Maschler]{Maxim Braverman${}^\dag$ and Gideon Maschler}

\thanks{${}^\dag$Partially supported by the Simons Foundation collaboration grant  \#G00005104.}





\subjclass[2010]{Primary 19K56; Secondary 53C56, 58J20}

\numberwithin{equation}{section}
\thispagestyle{empty}
\setcounter{equation}{0}
\address{Northeastern University\\
Department of Mathematics\\
360 Huntington Avenue\\
Boston, MA 02115, USA}
\email{maximbraverman@neu.edu}
\address{Department of Mathematics and Computer Science\\ Clark University\\
Worcester, Massachusetts 01610, U.S.A.}
\email{gmaschler@clarku.edu}

\begin{abstract}
We consider a generalized Atiyah-Patodi-Singer boundary problem for a $G$-invariant Dirac-type operator, which is not of product type near
the boundary.  We establish a delocalized version (a so-called Kirillov formula) of the equivariant index theorem for this
operator. We obtain more explicit formulas for different geometric Dirac-type operators. In particular, we get a formula for the equivariant signature  of a local system over a manifold with boundary.  In case of a trivial local system, our formula
can be viewed as a new way to compute the infinitesimal equivariant eta-invariant of S.~Goette. We explicitly compute all
the terms in this formula, which involve the equivariant Hirzebruch $L$-form and its transgression, for four-dimensional K\"ahler manifolds admitting a special K\"ahler-Ricci potential (SKR metrics), a class including many K\"ahler conformally Einstein manifolds, in the case where the boundary is given as the zero level set of this Killing potential. In the case
of SKR metrics which are local K\"ahler products, these terms are zero, and we obtain a vanishing result for the infinitesimal equivariant eta invariant.
\end{abstract}

\maketitle

\section{Introduction}
\setcounter{equation}{0}

In a seminal series of papers \cite{aps,aps2,aps3} Atiyah, Patodi and Singer studied the index of a boundary value problem for a first order elliptic differential operator $D$ on a compact manifold with boundary. They assumed that all the structures are product near the boundary and introduced the famous APS boundary conditions -- non-local boundary conditions defined using the spectrum of the restriction of $D$ to the boundary. They proved that the boundary value problem thus obtained is Fredholm and showed that its index is equal to the sum of contributions from the bulk and from the boundary. The bulk contribution is the integral of the usual Atiyah-Singer form, while the boundary contribution is the $\eta$-invariant of the restriction of $D$ to the boundary. One of the main applications of this index formula in \cite{aps2} is to the computation of the signature of a local system on a manifold with boundary.

Grubb \cite{grb} showed how to extend the Atiyah-Patodi-Singer construction  to the case where $D$ is not a product near the boundary. In this case the contribution of the boundary to the index consists of two terms -- the $\eta$-invariant and an integral of a certain differential form defined by $D$ over the boundary. Gilkey \cite{gil,glk1,glk1} used invariance theory to compute this integral for different geometric Dirac operators, including the twisted spin-Dirac and the twisted signature operator. In all cases the integrand turned out to be a certain transgression form over the boundary.

Hitchin \cite{hit} showed that in the case of the signature operator on 4-dimensional conformal compactifications of asymptotically hyperbolic Einstein manifolds, the integral of the transgression form over the boundary vanishes. Thus even in the non-product case, the index of the APS-boundary value problem is equal to the sum of the bulk contribution and the $\eta$-invariant. Moroianu \cite{mor} recently obtained a similar result for the signature operator on locally conformally flat manifolds in any dimension. Maschler \cite{mas} used the result of Hitchin to give, in dimenson four, an explicit
computation of the $\eta$-invariant for the conformally Einstein SKR metrics introduced in \cite{dm1,dm2}. One of the goals of the present paper is to give an equivariant analogue of the computations in \cite{mas}.

Donnelly \cite{don} considered an action of a compact Lie group $G$ on a compact manifold with boundary. Assuming that all the structures are product near the boundary, he defined a $G$-equivariant  analogue of the APS-index and computed it as a sum of fixed points contributions from the bulk and the equivariant $\eta$-invariant of the boundary. Goette \cite{goe} obtained a delocalized version (a so-called Kirillov formula \cite[Ch.~8]{bgv}) of the equivariant index theorem for $G$-equivariant Dirac operators which are product near the boundary. An important ingredient of Goette's formula is his infinitesimal version
of the equivariant $\eta$-invariant $\eta_X$, which is a function on the Lie algebra $\grg$ of $G$.

In this paper our main results are as follows. First, we generalize the results of \cite{goe} to the case of $G$-equivariant Dirac operators which are not product near the boundary. We then write the resulting index formula in more specialized forms, for the twisted spin-Dirac operator and the twisted signature operator, in accordance with different choices of specific boundary conditions. For the (non-twisted) signature operator, we compute the equivariant Hirzebruch $L$-form and the relevant component of its transgression form, both appearing in the index formula, for the case of the above mentioned SKR metrics on $4$-manifolds with a special type of boundary and equipped with a natural $U(1)$-action, even when the metrics are not conformal compactifications of Einstein metrics. Altogether these latter computations furnish, for these metrics, a formula for Goette's equivariant infinitesimal eta invariant $\eta_X$. For SKR metrics which are also reducible, that is, have the form of a (certain) local K\"ahler product, we obtain a vanishing result for the equivariant $L$-form and the degree three component of its transgression form, and consequently for $\eta_X$.

In more detail, our approach to the delocalized equivariant index theorem in the non-product case is similar to that of Gilkey \cite{glk1}, but we do not use invariance theory. Instead we give an explicit construction of the deformation of such an operator $D$ to a Dirac operator $D_0$ which is a product near the boundary. More precisely, suppose  $(\E,c,\n^\E)$ is a Dirac bundle over a manifold with boundary $M$. Here $c:T^*M\to \End(\E)$ is a Clifford action, and $\n^\E$ is a Clifford connection. We consider a deformation $(\E,c_t,\n^\E_t)$ ($0\le t\le 1$) of the Dirac structure on $\E$, which for $t=1$ coincides with  $(\E,c,\n^\E)$  and for $t=0$ is a product near the boundary. The existence of such a deformation is well known to experts, but we were not able to find the construction in the literature. We present a construction of such a deformation in the appendix. Let $D_t$ be the Dirac operator associated to the Dirac structure $(\E,c_t,\n^\E_t)$, so that $D_1=D$ and $D_0$ is a product near the boundary. Let $A_\pM$ denote the restriction of $D$ to the boundary and let $(D_t,A_\pM)$ be the operator $D_t$ with APS boundary conditions defined by $A_\pM$. Then, cf. Theorem~\ref{T:stability of the index}, the equivariant indices of $D$ and $D_0$ are equal. In Section~\ref{S:ind-form} we carefully study Goette's index formula for the operator $D_0$ and express the components of this formula in terms of the original Dirac bundle $(\E,c,\n^\E)$. This  leads to the delocalized equivariant index formula for a Dirac operator which is not a product near the boundary, cf. Theorem~\ref{T:genral APS}.

We note that the index in question depends on the boundary condition $A_\pM$. In Section~\ref{S:special bc} we make a convenient choice of the boundary condition closely associated with that of the twisted spin complex and get a more explicit version of the APS-type equivariant index formula in this case.

In Section~\ref{S:signature} we compute the equivariant index of an equivariant signature operator twisted by a vector bundle $\calV$, cf. Theorem~\ref{T:sign=}. For the special case where the bundle $\calV$ is flat, the equivariant signature of the local system defined by $\calV$ is closely related to the equivariant index of the twisted signature operator, cf. formula (2.4) of \cite{don} and formula \eqref{E:Gsgn} of the present paper. Using this fact we compute the equivariant signature of $\calV$ in Corollary~\ref{C:sign=}.

In the last part of Section~\ref{S:signature} we consider the (untwisted) signature of the manifold $M$. If the group $G$ is connected, the action of $G$ on cohomology is trivial. Hence, the equivariant signature is equal to the non-equivariant signature $\sign(M)$. We then view the equivariant index formula as a way to compute the infinitesimal $\eta$-invariant of Goette. Up to the integer term $\sign(M)$, this invariant is now equal to the sum of two integrals -- the integral of the equivariant $L$ form $L_\grg(X)$ ($X\in \grg$) over the bulk and the integral of its equivariant transgression form $TL_\grg(X)$  over the boundary.

In Section~\ref{S:signature 4D} we write the transgression form $TL_\grg(X)$ as a power series in $X\in \grg$ and present an explicit computation of its component of geometric degree three, having in mind the application of this result to the
computation of $TL_\grg(X)$'s boundary integral on a $4$-manifold. The rest of the paper is devoted to explicit computations of $L_\grg(X)$ and $TL_\grg(X)$ for a particular class of metrics in dimension four.

In Section~\ref{S:examples} we consider SKR metrics on a $4$-manifold with boundary $M$, where the boundary is given as the
zero level set of a certain Killing potential. After proving a formula giving the full curvature matrix of such a metric,
we develop an explicit expression for its associated equivariant $L$-form (Proposition~\ref{P:LXM} and Corollary~\ref{L-expl}). This may be considered a generalization to the equivariant case of the result in \cite{mas},
although the method of proof, adopted from \cite{goe}, is quite different. We end this section with a proof of formula~\Ref{final-trns} in Proposition \ref{TLg}, giving the pull-back to the boundary of the degree three component of $TL_\grg(X)$ for such SKR metrics. The derivation relies in part on the insights of Moroianu \cite{mor} in the non-equivariant case. This formula is fairly complex, and it is not trivial to determine from it whether the equivariant transgression form vanishes for (general) conformally Einstein SKR metrics. We leave this question open. Instead, cf.
Subsection~\ref{SS:vanishing}, we show that this pulled-back component of $TL_\grg(X)$ vanishes for the case where the SKR
metric on $M$ is reducible (whether or not it is conformally Einstein). As mentioned above, we thus have a partial analogue of Hitchin's vanishing result \cite{hit} in the equivariant setting. However, for such reducible SKR metrics, the calculation of the degree four component of $L_\grg(X)$ also yields zero, so that the infinitesimal equivariant eta invariant is determined only by the signature, which also happens to vanish for manifolds admitting such a metric. We thus obtain a vanishing of the infinitesimal
equivariant eta invariant in this case in Proposition \ref{van-eta}. The beginning
of section \ref{S:examples} contains some remarks on the relation of this result to
Hitchin's bound on the non-equivariant eta invariant, for conformal structures on the $3$-sphere
obtained from conformal compactifications of complete self-dual Einstein metrics on the $4$-ball.

\section*{Acknowledgements}
The second author would like to thank the Mathematics Department of Northeastern
University for hosting him as a visitor during the time work on this paper began.
Both authors thank the referee for insightful critique and several suggestions
that improved the presentation.

\section{The equivariant APS index}\label{S:APS index}

In this section we recall the definition of the equivariant index of a boundary value problem for a first order operator with  generalized Atiyah-Patodi-Singer  boundary conditions. We don't assume that our operator is a product near the boundary.

Throughout the paper $M$ is a compact  even dimensional manifold  with boundary and $G$ is a compact group, which acts on $M$ by
\begin{equation}\label{E:action}
		(\ell,x)\ \mapsto \ \ell\cdot x\in M, \qquad \ell\in G,\ x\in M.
\end{equation}
We endow $M$ with a $G$-invariant Riemannian metric $g^M$. We do not assume that this metric is a product near the boundary.

\subsection{The operator}\label{SS:operator}
Suppose $\E=\E^+\oplus\E^-$ is a complex $G$-equivariant  $\ZZ_2$-graded vector bundle over $M$ and let
\[
	D^+:\,C^\infty(M,\E^+) \ \longrightarrow \ C^\infty(M,\E^-)
\]
be a $G$-invariant first order linear elliptic  differential operator.

Let $h^\E$ be a $G$-invariant Hermitian metric on $\E$ such that the subbundles $\E^\pm$ are orthogonal to each other. We denote by $D^-:C^\infty(M,\E^-) \to C^\infty(M,\E^+)$ the formal adjoint of $D^+$ with respect to the $L^2$ scalar product defined by $h^\E$ and the Riemannian metric $g^M$. Consider the operator
\[
	D\ =\ D^+\oplus{}D^-:\,C^\infty(M,\E) \ \longrightarrow \ C^\infty(M,\E).
\]
This is a formally self-adjoint elliptic $G$-invariant operator on $M$.

\subsection{An equivariant collar}\label{SS:eqcollar}
To impose the boundary conditions on $D$ we first introduce convenient $G$-invariant coordinates near the boundary.

By \cite[Theorem~3.5]{kan} there exists a $G$-invariant open neighborhood $U$ of $\pM$ and a $G$-equivariant diffeomorphism
\[
	\psi:\,\pM \times (-\infty,0]\to U
\]
such that $\psi(y,0)=y$ for every $y\in\pM$ and the action of $G$ is given by
\begin{equation}\label{E:action collar}
		\ell\cdot \psi(y,u)\ = \ \psi(\ell\cdot y,u),
		\qquad \ell\in G,\ y\in\pM,\ u\in (-\infty,0].
\end{equation}

For simplicity of notation we omit $\psi$ and simply identify $U$ with the product $\pM \times (-\infty,0]\subset M$. Thus we write
\[
	M\ = \ \big(M\backslash U\big)\sqcup \big(\,\pM\times (-\infty,0]\,\big).
\]
We refer to $U\simeq \pM\times(-\infty,0]$ as a {\em $G$-equivariant collar} of $\pM$.

\subsection{A product structure on $\E$}\label{SS:product on E}
Fix a $G$-invariant grading-preserving connection $\n^\E$ on $\E$.
We use $\n^\E$ to identify the fibers of $\E$ along the rays
$(y,u)\in U=\pM\times(-\infty,0]\subset M$. Thus we obtain an isomorphism
\[
	\phi:\,\E\big|_U \ \tilde{\to}\ E\times(-\infty,0],
\]
where $E=E^+\oplus E^-$ is a $G$-invariant bundle over $\pM$. By a slight abuse of notation we will omit $\phi$ and simply write $\E\big|_U= E\times(-\infty,0]$. Thus we view a section $f$ of $\E^\pm$ over $U$ as a family of sections
\[
	f(\cdot,u) \ \in \ C^\infty(\pM,E), \qquad u\in (\infty,0].
\]
Thus for a section $f$, and $(y,u)\in \pM\times(-\infty,0]$ we identify $f(y,u)\in \E_{(y,u)}$ as an element of the fiber $E_y$ of $E$ over $y$. With this identification we have
\[
	\n^\E_{\frac{\p}{\p u}}f\ = \ \frac{\p}{\p u}f(y,u),
\]
where in the left hand side $\frac{\p}{\p u}$ denotes the vector field along the ray $\pM\times(-\infty,0]$.

In particular, if $y\in \pM$ and  $e\in E_y$, we consider $(e,u)$ ($u\le0$) as a section of $\E$ over the ray $\{y\}\times(-\infty,0]$. This is a parallel section, i.e.
\begin{equation}\label{E:flat(e,u)}
	\n^\E_{\frac{\p}{\p u}}\, (e,u)\ = \ 0.
\end{equation}

Since the connection $\n^\E$ is $G$-invariant the action of $G$ on $E\times(-\infty,0]$ is given by
\begin{equation}\label{E:action on E}
	\ell\cdot (e,u)\ = \ (\ell\cdot e, u),
\end{equation}
where $\ell\in G, \ e\in E,\ u\in (-\infty,0]$.

For $u\in (-\infty,0]$ we denote by $\n^\E_u$ the restriction of $\n^\E$ to $\pM\times\{u\}\subset M$.

\subsection{The operator near the boundary}\label{SS:boundary form}
The restriction of $D$ to $U=\pM\times(-\infty,0]$ can be written as
\begin{equation}\label{E:boundary form}
	D\big|_U\ = \ \gamma(u)\,\Big(\,
	  \frac{\p}{\p u}\ + \ B(u)\,\Big),
\end{equation}
where  for each $u\in (-\infty,0]$, $B(u):C^\infty(\pM,E^\pm\times\{u\})\to C^\infty(\pM,E^\pm\times\{u\})$ is a first order differential operator and $\gamma(u):E^+\to E^-$ is a bundle map.

In what follows we will consider the situation when $D$ is a Dirac operator associated to a Dirac bundle defined by the geometry of the manifold. Then it often happens that the connection $\n^\E_u$ is different from the natural geometric connection on $E\times\{u\}$. If we denote by $A(u)$ the Dirac operator on $E\times\{u\}$ associated to such a natural connection, then
\begin{equation}\label{E:Phi(u)}
	\Phi(u)\ := \ B(u)\ - \ A(u)
\end{equation}
is a bundle map, cf. \cite[p.~117]{bgv}. More generally, we fix a self-adjoint continuous family of $G$-invariant bundle maps $\Phi(u):E^\pm\to E^\pm$ and define $A(u)$ by \eqref{E:Phi(u)}. Then $A(u)$ is a $G$-equivariant operator and
\begin{equation}\label{E:boundary form 2}
	D\big|_U\ = \ \gamma(u)\,\Big(\,
	  \frac{\p}{\p u}\ + \ A(u)\ + \ \Phi(u)\,\Big).
\end{equation}

\subsection{Generalized APS boundary conditions}\label{SS:APS}
We are using the restriction of the operator $A(0)$ to $E^+\times \{0\}$, which we denote $A_\pM$,
to define the boundary conditions for $D$. Specifically, we denote by
\begin{equation}\label{E:PA}
	P_A:\, C^\infty(\p M,E^+)\ \to \ C^\infty(\p M,E^+),
\end{equation}
the spectral projection of the operator $A_\pM$, whose image is the
span of the eigensections of $A_\pM$ with non-negative eigenvalues. Let
\[
V^+_A= \left\{s\in C^\infty(M,\E^+):P_A(s\big|_{\pM})=0\right\}. 
\]
Then the operator
\begin{equation}\label{E:D on V+}
	 	D^+:\, V^+_A	\ \to \ C^\infty(M,\E^-),
\end{equation}
is Fredholm, cf. \cite[Section~2]{grb}. Hence, its kernel $\ker(D^+,A_\ppM)$ and cokernel $\coker(D^+,A_\ppM)$ are finite dimensional representations of $G$.

\begin{Definition}\label{D:equivariant index}
The {\em equivariant index}  $\indl(D,A_\pM)$ of $D$, with respect to generalized APS boundary conditions defined by $A_\pM$,
is a function of $\ell\in G$ defined by
\begin{equation}\label{E:equivariant index}
	\indl\big(D,A_\pM\big)\ := \ \Tr\ \ell\big|_{\ker(D^+\!,\,A_{\p\!M})}\ - \
	\Tr \ell\big|_{\coker(D^+\!,\,A_{\p\!M})},
	\qquad \ell\in G.
\end{equation}
\end{Definition}

\begin{remark}\label{R:need of P}
We refer to the boundary conditions defined by $A_\pM$ as {\em generalized APS boundary conditions}, whereas the classical APS boundary conditions are defined using the spectral projection of  $B(0)$. We use generalized APS boundary conditions, since in what follows, $A$ will be the geometric
Dirac operator on $\pM$, which is often different from the operator $B$ defined by the restriction of $D$ to the boundary. We notice, that, in general, the index of the obtained boundary value problem may vary with a change of $A$, cf. \cite[p. 314]{glk1}.
\end{remark}

\section{The equivariant characteristic classes}\label{S:eq char class}

The equivariant APS-index theorem computes  the index \eqref{E:equivariant index} in terms of the equivariant characteristic classes on $M$ and the equivariant eta-invariant. In this section we recall the construction of the equivariant characteristic classes. We mostly follow the exposition of \cite[Ch.~7]{bgv}.

\subsection{The equivariant De Rham complex}\label{SS:eq DeRham}
Let $\grg$ denote the Lie algebra of $G$. For $X\in \grg$ we denote by $X_M$ the vector field on $M$ defined by
\begin{equation}\label{E:XM}
	X_M(x) \ :=\ \frac{d}{dt}\Big|_{t=0}\exp\big(-tX)\cdot x, \qquad x\in M.
\end{equation}

We denote by $\Omega^\b(M)[\grg]$ the set of polynomials in $\grg$ with coefficients in $\Omega^\b(M)$. The group $G$ acts naturally on $\Omega^\b(M)[\grg]$, cf. \cite[\S7.1]{bgv}, and we denote by $\Omega_G(M)$ the set of $G$-invariant elements of $\Omega^\b(M)[\grg]$.

The elements $\omega\in \Omega_G(M)$ are called the {\em equivariant differential forms} on $M$.

Similarly, we define the space $\Omega_G(M,\E)$ of equivariant differential forms with values in a $G$-equivariant graded vector bundle $\E=\E^+\oplus\E^-$ over $M$.

The {\em equivariant de Rham differential} is the operator $d_\grg:\Omega_G(M)\to \Omega_G(M)$ defined by
\begin{equation}\label{E:eq differential}
	d_\grg(\omega)(X)\ := \ d\omega(X)\ - \ \iota_{X_M}\omega(X),
	\qquad X\in \grg, \ \ \omega\in \Omega_G(M).
\end{equation}
One checks that $d_\grg^2=0$.

The {\em equivariant cohomology} of $M$ is defined by
\[
	H_G(M)\ := \  \ker d_\grg/\operatorname{im} d_\grg.
\]

\subsection{The equivariant connection}\label{SS:eq connection}
Let $\n^\E$ be a $G$-invariant connection on $\E$. The {\em $G$-equivariant connection}  $\n^\E_\grg:\Omega_G(M,\E)\to \Omega_G(M,\E)$ is defined by the formula
\begin{equation}\label{E:eq connection}
	\n^\E_\grg \omega(X)\ := \ \n^\E\omega(X)\ - \iota_{X_M}\omega(X),
	\qquad X\in \grg, \ \ \omega\in \Omega_G(M,\E),
\end{equation}

\subsection{The moment}\label{SS:moment}
For $X\in \grg$ let  $\calL^\E_X$ denote the infinitesimal action of $X$ on $\Omega^\b(M,\E)$. This action extends to $\Omega_G(M,\E)$%
\footnote{Notice that even though $\Omega_G(M,\E)$ is the set of $G$-invariant elements of $\Omega^\b(M\,E)[\grg]\simeq \Omega^\b(M,\E)\otimes\CC[\grg]$ the restriction of $\calL^\E_X$ to $\Omega_G(M,\E)$ is not trivial. This is because $\calL_X^\E$ acts only on the first factor, while the elements of $\Omega_G(M,\E)$ are $G$-invariant with respect to the diagonal action of $G$.}.

The {\em moment of $X$ with respect to the connection $\n^\E$} is the element
\begin{equation}\label{E:moment}
	\mu^\E(X)\ := \ \calL^\E_X\ - \ \n^\E_{X_M} \ \in \
		\Omega^\b\big(M,\End(\E)\big).
\end{equation}

In the special case when $\E= TM$ is the tangent bundle endowed with the Levi-Civita connection $\n^{LC}$, the moment can be computed explicitly, cf. Example~7.8 in Section~7.1 of \cite{bgv}. The answer is
\begin{equation}\label{E:Riemannian moment}
	\mu^{TM}(X)\,Y\ = \ -\n^{LC}_YX_M, \qquad X\in \grg,\ Y\in C^\infty(M,TM),
\end{equation}
and the formula holds for the moment of any other torsion-free connection on $TM$.

\subsection{The equivariant curvature}\label{SS:eq curvature}
Let $F^\E:=(\n^\E)^2\in \Omega^2(M,\End(\E))$ be the curvature of the connection $\n^\E$.  The {\em equivariant curvature} $F^\E_\grg\in \Omega_G(M)$ of the equivariant connection $\n^\E_\grg$ is defined by
\begin{equation}\label{E:eq curvature}
	F^\E_\grg(X)\ := \ F^\E\ + \ \mu^\E(X), \qquad X\in \grg.
\end{equation}

If $\E=TM$ we denote the equivariant curvature by $R_\grg(X)$.

\subsection{The equivariant characteristic classes}\label{SS:eq char classes}
Let $f(z)$ be a polynomial in one complex variable $z$. Then $f(F_\grg^\E)\in \Omega_G(M,\End(E))$. Let $\Str:\Omega_G(M,\End(\E))\to \Omega_G(M)$ be the supertrace defined using the grading $\E=\E^+\oplus\E^-$. Then, cf. \cite[Theorem~7.7]{bgv} the form
\begin{equation}\label{E:betaX}
	\beta_\grg(\n^\E)(X)\ := \
	\Str\big(f(F^\E_\grg(X))\big)
\end{equation}
is equivariantly closed,
\[
	d_\grg\Str\big(f(F^\E_\grg)\big)\ = \ 0.
\]
We refer to $\beta_\grg$ as the {\em equivariant characteristic form} associated to $f$. Its class $\Big[\Str\big(f(F^\E_\grg)\big)\Big]\in H_G(M)$ is called the {\em equivariant characteristic class} corresponding to $f(z)$.

More generally if $f(z)$ is a germ of an analytic function near zero, the form
\[
	f(F_\grg^\E(X)) \ := \ f\big(F+\mu^\E(X)\big)
	\ \in \ \Omega^\b(M,\End(\E)),
\]
is well defined for small enough $X\in \grg$. Thus we can define an equivariant characteristic form $\beta_\grg(\n^\E)(X)$ by \eqref{E:betaX}.

There is also another equivariant characteristic form associated with $f$ defined by
\[
	\tilde{\beta}_\grg(\n^\E)(X) \ := \ \exp\Big(\Str\big(f(F^\E_\grg(X))\big)\Big).
\]
In particular, we consider the {\em equivariant Chern form}
\begin{equation}\label{E:chern}
	\ch_\grg(\n^\E)(X)\ := \ \Str \exp\big(-F_\grg^\E(X)\big),
\end{equation}
the equivariant $\hat{A}$-genus
\begin{equation}\label{E:A genus}
	\hat{A}_\grg(g^M)(X)\ := \ {\det}^{1/2}\left(
		\frac{R_\grg(X)/2}{\sinh\big(R_\grg(X)\big)/2}\right)
		\ = \ \exp\left(\, \frac12\Tr\log\left(
		\frac{R_\grg(X)/2}{\sinh\big(R_\grg(X)\big)/2}\right)\,\right),
\end{equation}
and the {\em equivariant Hirzebruch $L$-form}
\begin{equation}\label{E:L form}
	L_\grg(g^M)(X)\ := \ {\det}^{1/2}\left(
		\frac{R_\grg(X)/2}{\tanh\big(R_\grg(X)\big)/2}\right).
\end{equation}

\section{The equivariant APS-theorem in the product case}\label{S:product case}

In \cite{don}, Donnelly generalizes the Atiyah-Segal-Singer fixed point formula for manifolds with boundary. In particular, he defined a $G$-equivariant version of the eta-invariant. In \cite{goe}, S.~Goette introduced an infinitesimal version of the equivariant eta invariant and used it to obtain an equivariant version of the Atiyah-Patodi-Singer index for the case when all the structures are product near the boundary. The theorem of Goette extends the ``delocalized" equivariant index formula of Berline and Vergne \cite{BerlineVergne85} (see also \cite[Ch.~8]{bgv}) to manifolds with boundary.  In this section we briefly review the results of Goette.

\subsection{An equivariant Dirac bundle}\label{SS:Dirac bundle}
Let $\gM$ be a $G$-invariant metric on $M$. We do not assume that this metric has a product structure near the boundary.

Recall \cite[Definition~II.5.2]{lami} that a (graded) Clifford module over $M$ is a Hermitian bundle $\E=\E^+\oplus \E^-$ together with a  map
\begin{equation}\label{E:Clifford action}
		c:\, T^*M\ \to \ \End(\E).
\end{equation}
such that for all $\xi\in T^*M$, we have $c(\xi)^2= -\gM(\xi,\xi)$, and $c(\xi)$ is a skew adjoint bundle map such that $c(\xi)(\E^\pm)= \E^\mp$. We refer to \eqref{E:Clifford action}
as a {\em Clifford action} of $T^*M$ on $\E$.

We say that $\E=\E^+\oplus \E^-$ is a {\em $G$-equivariant Clifford module} over $M$ if $\E$ is a  $G$-equivariant bundle over $M$, the $G$ action preserves the Hermitian metric and the grading on $\E$, and the action \eqref{E:Clifford action} is $G$-equivariant.

A {\em Clifford connection} on $\E$  is a $G$-invariant Hermitian connection  $\n^\E=\n^{\E^+}\oplus\n^{\E^-}$ which is compatible with the Clifford action \eqref{E:Clifford action}  in the sense that
\begin{equation}\label{E:Clifford connection}
 	\n^\E_v \big(\,c(\xi)\cdot s\,\big)\ = \
 	\big(\,c(\n^{LC}_v\xi\,\big)\cdot s \ + \
 	c(\xi)\cdot\n^\E_vs,
\end{equation}
where $v\in TM,\ \xi\in C^\infty(M,T^*M),\ s\in C^\infty(M,\E)$, and $\n^{LC}$ stands for the Levi-Civita connection on $T^*M$.

\begin{Definition}\label{D:Dirac bundle}
A {\em $G$-equivariant Dirac bundle} over $M$ is a $G$-equivariant graded Clifford module over $M$ endowed with a $G$-invariant Clifford connection.
\end{Definition}

\subsection{A Dirac operator}\label{SS:Dirac operator}
The Clifford action \eqref{E:Clifford action} defines a map
\[
	c:\,C^\infty(M,T^*M\otimes\E)\ \to\ C^\infty(M,\E),
	\qquad c(\xi\otimes e)\ := \ c(\xi)e.
\]

\begin{Definition}\label{D:Dirac operator}
A {\em Dirac operator} associated to the Dirac bundle $(\E,c,\n^\E)$ is defined as the composition
\begin{equation}\label{E:Dirac operator}
  \begin{CD}
        C^\infty(M,\E) @>\n^\E>> C^\infty(M,T^*M\otimes \E) @>c>>
        C^\infty(M,\E).
  \end{CD}
\end{equation}
\end{Definition}

In local coordinates, this operator may be written as
$D=\sum\,c(dx^i)\,\n^\E_{\frac{\p}{\p x^i}}$. Note that $D$ sends even sections
to odd sections and vice versa, giving operators $D^\pm:\, C^\infty(M,\E^\pm)\to
C^\infty(M,\E^\mp)$. If the Dirac bundle $(\E,c,\n^\E)$ is $G$-equivariant, then so is $D$.

\subsection{A product Dirac bundle}\label{SS:produc Dirac}
Let $g^\pM:= g^M\big|_\pM$ denote the restriction of the Riemannian metric $g^M$ to the boundary. Let $g^{\pM\times(-\infty,0]}$ denote the metric on the cylinder $U= \pM\times(-\infty,0]$ which is the product of $\gpM$ on $\pM$ and the standard metric on $(-\infty,0]$.

For $y\in \pM$ and $u\le0$ we identify $T^*_{(y,u)}U$ with $T^*_u\pM\times\RR$.

\begin{Definition}\label{D:product gM}
Let $a<0$. We say that the metric $g^M$ is a {\em product} on the cylinder $\pM\times[a,0]\subset \pM\times(-\infty,0]$  if the  restriction of $g^M$ to  $\pM\times[a,0]$ coincides with the restriction of the product metric $g^{\pM\times(-\infty,0]}$.

If there exist $a<0$ such that $g^M$ is a product on $\pM\times[a,0]$ we say that $g^M$ is a  {\em product near the boundary}.
\end{Definition}

\begin{Definition}\label{D:product Dirac}
Let $a<0$ and suppose that the metric $g^M$ is a product on $\pM\times[a,0]$. A $G$-equivariant Dirac bundle $\E=\E^+\oplus \E^-$ over $M$ is called a {\em product} on $\pM\times[a,0]$ if the following conditions are satisfied
\begin{enumerate}
\item $\E= E\times [a,0]$, where $E=E^+\oplus E^-$ is a $G$-equivariant bundle over $\pM$;
 \item there exists a map $\oc:T^*\pM\times\RR\to \End(E)$ such that the Clifford action on $\E$ has the form
 \[
 	c(\xi)\cdot (e,u)\ = \ \big(\, \oc(\xi)\cdot e, u\,\big),
 \]
 where $\xi\in T^*U\simeq T^*\pM\times\RR$, $e\in E$, and $u\in (-\infty,0]$;
 \item the connection $\n^\E$ is a product of a connection $\n^E$ on $E$ and the trivial connection on $\CC\times[a,0]\to [a,0]$;
 \item
 the action of $G$ on $\E= E\times [a,0]$ is trivial on the  second factor
 \[
 	\ell\cdot  (e,u)\ = \ \big(\, \ell\cdot e, u\,\big).
 \]
\end{enumerate}

If there exist $a<0$ such that $\E$  is a product on $\pM\times[a,0]$ we say that $\E$ is a  {\em product near the boundary}.
\end{Definition}

\subsection{The equivariant APS-index in  the product case}\label{SS:eq APS index}
Suppose now that the equivariant  Dirac bundle $\E$ is a product near the boundary. In particular, this implies that the formula \eqref{E:boundary form} takes a product form
\begin{equation}\label{E:D product}
	D\big|_U\ = \ c(du)\,\Big(\, \frac{\p}{\p u}\ + \ B\,\Big),
\end{equation}
where $B:C^\infty(\pM,E^\pm)\to C^\infty(\pM,E^\pm)$ is independent of $u$. The equivariant APS-index theorem of Goette \cite[Theorem~1.9]{goe} computes the index $\ind(D,B)$. The important components  of the answer are the equivariant relative Chern form and the infinitesimal eta-invariant, which we define in the next two subsections.

\subsection{The twisting equivariant curvature}\label{SS:twisting curvature}
Let $\{e_1,\ldots,e_{n}\}$ be an orthonormal frame of $TM$. Recall that we denote by $R\in\Omega^2(M,\End(TM))$ the curvature of $g^M$. The {\em twisting curvature} of $\n^\E$ is defined by the formula
\begin{equation}\label{E:twisting curvature}
 	F^{\cal{E}/\cal{S}}\ = \ F^\EE-g^M\big(Re_k,e_l\big)c(e^k)c(e^l)/4.
\end{equation}

\begin{remark}\label{R:twisting curvature}
The notation is motivated by the following computation. Suppose  $M$ is a spin-manifold and let $\calS$ be a $G$-equivariant spinor bundle over $M$, then there is a $G$-equivariant vector bundle $\W$ over $M$ and a $G$-invariant connection $\n^\W$ on $\W$ such that
\begin{equation}\label{E:E=SotimesW}
		\E\ =\ \calS\otimes \W, \quad
	\n^\E\ =\ \n^\calS\otimes1\ + \  1\otimes\n^\W,
\end{equation}
where $\n^\calS$ denote the Levi-Civita connection on $\calS$. Let $F^\W=(\n^\W)^2$ denote the curvature of $\n^\W$. It is shown on Page~121 of \cite{bgv} that $1\otimes{}F^\W= F^{\E/\calS}$.
\end{remark}

Similarly to \eqref{E:twisting curvature} we define the {\em equivariant twisting curvature} of $\E$ by
\begin{equation}\label{E:eq twisting curvature}
 	F^{\cal{E}/\calS}_\grg(X)\ =\
 	F^\EE_\grg(X)\ - \ g^M\big(R_\grg(X)e_k,e_l\big)c(e^k)c(e^l)/4.
\end{equation}

We define the {\em equivariant relative Chern character} as the cohomology class in $H_G(M)$ of the {\em equivariant relative Chern form}
\begin{equation}\label{E:twisting Chern}
	\ch_\grg(\E/\calS)(X)\ := \
	 \Str_{\E/\calS} \exp\big(-F_\grg^{\E/\calS}(X)\big),
\end{equation}
where $\Str_{\E/\calS}$ is the {\em relative supertrace} defined on page~146 of \cite{bgv}. We note that in situation of Remark~\ref{R:twisting curvature}, $\ch_\grg(\E/\calS)(X)$ is equal to the equivariant Chern form of the bundle $\W$,
\begin{equation}\label{E:chE/S=chW}
 	\ch_\grg(\E/\calS) \ = \ \ch_\grg(\W).
 \end{equation}

\subsection{Goette's equivariant infinitesimal eta invariant}\label{SS:inf eta}
The {\em equivariant infinitesimal eta invariant} of Goette is a  formal power series in $X\in \grg$ given by
\begin{equation}\label{E:inf eta}
	\eta_X(B) \ = \
	   \int_0^\infty \frac 1{\sqrt{\pi t}}
	    \mathrm{Tr}\big(\,B_{X/t}e^{-tH_{B,X/t}}\,\big)\,du,
\end{equation}
where $B_X=B-c(X)/4$ and
\[
	H_{B,\,X}\ := \ \big(\,B+c(X)/4\,\big)^2+\calL_X
\]
is the {\em equivariant Bismut Laplacian}, with $c(X)$ denoting Clifford multiplication by $X_M$,
which is well-defined after identifying the tangent and cotangent bundles using $g^M$.
One of the main theorems in \cite{goe} asserts that the difference between $\eta_X$ and
the Donnelly's equivariant eta invariant $\eta_{e^{-X}}(B)$ is given by an explicit
locally computable integral on $\pM$.

\subsection{The equivariant APS-index theorem in the product case}\label{SS:eq APS product}
The following equivariant version of the Atiyah-Patodi-Singer index theorem is due to S.~Goette \cite[Theorem~1.17]{goe}.

\begin{theorem}\label{T:eq APS product}
Suppose that the equivariant  Dirac bundle $\E=\E^+\oplus\E^-$ is a product near the boundary.  Let $X\in\grg$ be a sufficiently close to zero and assume that  the vector field $X_M$ does not vanish anywhere on $\pM$. Then
\begin{equation}\label{E:eq APS produc}
	\ind(D,B) \ = \
		(2\pi i)^{-n/2}\int_M
		\hat{A}_\grg(g^M)(X)\cdot \mathrm{ch}_\grg(\cal{E}/\calS)(X)
		\ - \
		\frac{\eta_X(B)\ + \ h_{e^{-X}}(B)}2,
\end{equation}
where $h_{e^{-X}}(B)=\mathrm{Tr}\big(e^{-X}\big|_{\ker B }\big)$ and $n=\dim M$.
\end{theorem}

\section{A deformation of a Dirac bundle}\label{S:deformation}

The main goal of this paper is to compute the equivariant index of a Dirac operator with generalized Atiyah-Patodi-Singer (APS) boundary conditions in the situation when the metric $g^M$ is not a product near the boundary. For the non-equivariant case this computation was done by Gilkey \cite{gil,gil2} by extending the metric past the boundary to
one that becomes product near a new boundary, and studying the behavior of the different components of the APS index formula under this extension.

In this section we introduce a deformation of an equivariant Dirac bundle to
one which is a product near the boundary. It follows from the stability of the index
that the indices of the Dirac operators associated to the original and the deformed Dirac bundles coincide. In Section~\ref{S:ind-form} we use this deformation to give an explicit formula for the equivariant index in the non-product situation, by  adopting Gilkey's  argument \cite{gil,gil2} to the equivariant case.

\subsection{Deformation of the metric}\label{SS:deformation of the metric}

We use the notation of Subsection~\ref{SS:eqcollar}. In particular, we identify a $G$-equivariant collar $U$ of the boundary of $M$ with the product $U= \pM\times(-\infty,0]$. In this subsection we construct a deformation of $g^M$ to a metric $g^M_0$ which is a product near the boundary.

Let $\gpM$ denote the restriction of $g^M$ to $\pM$ and let $g^{U}$ denote the metric on $U = \pM\times(-\infty,0]$ defined as the product of $\gpM$ and the standard metric on $(-\infty,0]$.

Let $s:\RR\to \RR$ be a smooth non-decreasing function such that
\[
	s(u)\ =\ \begin{cases}
					1, \quad&\text{for all} \ \ u\, \le \, -1;\\
					0, \quad&\text{for all} \ \ u\, \ge \, -2/3.
			\end{cases}
\]
For $0\le{}t\le1$ we set
\[
	s_t(u)\ = \ s(u-t).
\]

Consider a family of  Riemannian metrics $\gM_t$ on $M$,  whose restriction to $M\backslash{U}$ is equal to $\gM$ and whose restriction to the cylinder $U$ is given by
\[
	\gM_t\ := \
	s_t\,\gM \ + \ (1-s_t)\, g^{U}.
\]

\subsection{Properties of $\gM_t$}\label{SS:properties gMa}

We note the following properties of the family $\gM_t$:
\begin{enumerate}
\item
the function $t\mapsto \gM_t$ is continuous in the $C^1$-topology. In particular, if we denote by $\n^{LC}_t$ the Levi-Civita connection associated to $\gM_t$ then the function $t\mapsto \n^{LC}_t-\n^{LC}_0$ is continuous in the $C^0$-topology;
\item
for each $0\le{}t\le1$, the metric $\gM_t$ is invariant with respect to the $G$-action;
\item
for $t=1$,  $g^M_t=g_1^M= g^M$; for $t< 1$ the restriction of $g^M_t$ to
\[
	(M\backslash U)\cup\big(\pM\times(-\infty, -1+t)\big),
\]
coincides with $\gM$;
\item
for $t< 2/3$, the restriction of $\gM_t$ to $\pM\times[-\frac23+t,0]$ is the product of $\gpM$ and the standard metric on $[-\frac23+t,0]$. In particular, if we use the natural identification of the fibers $\pM\times\{u\}$ with $\pM$ then the restriction of $\gM_t$ to $\pM\times\{u\}$ is equal to $\gpM$ for all $u\in [-\frac23+t,0]$.
\item
the restriction of $\gM_t$ to the boundary $\pM$ does not depend on $t$ and is $g^\pM$.
\end{enumerate}

\begin{Definition}\label{D:admissible gMt}
A family of Riemannian metrics $g^M_t$ on $M$ is called an {\em admissible deformation of  $g^M$} if it satisfies the above properties (i)-(v).
\end{Definition}

\subsection{Deformation of the Dirac bundle}\label{SS:deformation Dirac budnle}
Next we  construct a family of  Dirac bundle structures on $\E$, parametrized by $ t\in [0,1]$, such that for each $t$ this structure is compatible with the metric $\gM_t$, for $t=1$ it coincides with the original structure, and for $t=0$ it is a product near the boundary.

\begin{Definition}\label{D:family Dirac bundles}
A triple $(\E,c_t,\n^\E_t)$ $(0\le t\le1)$ is called a {\em family of Dirac bundles} if
\begin{enumerate}
\item
for $0\le t\le1$,
\begin{equation}\label{E:c2=ga}
	\gM_t(\xi,\xi) \ := \ -c_t(\xi)^2
\end{equation}
is a Riemannian metric on $T^*M$. Then $c_t$ is a Clifford action compatible with $g^M_t$;
\item
$\n^\E_t$ is a Clifford connection on $\E$ compatible with $c_t$ in the sense that (cf. \eqref{E:Clifford connection})
\[
	 	\n^\E_{t,v} \big(\,c(\xi)\cdot s\,\big)\ = \
 	c_t(\n^{LC}_{t,v}\xi\,\big)\cdot s \ + \
 	c_t(\xi)\cdot\n^\E_{t,v}s,
\]
where $\n^{LC}_t$ is the Levi-Civita connection of the metric $g^M_t$.
\end{enumerate}
In this situation we say that the family $(\E,c_t,\n^\E_t)$ $(0\le t\le1)$ is compatible with the family of metrics $g^M_t$.
\end{Definition}

\begin{Definition}\label{D:boundary connection}
We say that a connection $\tiln^E=\tiln^{E^+}\oplus\tiln^{E^-}$ on  $E=\E\big|_\pM$ is {\em a boundary connection compatible with the restriction of the Clifford action $c$ to the boundary} if for all $v\in T\pM,\ \xi\in C^\infty(M,T^*M\big|_\pM),\ s\in C^\infty(\pM,E)$, we have
\[
	 	\tiln^E_v \big(\,c(\xi)\cdot s\,\big)\ = \
 	c(\n^{\pM,LC}_v\xi\,\big)\cdot s \ + \
 	c(\xi)\cdot\tiln^E_vs,
\]
where $\n^{\pM,LC}$ denote the Levi-Civita connection on $\pM$ defined by the metric $g^\pM= g^M\big|_\pM$ (note that in general $\n^{\pM,LC}$ is not equal to the restriction of $\n^{LC}$ to the boundary).
\end{Definition}

\begin{Proposition}\label{P:deformation Dirac}
Given a $G$-invariant connection $\tiln^E$ on $E$ compatible with the restriction of $c$ to $\pM$ and an admissible deformation $g^M_t$ of $g^M$, there exists a family of $G$-equivariant Dirac bundles $(\E,c_t,\n^\E_t)$ \/ $(0\le{}t\le1)$, where $c_t:T^*M\to \End(\E)$ is a Clifford action compatible with the metric $\gM_t$ and $\n^\E_t$ is a Clifford connection on $(\E,c_t)$, such that

\begin{enumerate}
\item
for  $t=1$ the Dirac bundle $(\E,c_t,\n^\E_t)$ coincides with the original Dirac bundle $(\E,c,\n^\E)$;
\item
$c_t(\xi)= c(\xi)$ for every $\xi\in T^*M\big|_\pM$. In other words, the restriction of the Clifford module structure to the boundary is independent of $t$;
\item
the Dirac bundle $(\E,c_0,\n^\E_0)$ is a product near the boundary and the restriction of $\n^\E_0$ to the boundary is equal to $\tiln^E$;
\item
the families $c_t$ and $\n^\E_t$ are continuous in the $C^0$-topology for $0\le{}t\le1$;
\item
the restrictions of $c_t$ and  $\n^\E_t$ to $\big(M\backslash U\big)\sqcup \big(\,\pM\times (-\infty,-3]\,\big)$ are independent of \/ $0\le{}t\le1$ and equal to $c$ and $\n^\E$ respectively.
\end{enumerate}
\end{Proposition}

The existence of such a deformation is well known and often used in the literature without a proof. However the construction of such a deformation is not  totally trivial, since the Clifford relation \eqref{E:c2=ga} is non-linear. In particular, one can not first construct a non-equivariant version of the family of Clifford actions and then average it over $G$ to obtain an equivariant version.  Since we did not find a good reference for a construction of a family of Dirac structures we present it in full detail in the Appendix.

\begin{Definition}\label{D:admissible deformation}
An {\em admissible deformation} of the Dirac bundle $(\E,c,\n^\E)$, is a family of Dirac bundles $(\E,c_t,\n^\E_t)$ $(0\le{}t\le1)$ which satisfies conditions (i)-(iv) (but does not necessarily satisfy condition (v)) of Proposition~\ref{P:deformation Dirac} with $\tiln^E_0:= \n^\E_0\big|_\pM$.
\end{Definition}

\begin{remark}\label{R:no average}
Suppose $M$ is a spin-manifold, $\E= \calS$ is a spinor bundle over $M$, and $\tiln^E$ is a Levi-Civita connection on $S:=\calS\big|_\pM$ defined by the metric $g^{\pM}:= g^M\big|_\pM$. We note that in general $\tiln^E$ is not equal to the restriction of the Levi-Civita connection on $\calS$ to $\pM$.
In this situation the family of natural Dirac bundle structures on $\calS$ induced by the metrics $\gM_t$ satisfies the conditions of the proposition. More generally, if $\E= \calS\otimes{}W$ is a twisted spinor bundle, there is  a family of Dirac bundle structures on $\E$ induced by the family of Dirac bundle structures on $\calS$. One of the ways to construct a family of Dirac bundle structures in the general
case is to use a covering of $M$ by open contractable neighborhoods $M=\bigcup U_j$ such that $\E\big|_{U_j}\simeq \calS_j\otimes{}W_j$, construct a family of Dirac bundle structures on each $\E\big|_{U_j}$ and then glue them together using a partition of unity. However, it is not easy to obtain a family of {\em equivariant} Dirac bundle structures in this way. Because of this, in Appendix~\ref{S:pr of deformation} we use a completely different construction.
\end{remark}
\subsection{Stability of the index}\label{SS:stability}
Let $(\E,c_t,\n^\E_t)$ be an admissible deformation of $(\E,c,\n^E)$ and let $D_t$ \/ ($t\le 1$) denote the Dirac operator associated to the equivariant Dirac bundle $(\E,c_t,\n^\E_t)$. We view $D_t$ as a bounded operator from the Sobolev space $\mathcal{H}^1(M,\E)$ to the Sobolev space $\mathcal{H}^0(M,\E)$. Then $D_t$ depends continuously on $t$. The operator $D_{1}$ coincides with the original Dirac operator $D$, while $D_0$ is a product near the boundary.

For $u\in (-\infty,0]$, $y\in \pM$, $\xi\in T^*_{(y,u)}M$, $0\le{}t\le1$, we denote by $\oc_{t,u}(\xi):E_y\to E_y$ the linear map such that
\begin{equation}\label{E:octu}
		c_t(\xi)\cdot(e,u)\ = \ \big(\,\oc_{t,u}(\xi)\cdot e,u\,\big).
\end{equation}
Then, as in Subsection ~\ref{SS:boundary form}, we obtain
\begin{equation}\label{E:Da|U}
	D_t\big|_{U}\ = \ c_t(du)\,\left(\, \frac{\p}{\p u}+B_t(u)\,\right),
\end{equation}
where
\begin{equation}\label{E:B_t(u)}
	B_t(u)\ = \ -\sum_{j=1}^{n-1}\,
	\oc_{t,u}(du)\cdot \oc_{t,u}(dy^j)\,\n^E_{t,u,\frac{\p}{\p y^j}},
\end{equation}
and $\n^E_{t,u}:= \n^\E_t\big|_{\pM\times\{u\}}$ is the restriction of the connection $\n^\E_t$ to $\pM\times\{u\}\subset M$.

Let $\Phi_t(u):E^+\to E^+$ ($0\le{}t\le1,\ u\le0$) be a continuous family of bundle maps and set
\[
	A_t(u)\ := \ B_t(u) \ - \ \Phi_t(u).
\]
The stability of the equivariant index, \cite[\S{}III.9]{lami}, implies the following
\begin{Theorem}\label{T:stability of the index}
Suppose $\Phi_t$ is chosen such that $A_\pM:= A_t(0)$ is independent of $0\le{}t\le1$. Then the equivariant index $\indl(D_t,A_\pM)$ is independent of $t$. In particular,
\begin{equation}\label{E:stability of the index}
	\indl\big(D,A_\pM\big)\ = \ \indl\big(D_0,A_\pM\big), \qquad \ell\in G.
\end{equation}
\end{Theorem}

If $D_0$ is a product near the boundary and $B_0(u)= B_0(0)= A_\pM$, then  we can apply Theorem~\ref{T:eq APS product} to compute $\ind(D,A_\pM)$.

\section{The equivariant index formula in the general case}\label{S:ind-form}
\setcounter{equation}{0}

In this section we extend Goette's equivariant APS-index theorem~\ref{T:eq APS product} to a setting
where neither the metric nor the Dirac bundle structure are product near the boundary.

\subsection{Equivariant transgression forms}\lb{eq-trans}
Recall the construction of an equivariant transgression form \cite{bgv}. Let $\n_t^{\EE}$ $(0\le t\le1$) be a family of $G$-invariant connections
on $\EE$. Let $\n^\E_{t,\grg}$ denote the corresponding equivariant connection, cf. \eqref{E:eq connection}, and let $F_{t,\grg}^{\E}$  denote the equivariant curvature of the connection $\n_{t,\grg}^{\E}$, cf. \eqref{E:eq curvature}.

Let $\beta_\grg(\n_t^{\E})(X)$ be the equivariant characteristic form associated to a germ $f(z)$ of an analytic function near zero, cf. Subsections~\ref{SS:eq char classes} and \ref{SS:twisting curvature}. In our applications $\beta_\grg$ will be
either the equivariant $\hat{A}$-genus \eqref{E:A genus}, or the  equivariant $\hat{L}$-from \eqref{E:L form}.

Then, \cite[Therem~7.7]{bgv}, the class of $\beta_\grg(\n_t^{\E})$ in $H_G(M)$ is independent of $t$ and
\begin{equation}\label{E:transgression1}
	\beta_\grg(\n_1^{\E}) \ - \ \beta_\grg(\n_0^{\E}) \ = \
	d_\grg T\beta_\grg(\n_1^{\E},\n_0^{\E}),
\end{equation}
where $T\beta_\grg(\n_1^{\E},\n_0^{\E})$ is the {\em equivariant transgression form}  for the characteristic class $\beta_\grg$ and  the pair of connections $\n^\E_0,\ \n^\E_1$ . If $\beta_\grg(\n_t^{\E})(X)= \Str f(F^\E_{t,\grg}(X))$, then (cf. the proof of Theorem~7.7 in \cite{bgv})
\begin{equation}\label{E:transgression2}
	T\beta_\grg(\n_1^{\E},\n_0^{\E})(X)\ = \ \int_0^1\Str\left[
		\frac{d\n^{\EE}_{t,\grg}}{dt}\,f'\big(F^{\EE}_{t,\grg}(X)\big)
												\right]\,dt.
\end{equation}
If $\beta_\grg(\n_t^{\E})(X)= \exp\Big(\Str f\big(F^\E_{t,\grg}(X)\big)\Big)$, then
\begin{equation}\label{E:transgression}
	T\beta_\grg(\n_1^{\E},\n_0^{\E})(X)\ = \
	  \int_0^1\beta_\grg(\n_t^{\E})(X)\wedge \Str\left[
		\frac{d\n_{t,\grg}^{\EE}}{dt}\,f'\big(F^{\EE}_{t,\grg}(X)\big)
												\right]\,dt.
\end{equation}

\begin{remark}\label{R:transgression depends}
Note that the notation $T\beta_\grg(\n_1^{\E},\n_0^{\E})$ for the transgression form is slightly misleading, since this form depends not only on the connections $\n_1^{\E}$ and $\n_0^{\E}$ but on the whole family $\n^\E_t$. Though the class of $T\beta_\grg(\n_1^{\E},\n_0^{\E})\in \Omega_G(\pM)/d_\grg\Omega_G(\pM)$ depends only on  $\n_1^{\E}$ and $\n_0^{\E}$, for explicit computations in Sections~\ref{S:signature 4D} and \ref{S:examples} we need to keep track of the actual differential form $T\beta_\grg(\n_1^{\E},\n_0^{\E})$ and not only of its class in $\Omega_G(\pM)/d_\grg\Omega_G(\pM)$. Since, by \eqref{E:transgression1}, $dT\beta_\grg(\n_1^{\E},\n_0^{\E})$ depends only on $\n_1^{\E}$ and $\n_0^{\E}$, omitting the family  $\n^\E_t$ from the notation is unlikely to lead to confusion.
\end{remark}

When the bundle $\E$ is the tangent bundle and the connections $\n_1^{\E}$ and $\n_0^{\E}$ are the Levi-Civita ones for the metrics $g_1^M,\ g_0^M$, respectively, we may write the transgression as $T\beta_\grg(g_1^M,g_0^M)(X)$. Here we do not assume that $\n^\E_t$ is a Levi-Civita connection of some metric for $t\not=0,1$.

\subsection{Transgression for the equivariant relative Chern form}\label{SS:transgressionChern}
The equivariant relative Chern form $\ch_\grg(\E/\calS)$, \eqref{E:twisting Chern}, is not a characteristic form of $\E$ in the sense of Subsection~\ref{SS:eq char classes}. In particular, it depends not only on the connection $\n^\E$ but also on the Clifford action and the Riemannian metric, cf. \eqref{E:twisting curvature}. If only the Clifford connection is changing and the Clifford action remains unchanged,  then the construction of the of the transgression form presented in the previous subsection works without any changes. However, if the Clifford action is changing this construction does not apply.  In this subsection we present a construction of a transgression form for $\ch_\grg(\E/\calS)$. Our strategy is to consider a good open cover $\{U_i\}$ of $M$ and construct a family of local transgression forms on each $U_i$. Then we show that those forms can be chosen in such a way that they coincide on the intersections and, hence, define a global equivariant transgression form on $M$.

Let $g^M_t$ be a family of Riemannian metrics on $M$. Let $c_t:T^*M\to \End(\E)$ be a family of Clifford actions compatible
with $g^M_t$ and $\n^E_t$ be a family of $G$-invariant connections compatible with $c_t$ in the sense of \eqref{E:Clifford
connection}. Let $F^\E_{t,\grg}$ be the equivariant curvature of $\n^\E_t$ and let $\ch_{t,\grg} (\E/\calS)$ be the relative
Chern form defined by $g^M_t$, $c_t$, and  $\n^\E_t$.

\subsubsection{The case of a spin-manifold}\label{SS:case of spin}
Consider first the situation, when $M$ is  a spin-manifold endowed with a bundle of spinors $\calS$, as in Remark~\ref{R:twisting curvature}.  Let $\n^\calS_t$ be the Levi-Civita connection on $\calS$ defined by the family of metrics $g^M_t$. Then, as in Remark~\ref{R:twisting curvature}, there exist a bundle $\W$ over $M$ and a family of connections $\n^\W_t$ such that $\E= \calS\otimes \W$ and
\[
	\n^\E_t \ = \ \n^\calS_t\otimes1\ +\ 1\otimes \n^\W_t.
\]
By \eqref{E:chE/S=chW},  $\ch_{t,\grg}(\E/\calS) = \ch_\grg(\n^\W_t)$. Let $\Tch_\grg(\n^\W_1,\n^\W_0)$ denote the transgression form for $\ch_\grg(\n^\W_t)$. Then $\Tch_\grg(\n^\W_1,\n^\W_0)$ is also a transgression form for $\ch_\grg(\E/\calS)$ in the sense that
\[
	d\Tch_\grg(\n^\W_1,\n^\W_0) \ = \
	\ch_{1,\grg}(\E/\calS)\ - \ \ch_{0,\grg}(\E/\calS).
\]

\subsubsection{Local coordinates}\label{SS:local coordinates}
If $M$ is not a spin-manifold, let $\{U_1,\ldots,U_N\}$ be a cover of $M$ by open contractable sets. Let $\calS_i\to U_i$ ($i=1,\ldots,N$) be a bundle of spinors over $U_i$ and let $\n^{\calS_i}_t$ denote the Levi-Civita connection on $\calS_i$ defined by the metric $g^M_t\big|_\pM$. For each pair $i,j\in \{1,\ldots,N\}$ we fix an isomorphism
\begin{equation}\label{E:phiij}
		\phi_{ij}:\, \calS_j\big|_{U_i\cap U_j}\ \to \ \calS_i\big|_{U_i\cap U_j},
\end{equation}
which commutes with the Levi-Civita connection and the Clifford action. Note that, in general, $\phi_{ij}$ do not satisfy the cocycle condition, i.e., $\phi_{ij}\circ\phi_{jk}\not= \phi_{ik}$.

Let $\E_i= \E\big|_{U_i}$  denote the restriction of $\E$ to $U_i$ and let $\n^{\E_i}_t$ denote the restriction of $\n^\E_t$ to $\E_i$. We fix isomorphisms
\begin{equation}\label{E:psii}
		\psi_{i}:\,\calS_i\otimes\CC^k\ \to \ \E_i, \qquad i=1,\ldots,N.
\end{equation}
Let $\psi_i^*\n^{\E_i}_t:= \psi_i^{-1}\circ\n^{\E_i}_t\circ\psi_i$ denote the pull-back of the connection $\n^{\E_i}_t$ to $\calS_i\otimes\CC^k$. Then
\begin{equation}\label{E:psiinabla}
		\psi_i^*\n^{\E_i}_t \ = \
			\n^{\calS_i}_t\otimes1\ +\ 1\otimes \n^{\CC^k}_{i,t},
\end{equation}
where $\n^{\calS_i}_t$ is the Levi-Civita connection on $\calS_i$ associated to the metric $g^M_t$ and
\[
	\n^{\CC^k}_{i,t} \ = \ d\ + \ A_{i,t}, \qquad A_{i,t}\in
	  \Omega^1(U_i,\operatorname{Mat}_{k\times k}(\CC)),
\]
is a family of connections on $M\times\CC^k\to M$.  Let
\[
	F_{i,t}\ : = \ dA_{i,t}\ + \ A_{i,t}\wedge A_{i,t},
\]
denote the curvature of $\n^{\CC^k}_{i,t}$.

Note that the sets $U_i$ are not necessarily $G$-invariant. Nevertheless, for $X\in \grg$ the restriction of the infinitesimal action $\calL^\E_X$ on $\Omega^\b(M,\E)$ to $\Omega^\b(U_i,\E_i)$ is well defined. Similarly, the infinitesimal action $\calL^{\calS_i}_X$ of $X$ on $\calS_i$ is well defined and there exists an action $\calL_{i,X}$ of $X$ on $M\times\CC^k\to M$ such that
\[
	\calL^\E\big|_{U_i}\ = \ \calL^{\calS_i}\otimes1\ + \ 1\otimes \calL_{i,X}.
\]
Using the action $\calL_{i,X}$ we define an equivariant version of $\n^{\CC^k}_{i,t}$ as follows. Set
\[
	\Omega_G(U_i,\CC^k)\ = \
	\left\{\, \omega\in \Omega^\b(U_i,\CC^k)[\grg]:\,
	\calL_{i,X}\omega(Y)= \omega([X,Y]) \ \
	\text{for all} \ X,Y\in \grg\,\right\}.
\]
Then (cf. \eqref{E:eq connection})
\[
	\n^{\CC^k}_{i,t,\grg} \omega(X)\ := \
	\n^{\CC^k}_{i,t}\omega(X)\ - \iota_{X_M}\omega(X),
	\qquad X\in \grg, \ \ \omega\in \Omega_G(U_i,\CC^k).
\]
Similarly, for $X\in \grg$, we define the moment of $X$ with respect to $\n^{\CC^k}_{i,t}$ by
\[
	\mu_{i,t}(X)\ := \ \calL_{i,X}\ - \ \n^{\CC^k}_{i,t;X_M},
\]
and the equivariant curvature  (cf. \eqref{E:eq curvature})
\[
	F_{i,t,\grg}(X) \ := \  F_{i,t}\ + \ \mu_{i,t}(X), \qquad X\in \grg.
\]

\subsubsection{Local equivariant transgression forms}\label{SS:local transgression}

We remark that if $M$ is a spin-manifold, $\calS$ is a spinor bundle over $M$, and $\E= \calS\otimes\W$ then $\n^{\CC^k}_{i,t,\grg}$ and $F_{i,t,\grg}$ coincide with the restriction of $\n^\W_{t,\grg}$ and $F^\W_{t,\grg}$ to $U_i$. In particular, the equivariant Chern form
\[
	\ch_{t,\grg}\big(\n^{\CC^k}_{i,t}(X)\big) \ := \ \Str\exp\big(-F_{i,t,\grg}(X)\big)
\]
is equal to the restriction of $\ch_{t,\grg}(\E/\calS)$ to $U_i$.

We define a {\em local equivariant transgression form}
\begin{equation}\label{E:local transgression}
		\Tch_\grg(\n^{\CC^k}_{i,1,\grg},\n^{\CC^k}_{i,0,\grg})(X) \ := \
	-\,\int_0^1 \
	 \Str\left[
		\frac{d\n^{\CC^k}_{i,t,\grg}(X)}{dt}\,\exp\big(-F_{i,t,\grg}(X)\big)
												\right]\,dt.
\end{equation}
The arguments of  \cite[Therem~7.7]{bgv} show that
\begin{multline}\notag
	d\Tch_\grg(\n^{\CC^k}_{i,1,\grg},\n^{\CC^k}_{i,0,\grg})(X) \ = \ 	
	\ch_{1,\grg}(\n^{\CC^k}_{i,1,\grg})(X) \ - \ \ch_{0,\grg}(\n^{\CC^k}_{i,0,\grg})(X)  \\ = \
	\ch_{1,\grg}(\E/\calS)\big|_{U_i} \ - \ \ch_{0,\grg}(\E/\calS)\big|_{U_i}.
\end{multline}

\subsubsection{A construction of a transgression form for the relative Chern form}\label{SS:transgression rel chern}
To finish the construction of the transgression form for the relative Chern form we need to glue the local transgression forms $\Tch_\grg\big(\n^{\CC^k}_{i,1,\grg}\big)$ into one global form on $M$. This is possible because of the following

\begin{lemma}\label{L:trangression relative chern}
For $i,j\in \{1,\ldots,N\}$ we have
\begin{equation}\label{E:transgression relative chern}
	\Tch_\grg(\n^{\CC^k}_{i,1,\grg},\n^{\CC^k}_{i,0,\grg})(X)\big|_{U_i\cap U_j}
	\ = \
	\Tch_\grg(\n^{\CC^k}_{j,1,\grg},\n^{\CC^k}_{j,0,\grg})(X)\big|_{U_i\cap U_j}.
\end{equation}
\end{lemma}
It follows from the above lemma that there exists a unique  global equivariant differential form $\Tch_\grg(\E/\calS)(X)$ such that for each $i=1,\ldots,N$
\[
	\Tch_\grg(\E/\calS)(X)\big|_{U_i}\ = \
	\Tch_\grg(\n^{\CC^k}_{i,1,\grg},\n^{\CC^k}_{i,0,\grg})(X),
\]
and, hence,
\begin{equation}\label{E:transgression rel chern}
	d\Tch_\grg(\E/\calS)(X) \ = \
	\ch_{1,\grg}(\E/\calS) \ - \ \ch_{0,\grg}(\E/\calS).
\end{equation}

\subsubsection*{Proof of Lemma~\ref{L:trangression relative chern}}
Using \eqref{E:phiij} and \eqref{E:psii} we define the transition functions
\[
	\Phi_{ij}(x):\, \CC^k\ \to \ \CC^k, \qquad x\in U_i\cap U_j,
\]
such that
\[
	\phi_{ij}\otimes 1\ + \ 1\otimes\Phi_{ij} \ = \ \psi_j^{-1}\circ \psi_i.
\]
From \eqref{E:psiinabla} we now obtain
\begin{multline}\label{E:psi*nabla}
	\n^{\calS_j}_t\otimes1\ +\ 1\otimes \n^{\CC^k}_{j,t}	 \ = \
	\psi_j^*\n^{\E_i}_t\ = \
	\psi_j^{-1} \psi_i\circ\Big(\,\psi_i^{-1}\circ\n^{\E_i}_t\circ\psi_i\,\Big)
	\circ\psi_i^{-1} \psi_j
	\\ = \
	\psi_j^{-1} \psi_i\circ
	\big(\,\n^{\calS_i}_t\otimes1\ +\ 1\otimes \n^{\CC^k}_{i,t}\,\big)
	\circ  \psi_i^{-1} \psi_j
	\\ = \
	\phi_{ij}\circ\n^{\calS_i}_t\circ\phi_{ij}^{-1}\otimes1\ +\
	 1\otimes \Phi_{ij}\circ\n^{\CC^k}_{i,t}\circ\Phi_{ij}^{-1}.
\end{multline}
Since by construction $\phi_{ij}$ commute with the Levi-Civita connection, we conclude from \eqref{E:psi*nabla} that
\[
	\n^{\CC^k}_{j,t}\big|_{U_i\cap U_j} \ = \
	\Phi_{ij}^{-1}\circ\n^{\CC^k}_{i,t}\big|_{U_i\cap U_j}\circ \Phi_{ij}
\]
and, hence,
\[
	A_{j,t}(x)\ = \ \Phi_{ij}^{-1}\circ A_{i,t}(x)\circ\Phi_{ij} \ + \
	\Phi_{ij}^{-1}\circ d\Phi_{ij}, \qquad x\in U_i\cap U_j.
\]
It follows that
\[
	\frac{d\n^{\CC^k}_{j,t,\grg}}{dt}\ = \ \frac{dA_{j,t}}{dt} \ = \
	\Phi_{ij}^{-1}\circ \frac{d\n^{\CC^k}_{i,t,\grg}}{dt}\circ\Phi_{ij}.
\]
Also
\[
	F_{j,t,\grg}(X)\ = \
	\Phi_{ij}^{-1}\circ F_{i,t,\grg}(X)\circ\Phi_{ij}.
\]
Hence,
\[
	\frac{d\n^{\CC^k}_{j,t,\grg}}{dt}\,\exp\big(-F_{j,t,\grg}\big)\ = \
	\Phi_{ij}^{-1}\circ
	\frac{d\n^{\CC^k}_{i,t,\grg}}{dt}\,\exp\big(-F_{i,t,\grg}\big)
	\circ\Phi_{ij},
\]
and
\begin{equation}\label{E:Str const}
		\Str\left[
		\frac{d\n^{\CC^k}_{j,t,\grg}}{dt}\,\exp\big(-F_{j,t,\grg}(X)\big)
												\right]
	\ = \
	\Str\left[
		\frac{d\n^{\CC^k}_{i,t,\grg}}{dt}\,\exp\big(-F_{i,t,\grg}(X)\big)
												\right].
\end{equation}				
From \eqref{E:Str const} and \eqref{E:local transgression} we obtain \eqref{E:transgression relative chern}.\hfill$\square$

We remark that the transgression form $\Tch_\grg(\E/\calS)$, constructed in Lemma~\ref{L:trangression relative chern}, depends not only of the family of connections $\n^\E_t$ but also on the family of Clifford actions $c_t$.

\subsection{Transgression of a product}\label{SS:transgression_	product}
Suppose now that $\E_1$ and $\E_2$ are two $G$-equivariant vector bundles over $M$ and  $\n^{\E_1}_t$, $\n^{\E_2}_t$ are smooth families of connections on $\E_1$ and $\E_2$ respectively. Let
\[
	\beta_{1,\grg}(\n^{\E_1}_t)(X), \ \ \beta_{2,\grg}(\n^{\E_1}_t)(X)
	\ \in \ \Omega^\b(M),
\]
be two characteristic forms and consider the product form
\[
	\beta_{\grg,t}(X)\ := \ \beta_{1,\grg}(\n^{\E_1}_t)(X)\wedge
	\beta_{2,\grg}\n^{\E_2}_t)(X).
\]
We define the transgression $T\beta_\grg(X)$ by
\begin{equation}\label{E:transgression_product}
	T\beta_\grg(X) \ := \
	T\beta_{1,\grg}(\n^{\E_1}_0,\n^{\E_1}_1)(X)\wedge \beta_{2,\grg}(\n^{\E_2}_1)(X)
	 +
	\beta_{1,\grg}(\n^{\E_1}_0)(X)\wedge T\beta_{2,\grg}(\n^{\E_2}_0,\n^{\E_2}_1)(X).
\end{equation}
One readily sees that
\[
	d_\grg\, T\beta_\grg(X)\ = \ \beta_{\grg,1}(X)\ - \ \beta_{\grg,0}(X).
\]

\subsection{Transgression of the Atiyah-Singer integrand}\label{SS:transgression AS}
Let
\[
		\al_\grg(\nab^\EE)(X)\ :=\
	\hat{A}_\grg(g^M)(X)\cdot\ch_\grg(\cal{E}/\calS)(X),
\]
denote the integrand of the Atiyah-Singer index formula. This is a product of two forms. Suppose $(\E,c_t,\n^\E_t)$ is a family of Dirac bundles as in Section~\ref{S:deformation}. We now construct a transgression form for $\alpha_\grg$ and the pair of connections $\n^\E_0$ and $\n^\E_1$.

The form $\hat{A}_\grg$ is a usual characteristic form and has a transgression form defined by \eqref{E:transgression}. The transgression form of the relative Chern class $\ch_\grg(\cal{E}/\calS)(X)$ is defined in Subsection~\ref{SS:transgression rel chern}. Hence, we can define the transgression form $T\al_\grg(\nab^\EE,\n_0^{\EE})(X)$ by formula \eqref{E:transgression_product}, which is valid even when one of the characteristic
forms is not of the standard type described in Subsection~\ref{SS:eq char classes}. We note that even though the differential
\begin{equation}\label{E:AStransgression}
		dT\al_\grg(\nab^\EE,\n_0^{\EE})(X) \ = \alpha_g(\n^\E_1)(X)\ - \ \alpha_g(\n^\E_0)(X)
\end{equation}
depends only on the family of connections $\n^\E_t$, as in Remark~\ref{R:transgression depends}, the transgression form $T\alpha_\grg$ depends also on the family of Clifford actions $c_t$.

\subsection{The general equivariant APS-index theorem}\label{SS:general APS}

We are now ready to formulate our principal result -- the equivariant APS-type index theorem in the non-product case.

Let $(\E,c,\n^\E)$ be a $G$-equivariant Dirac bundle over $M$ and let $E= \E\big|_\pM$.

\begin{Definition}\label{D:admissible bc}
A boundary operator $A_\pM:C^\infty(\pM,E^+)\to C^\infty(\pM,E^+)$ is called {\em admissible for $\E$} if there exists a connection $\tiln^E=\tiln^{E^+}\oplus\tiln^{E^-}$ on $E$ compatible with the restriction of the Clifford action to the boundary (cf. Subsection~\ref{SS:deformation Dirac budnle}) such that
\begin{equation}\label{E:admissible ApM}
	A_\pM\ = \ -\sum_{j=1}^{n-1} c(e^n)\,c(e^j)\,\tiln^{E^+}_{e_j}.
\end{equation}
Here $\{e_1,\ldots,e_{n-1}\}$ is a basis of $T\pM$, $e_n$ is the inward unit normal vector to $\pM$ and $\{e^1,\ldots,e^{n}\}$ is the dual bundle of $T^*M$. We say that {\em $A_\pM$ is defined by the boundary connection $\tiln^E$}.
\end{Definition}
\

It follows from   Proposition~\ref{P:deformation Dirac} that for every admissible boundary operator $A_\pM$ there exists an admissible deformation $(\E,c_t,\n^\E_t)$ of the Dirac bundle $(\E,c,\n^\E)$ such that $\n^\E_0\big|_\pM= \tiln^E$.   Let $(\E,c_t,\n^\E_t)$ be such an admissible deformation and let $D_t$ denote the Dirac operator associated with $(\E,c_t,\n^\E_t)$. Then $D_1=D$ and near the boundary $D_0$ has the form
\[
	D_0\ = \ c(du)\,\left(\,\frac{\p}{\p u}+ A_\pM\,\right).
\]

\begin{Theorem}\label{T:genral APS}
Let $(M,g^M)$ be a compact oriented even-dimensional Riemannian manifold with boundary $\pM$, and let $G$ be a compact Lie group acting
by orientation-preserving isometries on $M$. Fix a $G$-equivariant Dirac bundle  $(\E,c,\n^\E)$ over $M$ with Clifford connection $\n^\E$ and associated Dirac operator $D$. Let $A_\pM$ be an admissible boundary operator defined by  the boundary connection $\tiln^E$.

If $X\in\grg$ is small and the corresponding vector field  $X_M$ has no zeros on $\pM$, then for any admissible deformation $(\E,c_t,\n^\E_t)$ with $\n^\E_0\big|_\pM= \tiln^E$ we have
\begin{equation}\label{E:general APS}
	\ind(D,A_\ppM)\ = \
	(2\pi i)^{-n/2}\left[\int_M\al_\grg(\nab^\EE)(X)
	\ - \
	\int_\pM T\al_\grg(\nab^\EE,\n_0^{\EE})(X)\right]\ -\ \tilde{\eta}_X(A_\ppM),
\end{equation}
where
\begin{equation}\label{E:tilde eta}
	\al_\grg(\nab^\EE)(X)\ :=\
	\hat{A}_\grg(g^M)(X)\cdot\ch_\grg(\cal{E}/\calS)(X),\qquad
	\tilde{\eta}_X(A_{\ppM})\ :=\
	\frac{\eta_X(A_{\ppM})+h_{e^{-X}}(A_{\ppM})}2,
\end{equation}
and $T\al_\grg(\nab^\EE,\n_0^{\EE})(X)$ is the equivariant
transgression form of $\al_\grg(\nab^\EE)(X)$ described in Subsection~\ref{SS:transgression AS}.
\end{Theorem}

\begin{remark}\label{R:conpare w/product}
Comparing this fomrula with Goette's formula \eqref{E:eq APS produc},
aside from considering generalized APS boundary conditions, the sole difference
is the appearance of a boundary transgression term. This exactly mirrors Gilkey's
addition to the Atiyah-Patodi-Singer formula for the (non-equivariant) index.
The proof will also follow Gilkey's idea, but we will not be using
any a priori knowledge regarding the form of the boundary term, thus
avoiding the use of invariance theory. A more superficial difference
from his proof is that the ($G$-equivariant) collar we use is a part of $M$,
rather than being external to it.
\end{remark}

\begin{remark}\label{R:transgression depends2}
As in Remark~\ref{R:transgression depends}, the transgression form $T\al_\grg(\nab^\EE,\n_0^{\EE})(X)$ depends on the choice of the admissible deformation $(\E,c_t,\n^\E_t)$. However, formula \eqref{E:general APS} shows that the integral $\int_\pM T\al_\grg(\nab^\EE,\n_0^{\EE})(X)$ does not depend on this choice. It depends only on the restrictions of the Dirac bundles $(\E,c_i,\n^\E_i)$ ($i=0,1$) to the boundary.
\end{remark}

\begin{proof}
Let $D_0$ be the Dirac operator corresponding to $\n_0^{\EE}$. Then
by Theorems~\ref{T:stability of the index}  and \ref{T:eq APS product}
\begin{multline}\label{E:computation of ind}
	\ind(D,A_\ppM)\ =\ \ind(D_0,A_\ppM)\
	=\ \int_M\alpha_\grg(\n_0^{\EE})(X) \ -\ \tilde{\eta}_X(A_\ppM)
  	\\ = \
    \int_{M} \alpha_\grg(\n^{\EE})(X)\ +\
    \int_M \Big[\,\alpha_\grg(\n_0^{\EE})(X)- \alpha_\grg(\nab^{\EE})(X)\,\Big]
    \ - \  \tilde{\eta}_X(A_\ppM)
\end{multline}
Using the transgression form \eqref{E:AStransgression}, the second term in the right hand side of \eqref{E:computation of ind} can be computed  as
\begin{equation}\label{E:intU1}
		\int_M \Big[\,\alpha_\grg(\n_0^{\EE})(X)- \alpha_\grg(\nab^{\EE})(X)\,\Big]
		\ = \
	-\int_M d_\grg\,T\alpha_\grg(\nab^{\EE},\n_0^{\EE})(X)
	\ = \
	-\int_{\pM} T\alpha_\grg(\nab^{\EE},\n_0^{\EE})(X).
\end{equation}

Combining \eqref{E:computation of ind} and \eqref{E:intU1} we obtain \eqref{E:general APS}.
\end{proof}

\section{A choice of the boundary conditions}\label{S:special bc}

The equivariant APS index formula take an especially nice and useful form if we choose convenient boundary conditions $A_\pM$. In the next two sections we discuss some such choices. In this section we introduce boundary conditions convenient for a general generalized Dirac operator (in particular, for a twisted Dirac operator on a spin manifold). In the next section we discuss boundary conditions suitable for studying the twisted signature operator.

\subsection{A choice of a connection on the product bundle}\label{SS:connection on product}
A product connection on the cylinder $\pM\times[a,0]$ is defined by its restriction to the boundary. One possible choice of such a connection is the restriction $\n^\E|_{\pM}$ of $\n^\E$ to $\pM$. However, that is not the most natural and the most convenient choice for many applications. To see an example of this, suppose that $M$ is a spin-manifold and  $\E=\calS$ is the bundle of spinors. Let $\n^\calS$ be the Levi-Cevita connection on $\calS$.

Let $\{e_1,\ldots,e_{n-1}\}$ be a local orthonormal frame of $T\pM$ and let $e_n$ be the inward unit normal vector to $\pM$.
We call $\{e_1,\ldots,e_{n}\}$ an adapted orthonormal frame for $TM$ along $\pM$. We use the parallel transport along the rays $y\times(-\infty,0]$ to extend $\{e_1,\ldots,e_{n}\}$ to an orthonormal frame of $T\big(\pM\times(-\infty,0]\big)$. Let $\{e^1,\ldots,e^{n}\}$ denote the dual frame of $T^*\big(\pM\times(-\infty,0]\big)$. The frame $\{e_1,\ldots,e_{n}\}$ induces a trivialization of $\calS$ and we denote by $\p_i$ the derivative along $e_i$ defined by this trivialization.  By formula (3.13) of \cite{bgv} we have
\begin{equation}\label{E:nS}
	\n^\calS_{e_i} \ = \ \p_i\ + \
	 \frac14\,\sum_{j,k=1}^n\,\omega_{ij}^k\,c(e^j)\,c(e^k),
\end{equation}
where $\omega_{ij}^k:= e^k(\n_{e_i}^{LC}e_j)$ are the coefficients of the Levi-Civita connection. Notice that
\begin{equation}\label{E:omega=-omega}
	\omega_{ij}^k\ =\ -\omega_{ik}^j.
\end{equation}

Let $S$ denote the restriction of $\calS$ to $\pM$. Since $n$ is even, $S$ is a direct sum of two copies of spinor bundles on $\pM$. Let $\n^S$ denote the Levi-Civita connection on $S$. Using \eqref{E:nS} and \eqref{E:omega=-omega} we obtain
\begin{equation}\label{E:nS-ncalS}
	\n^S_{e_i} \ = \ \p_i\ + \
	 \frac14\,\sum_{j,k=1}^{n-1}\,\omega_{ij}^k\,c(e^j)\,c(e^k)
	 \ = \ \n^\calS_{e_i}\ - \ \frac{1}{2}\,
	    \sum_{j=1}^{n-1}\, L_{ij}\,c(e^j)\,c(e^n),
\end{equation}
where $L_{ij}:= \omega_{ij}^n$ is the second fundamental form of the embedding of $\pM$ into $M$. In particular, $\n^S\not= \n^\calS\big|_\pM$.

Let now $\E$ be an arbitrary Dirac bundle over $M$ and let $E$ denote its restriction to $\pM$. Motivated by \eqref{E:nS-ncalS} we define a connection $\tiln^E$ on $E$ by
\begin{equation}\label{E:nE}
	\tiln^E_{e_i}  \ := \ \n^\E_{e_i}\ - \ \frac{1}{2}\,
	    \sum_{j=1}^{n-1}\, L_{ij}\,c(e^j)\,c(e^n),
	   \qquad i=1,\ldots, n-1.
\end{equation}

\subsection{The restriction of the twisting curvature}\label{SS:restriction of twisting curvature}
The curvature of the connection $\tiln^E$ is not equal to the restriction of the curvature of $\n^\E$ to $\pM$. However, the  twisting curvature (cf. Subsection~\ref{SS:twisting curvature}) of $\tiln^E$ is equal to the restriction of the twisting curvature of $\n^\E$, as we now explain.

\begin{lemma}\label{L:twisting curvature}
Let $\tiln^E$ be defined by \eqref{E:nE} and let $\tilF^{E/S}$ denote the twisting curvature of $\tiln^E$. Then $\tiln^E$ is compatible with the restriction of the Clifford action to the boundary (cf. Subsection~\ref{SS:deformation Dirac budnle}) and
\begin{equation}\label{E:twisting=restriction}
	\tilF_\grg^{E/S}\ = \ F_\grg^{\E/\calS}\big|_{\pM}.
\end{equation}
\end{lemma}

\begin{proof}
Since the statement is local we can assume that $M$ is a spin-manifold and $\E=\calS\otimes\W$, cf. Proposition~3.35 of \cite{bgv}. Moreover, by Proposition~3.40 of \cite{bgv} there exists a connection $\n^\W$ on $\W$ such that
\[
	\n^\E\ = \ \n^\calS\otimes1\ + \ 1\otimes\n^\W.
\]
Let $W=\W|_{\pM}$ denote the restriction of $\W$ to the boundary and set $\n^W:= \n^\W\big|_{\pM}$. Then the equivariant curvature
\[
	\tilF^W_\grg\ =\  (\n^W)^2+\mu^W(X)
\]
is given by the restriction of $F^\W_\grg$ to the boundary: $F^W_\grg= F^\W_\grg\big|_{\pM}$.

Recall that we denote by $S$ the restriction of $\calS$ to the boundary. Then $E= S\otimes{}W$ and
\[
	 \n^\E\big|_\pM\ = \ \n^\calS\big|_\pM\otimes1\ + \ 1\otimes\n^W.
\]
From \eqref{E:nS-ncalS} and \eqref{E:nE} we now conclude that
\[
	\tiln^E\ = \ \n^S\otimes 1\ + \ 1\otimes\n^W.
\]
Since the Levi-Civita connection $\n^S$ is  compatible with the restriction of the Clifford action to the boundary, so is $\tiln^E$. Also, by Remark~\ref{R:twisting curvature}, $\tilF_\grg^{E/S}= 1\otimes{}F_\grg^W$, cf. Page~121 of \cite{bgv}. Similarly, $F_\grg^{\E/\calS}= 1\otimes{}F_\grg^\W$ so that $F_\grg^{\E/\calS}\big|_\pM= \tilF_\grg^{E/S}$.
\end{proof}

\subsection{The APS-index with special boundary conditions}\label{SS:special APS}
We now formulate a refinement of the Index Theorem~\ref{T:genral APS}, for boundary conditions defined by the connection $\tiln^E$.

\begin{theorem}\label{T:special APS}
Let $\tiln^E$ be given by \eqref{E:nE} and let $A_\pM$ be an admissible boundary operator  defined by $\tiln^E$, cf. Definition~\ref{D:admissible bc}. Let $g_t^M$ be an admissible deformation of $g^M$, cf. Definition~\ref{D:admissible gMt}. Let $T\hat{A}_\grg(g^M,g_0^M)(X)$ be the transgression form of the $\hat{A}$-form associated to this family. Then for small enough $X\in\grg$ such that the corresponding vector field  $X_M$ has no zeros on $\pM$, we have
\begin{multline}\label{E:special APS}
	\ind(D,A_\ppM)\ = \
	(2\pi i)^{-n/2}\Big[\,
	\int_M\hat{A}_\grg(g^M)(X)\cdot\ch_\grg(\cal{E}/\calS)(X)
	\\ - \
	\int_\pM T\hat{A}_\grg(g^M,g_0^M)(X)\cdot\ch_\grg(F^{\Es})(X)
	\,\Big]\ - \ \tilde{\eta}_X(A_\ppM),
\end{multline}
where $T\hat{A}_\grg(g^M,g_0^M)(X)$ is the equivariant
transgression form of the $\hat{A}_\grg$-form described in Subsection~\ref{eq-trans}.
\end{theorem}

\begin{proof}
Let $\al_\grg(\nab^\EE)(X)$ be as in Theorem~\ref{T:genral APS}. By Proposition~\ref{P:deformation Dirac} there exists an admissible deformation $(\E,c_t,\n^\E_t)$ of $(\E,c,\n^\E)$ such that $c_t(\xi)^2=-g^M_t(\xi,\xi)$ for all $\xi\in T^*M$ and the restriction of $\n^\E_0$ to the boundary is equal to $\tiln^E$.  Let $\n^E_t= \n_t^\E\big|_\pM$ and let $L_{t,ij}$ be the second fundamental form of the embedding of $\pM$ into $M$ defined by the metric $g^M_t$. Set (cf. \eqref{E:nE})
\begin{equation}\label{E:nEt}
	\tiln^E_{t,e_i}  \ := \ \n^E_{t,e_i}\ - \ \frac{1}{2}\,
	    \sum_{j=1}^{n-1}\, L_{t,ij}\,c_t(e^j)\,c_t(e^n),
	   \qquad i=1,\ldots, n-1.
\end{equation}
Since the metric $g^M_0$ is a product near the boundary, $\n^E_0= \tiln^E_0$. Hence, $\tiln^E= \tiln_0^{E}$. It follows that
the twisting equivariant curvatures $\tilF^{E/S}_\grg$ and $\tilF^{E/S}_{0,\grg}$ of the connections $\tiln^E$ and $\tiln^E_0$ are equal. By Lemma~\ref{L:twisting curvature}, the restrictions of $F^{\E/\calS}_{\grg}$ and $F^{\E/\calS}_{0,\grg}$ to the boundary are also equal.

Let $\ch_\grg(\E/\calS)$ and $\ch_{0,\grg}(\E/\calS)$ denote the equivariant relative Chern characters of $(\E,c,\n^\E)$ and $(\E,c_0,\n^\E_0)$, respectively. It follows from the previous paragraph that the restrictions of $\ch_\grg(\E/\calS)$ and $\ch_{0,\grg}(\E/\calS)$ to $\pM$ are equal. Hence, along $\pM$
\begin{multline*}
	d_\grg T\alpha_\grg(\nab^{\EE},\n_0^{\EE})(X)\ = \
	\alpha_\grg(\nab^{\EE})(X)-\alpha_\grg(\n_0^{\EE})(X)
	\\ = \
	\hat{A}_\grg(g^M)(X)\cdot\ch_\grg(\E/\calS)(X)
	      \ - \ \hat{A}_\grg(g_0^M)(X)\cdot\ch_{0,\grg}(\E/\calS)(X)
	\\ = \
	\Big[\,\hat{A}_\grg(g^M)(X)-\hat{A}_\grg(g_0^M)\,\Big]\cdot
			\ch_\grg(\E/\calS)(X)
	\\ = \
	d_\grg\,T\hat{A}_\grg(g^M,g_0^M)(X)\cdot\ch_\grg(\E/\calS)(X).
\end{multline*}
The theorem now follows from Theorem~\ref{T:genral APS}.
\end{proof}

\section{The equivariant signature}\label{S:signature}

In this section we apply Theorem~\ref{T:genral APS} to compute the equivariant index of the twisted signature operator with boundary conditions which are very natural for this operator (we note that these boundary conditions are quite different from those considered in Section~\ref{S:special bc}). Throughout the section we assume that $\dim M=n=2m$, and let $\calV$ be a Hermitian vector bundle over $M$ endowed with a Hermitian connection $\n^\calV$. We study the equivariant index of the twisted   signature operator $D=\n^\calV+(\nabla^\calV)^*$. In the case when the connection $\n^\calV$ is flat this index is equal to the equivariant signature of the local system $(\calV,\n^\calV)$.

\subsection{The signature operator}\label{SS:signature operator}
Consider the bundle
\[
	\E \ = \ \Lambda^\b{}T^*M\times\calV.
\]
Then the space of smooth sections of $\E$ is naturally identified with the space $\Omega^\b(M,\calV)$ of differential forms on $M$ with values in $\calV$. The connection $\n^\calV$ on $\calV$ and the Levi-Civita connection on $T^*M$ define a $G$-invariant connection $\n^\E$ on $\E$. Let $c:T^*M\to \End(\E)$ denote the natural Clifford action, given, cf. \cite[Ch.~3]{bgv}, by
\[
	c(e^i)\,\alpha\ := \ e^i\wedge\alpha \ - \ \iota_{e_i}\,\alpha,
	\qquad \alpha\in \Lambda^\b(T^*M)\otimes \calV, \ \ i=1,\ldots,n.
\]
Here  $\iota_{e_i}$ denotes the interior multiplication by $e_i$.

 We consider the grading $\E=\E^+\oplus\E^-$ defined by the $\pm1$ eigenspaces of the {\em chirality operator}
\begin{equation}\label{E:chirality}
 \begin{aligned}
	\Gamma_M:\Omega^\b(M,\calV)&\to \Omega^{n-\b}(M,\calV), \\
	\Gamma_M\,\omega\ := \ i^{n/2+k(k+1)}\,\ast\,&\omega,
	\qquad \text{for}\quad \omega\in \Omega^k(M,\calV),
 \end{aligned}
\end{equation}
where $\ast$ denotes the Hodge star operator. The power of $i$ in \eqref{E:chirality} is chosen so that $\Gamma_M^2=1$. We also notice that
\begin{equation}\label{E:chirality 2}
	\Gamma_M \ = \ i^{n/2}\,c(e^1)\cdots c(e^n),
\end{equation}
where $e^1,\ldots,e^n$ is an orthonormal basis of $T^*M$.

The bundle $\E$ endowed with the Clifford action, the grading, and the connection defined above is a Dirac bundle. The corresponding Dirac operator $D$ is called the {\em twisted signature operator}. By \cite[Proposition~3.53]{bgv}
\[
	D\ = \ \n^\calV\ + \ \big(\n^\calV\big)^*,
\]
where
\[
	\big(\n^\calV\big)^*\ = \ \Gamma_M\circ \n^\calV\circ \Gamma_M:\,
	\Omega^\b(M,\calV)\ \to \ \Omega^{\b-1}(M,\calV)
\]
denotes the formal adjoint of $\n^\calV$.

\subsection{The odd signature operator}\label{SS:odd signature}
To fix the boundary conditions for the signature operator $D$ we first describe a natural operator on $\pM$, called the {\em odd signature operator}, cf. \cite{aps2}.

Let $V$ and $E^+$ denote the restriction of $\calV$ and $\E^+$ to  $\pM$. We set $\n^V:= \n^\calV\big|_\pM$.

Let $e_1,\ldots,e_n$ be an adapted orthonormal basis of $TM\big|_\pM$ constructed as in Subsection~\ref{SS:connection on product}. In particular, this means that $e_1,\ldots,e_{n-1}$ is an orthonormal basis of $T\pM$. Let $e^1,\ldots,e^n$ be the dual basis of $T^*M\big|_\pM$. Then
\[
	\Gamma_\pM\ = \ i^{n/2}\,c(e^1)\cdots  c(e^{n-1}):\,
	\Omega^\b(\pM,V)\ \to\ \Omega^{n-1-\b(}\pM,V)
\]
satisfies $\Gamma_\pM^2=1$. We refer to this operator as the {\em chirality operator on $\pM$}. Notice that
\begin{equation}\label{E:Gamma=Gamma}
	\Gamma_M\ = \ \Gamma_\pM\cdot c(e^n) \ = \ -c(e^n)\cdot \Gamma_\pM.
\end{equation}

\begin{Definition}\label{D:odd signature}
The {\em odd signature operator} is the operator
\begin{equation}\label{E:odd signature}
	D_\pM\ : = \ \Gamma_\pM\circ\n^V\ + \ \n^V\circ\Gamma_\pM:\,
	\Omega^\b(\pM,V) \ \longrightarrow \ \Omega^\b(\pM,V).
\end{equation}
\end{Definition}
By \cite[Proposition~3.58(3)]{bgv}, $(\n^V)^*=\Gamma_\pM\circ\n^V\circ\Gamma_\pM$. Hence,
\begin{equation}\label{E:DpM=GammaD}
	D_\pM\ = \ \Gamma_\pM\circ \Big(\,\n^V+(\n^V)^*\,\Big).
\end{equation}

Since $D_\pM$ preserves the parity of differential forms,
\begin{equation}\label{E:Dpm=+-}
	D_\pM\ = \
	\begin{pmatrix}
	D_\pM^+&0\\0&D_\pM^-
	\end{pmatrix},
\end{equation}
where $D_\pM^+:\Omega^\even(\pM,V)\to \Omega^\even(\pM,V)$ and $D_\pM^-:\Omega^\odd(\pM,V)\to \Omega^\odd(\pM,V)$. We note that
\[
	\Gamma_\pM\circ D_\pM^+\circ \Gamma_\pM\ = \ D_\pM^-.
\]
In particular, $D_\pM^+$ and $D_\pM^-$ have the same spectrum. It follows that
\begin{equation}\label{E:eta Dpm}
 \begin{aligned}
	\eta_X(D_\pM) \ &= \ 2\, \eta_X(D_\pM^+) \ = \ 2\,\eta_X(D_\pM^-),\\
	h_\ell(D_\pM) \ &= \ 2\, h_\ell(D_\pM^+) \ = \ 2\, h_\ell(D_\pM^-),
	\qquad \ell\in G,
 \end{aligned}
\end{equation}
where $\eta_X$ stands for Goette's infinitesimal eta-invariant \eqref{E:inf eta} and $h_\ell(B):= \Tr\big(\ell\big|_{\ker B}\big)$.

\subsection{An identification of $E^+$ and $\Lambda^\b(T^*\pM)\otimes{}V$}\label{SS:E^+=Lambda}

Let $e_1,\ldots,e_n$ be an adapted orthonormal basis of $TM\big|_\pM$, constructed as in Subsection~\ref{SS:connection on product}.  Then
\begin{equation}\label{E:Lam on pM}
	\Lambda^\b(T^*M)\big|_\pM \ = \
	\Lambda^\b(T^*\pM)\,\oplus\, \Big(\, e^n\wedge\Lambda^\b(T^*\pM)\,\Big).
\end{equation}
Using this decomposition we identify $\Lambda^\b(T^*\pM)\otimes{}V$ with a subspace of $E= \Lambda^\b(T^*M)\big|_\pM\otimes{}V$. For $\omega\in \Lambda^\b(T^*\pM)$  we have $(1+\Gamma_M)\omega\in E^+$. Thus
\begin{equation}\label{E:1+Gamma}
	1+\Gamma_M:\,\Lambda^\b(T^*\pM)\otimes{}V \ \longrightarrow \ E^+.
\end{equation}
Moreover, using \eqref{E:Gamma=Gamma} we obtain
\begin{equation}\label{E:1+Gamma omega}
	\big(\,1+\Gamma_M\,\big)\,\omega\ = \ \omega\ - \ c(e^n)\cdot \Gamma_\pM\,\omega
	\ = \ \omega\ - \ e^n\wedge \Gamma_\pM\,\omega.
\end{equation}
Let
\[
	\Pi:\, E=\Lambda^\b(T^*M)\big|_\pM\otimes V \ \longrightarrow \ \Lambda^\b(T^*\pM)\otimes V
\]
denote the projection onto the first summand of \eqref{E:Lam on pM} and let $\Pi^+$ denote the restriction of $\Pi$ to $E^+$. It follows from \eqref{E:1+Gamma omega} that $\Pi^+$ is the inverse of \eqref{E:1+Gamma}. Thus \eqref{E:1+Gamma} is an isomorphism of bundles which we use to identify $E^+$ with $\Lambda^\b(T^*\pM)\otimes{}V$. Via this identification $D_\pM$ induces an operator
\begin{equation}\label{E:tilD}
	\tilD_\pM\ := \ \big(\,1+\Gamma_M\,\big)\circ D_\pM\circ \Pi^+:\,
	C^\infty(\pM,E^+)\ \to \ C^\infty(\pM,E^+).
\end{equation}

We use the operator $\tilD_\pM$ to define the boundary conditions for $D$. It is shown below that this is a generalized APS boundary condition, cf. Subsection~\ref{SS:APS}. In particular, the equivariant index $\indl(D,\tilD_\pM)$ ($\ell\in G$) is well defined.

From \eqref{E:tilD} we conclude that
\begin{equation}\label{E:D=tildeD}
	\eta_X(\tilD_\pM)\ = \ \eta_X(D_\pM), \qquad
	h_\ell(\tilD_\pM)\ = \ h_\ell(D_\pM),
	\qquad\text{for all}\quad X\in\grg, \ \ \ell\in G.
\end{equation}

\subsection{The product case}\label{SS:product signature}
As in \eqref{E:boundary form} the restriction of the  signature operator $D$ to the boundary can be decomposed as
\begin{equation}\label{E:signature boundary}
	D\big|_U\ = \ c(e^n)\,\left(\,\frac{\p}{\p u} \ + \ B(u)\,\right),
\end{equation}
with
\begin{equation}\label{E:B signature}
	B(0) \ = \ -\,\sum_{i=1}^{n-1}\,c(e^n)\,c(e^i)\,\n^{E^+}_{e_i},
\end{equation}
where, as usual,
\[
	\n^E\ := \ \n^{\Lambda^\b T^*M}\big|_{\pM}\otimes 1\ + \ 1\otimes \n^V
\]
denotes the restriction of $\n^\E$ to the boundary.

\begin{lemma}\label{L:B=DpM}
Suppose the metric $g^M$ is product near the boundary. Then
\begin{equation}\label{E:B(0)=DpM}
	B(0)\ = \ \tilD_\pM.
\end{equation}
\end{lemma}

\begin{proof}
Since $\Gamma_M$ commutes with $\n^\E$ and $c(e^i)$ ($i=1,\ldots,n$), we conclude from \eqref{E:B signature} and \eqref{E:Gamma=Gamma} that
\begin{multline}\notag
	B\circ(1+\Gamma_M)\ = \
	-\,(1+\Gamma_M)\circ c(e^n)\cdot
	\left(\, \sum_{i=1}^{n-1}\,c(e^i)\,\n^{E^+}_{e_i}\,\right)
	\\ = \ -\,\big(\,c(e^n)-\Gamma_\pM\,\big)\circ
	\left(\, \sum_{i=1}^{n-1}\,c(e^i)\,\n^{E^+}_{e_i}\,\right).
\end{multline}

Since the metric $g^M$ is a product near the boundary, the subbundle $\Lambda^\b(T^*\pM)\otimes V\subset E$ is equal to $E^+$ and the connection $\n^{E^+}$ is equal to the product of the Levi-Civita connection $\n^{\Lambda^\b(T^*\pM)}$ on $\Lambda^\b(T^*\pM)$ and $\n^V$.
Hence, by \cite[Proposition~3.53]{bgv},
\begin{equation}\label{E:nV+nV*=DpM}
		\sum_{i=1}^{n-1}\,c(e^i)\,\n^{E^+}_{e_i} \ = \
		\n^V\ + \ \big(\,\n^V\,\big)^*.
\end{equation}

For $\omega\in \Omega^\b(\pM,V)$ we have  $\sum_{i=1}^{n-1}\,c(e^i)\n^{E^+}_{e_i}\omega\in \Omega^\b(\pM,V)$. Hence,
\[
	\Pi\circ B\circ (1+\Gamma_M)\,\omega \ = \ \Gamma_\pM\circ
	\sum_{i=1}^{n-1}\,c(e^i)\,\n^{E^+}_{e_i}\,\omega.
\]
The lemma follows now from \eqref{E:nV+nV*=DpM} and \eqref{E:DpM=GammaD}.
\end{proof}

\subsection{The index formula for the  twisted signature operator}\label{SS:index signature}
The main purpose of this section is to prove the following special case of Theorem~\ref{T:genral APS}.

\begin{theorem}\label{T:sign=}
Let $g^M_0$ be a $G$-invariant metric on $M$ which is a product near the boundary and coincides with $g^M$ on $\pM$. Let $L_\grg(g^M)(X)$ be the equivariant $L$-form defined in \eqref{E:L form}. Let $\n^{TM}_t$ be a smooth family of $G$-equivariant connections on $TM$ such that $\n^{TM}_i$ ($i=0,1$) is the Levi-Civita connection of the metric $g^M_i$.%
\footnote{Note that we do not assume that $\n^\E_t$ is a Levi-Civita connection of some metric for $t\not=0,1$.}
Let $TL_\grg(g^M,g_0^M)(X)$ be  the corresponding equivariant transgression form of the $L_\grg$-form described in Subsection~\ref{eq-trans}.
If $X\in \grg$  is sufficiently small and has no zeros on $\pM$, then
\begin{multline}\label{E:sign=}
	\ind\big(D,\tilD_\pM\big) \ = \
	(\pi i)^{-n/2}\,
	\Big[\,\int_M L_\grg(g^M)(X)\cdot \ch_\grg(\n^\calV)(X)\\
	 - \ \int_{\pM} TL_\grg(g^M,g^M_0)(X)\cdot \ch_\grg(\n^\calV)(X)\,\Big]
	\ -\ \eta_X(D_\pM^+) \ - \ h_{e^{-X}}(D_\pM^+),
\end{multline}
where
\[
	h_{e^{-X}}(D_\pM^+)\ :=\ \Tr\big(\,e^{-X}\big|_{\ker D_\pM^+}\,\big).
\]

\end{theorem}

\begin{proof}
Consider the admissible deformation $g_t^M= (1-t)g^M_0+tg^M$ of $g^M$ (cf. Definition~\ref{D:admissible gMt}). Let $c_t:T^*M\to \End(\Lambda^\b(T^*M))$ and $\n^{\Lambda^\b(T^*M)}_t$ be the corresponding families of Clifford actions and Levi-Civita connections on $\Lambda^\b(T^*M)$. By a slight abuse of notation we also denote by $c_t$ the induced action of $T^*M$ on $\E= \Lambda^\b(T^*M)\otimes\calV$. Let $\n^\calV_t$ be a family of Hermitian connections on $\calV$ such that $\n^\calV_1= \n^\calV$, the restriction of $\n^\calV_t$ to $M\backslash{}U\sqcup\big(\pM\times(-\infty,3]\big)$ is equal to $\n^\calV$, the restriction of $\n^\calV_t$ to $\pM$ is equal to $\n^V:= \n^\calV\big|_\pM$ and the connection $\n^\calV_0$ is product near the boundary. Note that the restriction of $\n^\calV_t$ to the boundary is independent of $t$. In particular,
\begin{equation}\label{E:chV=chV0}
	\ch_\grg(\n^\calV_t)(X)\big|_\pM\ = \ \ch_\grg(\n^\calV_0)(X)\big|_\pM,
	\qquad\text{for all}\quad t\in [0,1], \ \ X\in \grg.
\end{equation}

Let
\begin{equation}\label{E:ntE=}
		\n_t^\E\ := \ \n^{\Lambda^\b(T^*M)}_t\otimes 1 \ + \ 1\otimes \n^\calV_t.
\end{equation}
This is a Clifford connection with respect to the Clifford action $c_t$.  One easily checks that the family of Dirac bundles $(\E,c_t,\n^\E_t)$ is an admissible deformation of $(\E,c,\n^\E)$ (cf. Definition~\ref{D:admissible deformation}) and
\[
		\tiln^E\ := \ \n^E_0 \ = \
		\n^{\Lambda^\b(T^*M)}_0\big|_\pM\otimes 1 \ + \ 1\otimes \n^V.
\]

Recall from Subsection~\ref{SS:properties gMa} that the metric $g^M_0$ is product near the boundary. Hence, by Lemma~\ref{L:B=DpM},
\[
	\tilD_\pM \ = \ -\,\sum_{i=1}^{n-1}\,c_0(e^n)\,c_0(e^i)\,\n^{E^+}_{0,e_i}
	\ = \ -\,\sum_{i=1}^{n-1}\,c_0(e^n)\,c_0(e^i)\,\tiln^{E^+}_{e_i}.
\]
Comparing  with \eqref{E:admissible ApM} we conclude that $\tilD_\pM$  is an admissible boundary operator defined by the boundary connection $\tiln^E$ (cf. Definition~\ref{D:admissible bc}). Therefore we can apply Theorem~\ref{T:genral APS} to compute $\ind(D_1,\tilD_\pM)$. Let $\alpha_\grg$ be as in this theorem.   By \cite[p.~150]{bgv}
\begin{equation}\label{E:L=Ac}
	 \alpha_\grg(\n^\E)(X)\ = \
	 2^{n/2}\, L_\grg(g^M)(X)\cdot \ch_\grg(\n^\calV)(X).
\end{equation}
Hence, using \eqref{E:ntE=} and \eqref{E:chV=chV0}, we conclude that on $\pM$
\begin{multline*}
	d_\grg\, T\alpha_\grg(\nab^{\EE},\n_0^{\EE})(X)\ = \
	\alpha_\grg(\nab^{\EE})(X)-\alpha_\grg(\n_0^{\EE})(X)
	\\ = \
	2^{n/2}\, L_\grg(g^M)(X)\cdot \ch_\grg(\n^\calV)(X)
	      \ - \ 2^{n/2}\, L_\grg(g^M_0)(X)\cdot \ch_\grg(\n^\calV)(X).
	\\ = \
	2^{n/2}\,\Big[\,L_\grg(g^M)(X)-L_\grg(g_0^M)\,\Big]\cdot
			\ch_\grg(\n^\calV)(X).
\end{multline*}

Notice that so far we did not use the given family $\n^{TM}_t$ of connections on $TM$. We now consider this family and use \eqref{E:transgression1} to  obtain
\[
	d_\grg T\alpha_\grg(\nab^{\EE},\n_0^{\EE})(X) \ = \
	2^{n/2}\, d_\grg\, TL_\grg(g^M,g^M_0)(X)\cdot \ch_\grg(\n^\calV)(X).
\]
The equality \eqref{E:sign=} follows now from Theorem~\ref{T:genral APS}, \eqref{E:L=Ac}, \eqref{E:D=tildeD} and \eqref{E:eta Dpm}.
\end{proof}

\subsection{The equivariant twisted signature}\label{SS:signature}
We now assume that the connection $\n^\calV$ is flat. Recall that we assume that $\calV$ is endowed with a flat $G$-equivariant Hermitian metric. For the product case Atiyah-Patodi-Singer \cite{aps2} and Donnelly \cite{don} showed that $\indl(D,\tilD_\pM)$ computes the twisted equivariant signature $\sign(l,M,\calV)$. We now use this result and Theorems~\ref{T:stability of the index} and \ref{T:sign=} to compute  $\sign(e^{-X},M,\calV)$ in the non-product case.

Let $H^\b(M,\calV)$ and $H^\b(M,\pM,\calV)$ denote the  absolute and relative cohomology of $M$ with coefficients in $\calV$. We denote by $\hat{H}^\b(M,\calV)$ the image of  $H^\b(M,\pM,\calV)$ in $H^\b(M,\calV)$. The cup product and the fiber metric on $\calV$ define a quadratic form  $Q:\hat{H}^m(M,\calV)\otimes{}\hat{H}^m(M,\calV)\to \RR$, cf. \cite[\S2]{aps2}.  We can decompose $\hat{H}^m(M,\calV)$ as
\[
	\hat{H}^m(M,\calV)\ = \ H_+\oplus H_-,
\]
where $H_+$ and $H_-$ are $Q$-orthogonal $G$-invariant subspaces such that $Q$ is positive definite on $H_+$ and negative definite on $H_-$.

\begin{Definition}\label{D:eq signature}
The {\em equivariant signature} of $\calV$ is defined by (cf. \cite{Hirzebruch71})
\begin{equation}\label{E:eq signature}
	\sign(\ell,M,\calV) \ := \ \Tr\big(\,\ell|_{H_+}\,\big) \ - \ \Tr\big(\,\ell|_{H_-}\,\big),
	\qquad \ell\in G.
\end{equation}
In the case when $\calV=\CC$ is the trivial line bundle, we set
\[
	\sign(\ell,M)\ := \ \sign(\ell,M,\CC),
\]
and refer to  it as the {\em (untwisted) equivariant signature of $M$}.
\end{Definition}

\begin{Corollary}\label{C:sign=}
In the situation of Theorem~\ref{T:sign=} assume that $\n^\calV$ is flat. Then
\begin{multline}\label{E:sign(V)=}
	\sign\big(\,e^{-X},M,\calV\,\big) \ = \
	(\pi i)^{-n/2}\,
	\Big[\,\int_M L_\grg(g^M)(X)\cdot \ch_\grg(\n^\calV)(X)\\
	 - \ \int_{\pM} TL_\grg(g^M,g^M_0)(X)\cdot \ch_\grg(\n^\calV)(X)\,\Big]
	\ -\ \eta_X(D_\pM^+).
\end{multline}
\end{Corollary}

\begin{proof}
If all the structures are product near the boundary, then
\begin{equation}\label{E:Gsgn}
	\sign(\ell,M,\calV)\ = \ \indl(D,\tilD_\pM)\ +\ h_\ell(\tilD_\pM^+),
\end{equation}
by formula~(2.4) in \cite{don} (see also \cite[Theorem~2.2]{aps2} for the non-equivariant case). By Theorem~\ref{T:stability of the index} the same formula holds in the non-product case. Equality \eqref{E:sign(V)=} follows now from Theorem~\ref{T:sign=}.
\end{proof}

\subsection{The untwisted signature operator}\label{SS:untwistedsignature}
Consider now the case when $\calV=\CC$ is the trivial line bundle. Then
\[
	D\ = \ d\ + \ d^*:\, \Omega^\b(M)\ \to \ \Omega^\b(M),
\]
and $H^\b(M,\calV)= H^\b(M)$ is the cohomology of $M$ with complex coefficients. Since the map $e^{-X}:M\to M$ is homotopic to the identity it acts trivially on the cohomology $H^\b(M)$. Hence, the equivariant signature \eqref{E:eq signature} is independent of $X$ and is equal to the ordinary signature
\[
	\sign(M)\ := \ \dim H^+ \ -  \ \dim H^-\ \in \ \ZZ.
\]	
It follows that the right hand side of \eqref{E:sign(V)=} in this case is also independent of $X$.

We now change our point of view and use \eqref{E:sign(V)=} to compute the equivariant infinitesimal eta-invariant. More precisely, let
\[
	D_\pM\ := \ \Gamma_\pM\circ d\ + \ d\circ \Gamma_\pM
\]
and let $D_\pM^+:= D_\pM\big|_{\Omega^{\even}(M)}$.

From Corollary~\ref{C:sign=} we now obtain the following
\begin{Corollary}\label{C:signuntwisted}
In the situation of Theorem~\ref{T:sign=} assume that $\calV=\CC$. Then
\begin{equation}\label{E:signuntwisted}
	\eta_X(D_\pM^+) \ = \
	(\pi i)^{-n/2}\,
	\Big[\,\int_M L_\grg(g^M)(X) \ - \
	 \int_{\pM} TL_\grg(g^M,g^M_0)(X)\,\Big] \ - \ \sign(M).
\end{equation}
\end{Corollary}

Thus, modulo the integer correction term $\sign(M)$ the equivariant infinitesimal eta-invariant $\eta_X(D_\pM^+)$ is computed as a sum of two integrals. In the rest of this paper we compute the integrands of these integrals for a class of $4$-manifolds called SKR manifolds, \cite{dm1,dm2}.

\section{The degree three component of the transgression of the equivariant $L$-form}\label{S:signature 4D}
The transgression form $TL_\grg(g^M,g^M_0)$ is a sum of forms of different geometric degrees. If the dimension of $M$ is small it is possible to give a simple formula for the the components of each degree. The purpose of this section is to obtain formula \Ref{transg} for the degree three component of $TL_\grg(g^M,g^M_0)$. In the case when $\dim{}M=4$ this is all that we need for the computation of the
second term in \eqref{E:signuntwisted}. In the next section we apply this formula to get an explicit expression for  the degree three component of $TL_\grg(g^M,g^M_0)$ on four dimensional SKR manifolds.

\subsection{Notation}\label{SS:simple notation}
To make the  explicit formulas obtained in this and the next section more readable we use a slightly shorter notation than before. We set $g=g_1 := g^M$. Let $g_0:= g^M_0$ be a $G$-invariant Riemannian metric on $M$ which is a product near the boundary and such that $g_0\big|_\pM = g\big|_\pM$.

Let $\n^1$ and $\n^0$ denote the Levi-Civita connections on $TM$ of the metrics $g_0$ and $g= g_1$, respectively. Set
\begin{equation}\label{E:form theta}
	\Theta\ := \ \n^1\ - \ \n^0,
\end{equation}
and consider the family of connections
\begin{equation}\label{E:nt}
	\n^t\ := \ \n^0\ + \ t\,\Theta \ = \ (1-t)\,\n^0\ + \ t\,\n^1,
	\qquad 0\le t\le 1.
\end{equation}
Notice that, in general, $\n^t$ is not a Levi-Civita connection of any metric. However, one easily checks that $\n^t$ is a torsion free connection for all $t$. Therefore, from Subsections \ref{SS:moment} and \ref{SS:eq curvature} we conclude that the equivariant curvature $R_\grg^t(X)$ of the connection $\n^t$  takes the form
\begin{equation}\label{E:Rt}
	R^t_\grg(X)\ = \ R^t\ - \ \nab^t\!X,
\end{equation}
where $R^t$ stands for the usual (non-equivariant) curvature of $\n^t$.

\subsection{The equivariant $L$-form}\label{SS:Lform}
Set
\begin{equation}\label{E:function f}
	\tilde{f}(x)=\frac{x/2}{\tanh(x/2)} \qquad \text{and}\qquad
	f= \frac{\log(\tilde{f})}2,
\end{equation}
and let
\[
	L_\grg(\n^t)(X) \ = \ {\det}^{1/2}\big(\,\tilde{f}(R^t_\grg(X)\,\big)
	\ = \
	\exp\left(\Tr\big[f\big(R^t_\grg(X)\big)\big]\right),
	 \qquad X\in \grg,
\]
be the $L$-form associated to the connection $\n^t$, cf. \eqref{E:L form}.

Since $\frac{d\n^t_\grg(X)}{dt}=\Theta$, it follows from \eqref{E:transgression2} that the form
\begin{equation}\label{E:transgression L}
	TL_\grg(g,g_0)\ := \ \int_0^1\,
	\exp\left(\Tr\big[f\big(R_\grg^t(X)\big)\big]\right)\cdot
	\Tr\Big[\, \Theta\,f'\big(R_\grg^t(X)\big)\,\Big]\,dt
\end{equation}
is a transgression form for $L_\grg$ in the sense that
\[
	d_\grg(TL_\grg(g,g_0)(X)\ =\ L_\grg(\nab^1)(X)\ - \ L_\grg(\nab^0)(X).
\]

\subsection{The geometric degree of an equivariant differential form}\label{SS:geom degree}
Recall that an equivariant differential form $\omega\in \Omega_G(M,\E)\subset \Omega^\bullet(M,\E)[\grg]$ is said to be of {\em geometric degree $k$} if $\omega\in \Omega^k(M,\E)[\grg]$. We note that the geometric degree of $R^t$ is 2, the geometric degree of $\Theta$ is 1, and the geometric degree of $\n^tX$ is 0.

For $\omega\in \Omega_G(M,\E)$ we denote by $\omega_{[k]}$ the component of $\omega$ in $\Omega^k(M,\E)[\grg]$ and we refer to it as the {\em degree $k$ component of $\omega$}. For application to computation of the infinitesimal $\eta$-invariant using \eqref{E:signuntwisted} we only need to know the degree 3 component $TL(g,g_0)_{[3]}$ of the transgression form $TL_\grg(g,g_0)$.

Clearly, all forms of geometric degrees greater than $\dim M$ vanish. If $\dim{}M=4$, in Corollary~\ref{C:signuntwisted} we are only interested in restrictions of $TL_\grg(g,g_0)$ to the three-dimensional manifold $\pM$. Therefore, we say that two equivariant forms $\omega_1$ and $\omega_2$ are {\em equivalent modulo forms of geometric degree $>3$} and write $\omega_1\equiv \omega_2$, if $\omega_1-\omega_2$ is a sum of forms whose geometric degrees are $\ge4$. Clearly, if $\omega_1\equiv \omega_2$ then the restrictions of $\omega_1$ and $\omega_2$ to $\pM$ coincide.

\subsection{A non-comutative analogue of the second derivative}\label{SS:noncommutative derivative}
Consider the function
\[
	H_n(a,b)\ := \ \sum_{q=0}^{n-1}a^qb a^{n-1-q}
\]
of two non-commuting variables $a$ and $b$ and set
\begin{equation}\label{E:def of star}
	f^{[2]}(a)*b\ := \ \sum_{n=0}^\infty\fr 1{n!}f^{(n+1)}(0)\,H_n(a,b)
\end{equation}
Note that if $a$ and $b$ commute then $f^{[2]}(a)*b= f''(a)b$. In this sense, the first factor in \eqref{E:def of star} can be viewed as a non-commutative analogue of the second derivative of $f$.

In what follows we often make use of expressions of the type $f^{[2]}(\pm\n^t\!X)*R^t$. Note that since the function $f$ is even, $f^{(n+1)}(0)=0$ for all even $n$. Hence,
\begin{equation}\label{E:star pm}
	f^{[2]}(-\n^t\!X)*R^t \ = \ f^{[2]}(\n^t\!X)*R^t.
\end{equation}

\subsection{Computation of $TL_\grg(g,g_0)_{[3]}$}\label{SS:TL3}
We are now ready to formulate the main result of this section - the computation of the degree 3 component of the transgression form  $TL_\grg(g,g_0)_{[3]}$. In the non-equivariant case such a computation was indicated in \cite{egh}, where the following  formula is proven
\[
	TL(g,g_0)_{[3]}\ = \ 2\,\int_0^1\Theta\we R^t\,dt.
\]
The equivariant version is notably more involved.

\begin{Proposition}\lb{4transg}
The degree 3 component of the equivariant transgression form of the Hirzebruch $L$-form
is given by
\begin{multline}\label{transg}
	TL_\grg(g,g_0)(X)_{[3]}\ = \
	\int_0^1
	\exp\left(\Tr\left[f(\nab^t\!X)\right]\right)\cdot\\
	\Big(
	  \Tr\left[\Theta f'(\nab^t\!X)\right]\cdot \Tr\left[f'(\nab^t\!X)R^t\right]
	  \ + \
	  \Tr\big[\big(f^{[2]}(\nab^t\!X)*\Theta\big) R^t\big]
	\Big)\,dt.
\end{multline}
\end{Proposition}

We present a proof of the proposition in Subsections~\ref{SS:exponential term}--\ref{SS:pr4transg}.
Here we note that the proposition holds in more generality with the same proof.
Namely, it holds whenever the characteristic form is defined via a germ of an {\em even}
analytic function $f$.

\subsection{The exponential term in $TL_\grg(g,g_0)$}\label{SS:exponential term}
We first calculate the first term in the integrand of \eqref{E:transgression L}:
\begin{equation}\label{E:Trf1}
		\exp\Big[\mathrm{Tr}\big(f(R^t_\grg(X))\big)\Big]
	\ = \ \exp\Big[\mathrm{Tr}\big(f(R^t-\nab^t\!X)\big)\Big]
	\ = \ \sum_{n=0}^\infty\fr 1{n!}\Big(\Tr\big(f(R^t-\n^t\!X)\big)\Big)^n.
\end{equation}
In the next computation we use the Taylor expression of $f$ and the fact that  $(R^t)^n\equiv 0$ modulo forms of geometric degree $\ge4$  for all $n>1$. We thus obtain
\begin{multline}\label{E:Trfn}
	\Tr\big(f(R^t-\n^t\!X)\big) \ = \
	\Tr\Big[\,
	\sum_{k=0}^\infty \frac1{k!}f^{(k)}(0)\,\big(R^t-\n^tX\big)^k\,\Big]
	\\ \equiv  \
	\Tr\Big[\,
	\sum_{k=0}^\infty \frac1{k!}f^{(k)}(0)\,\Big(\,
	(-\n^tX)^k+ k(-\n^tX)^{k-1}R^t\,\Big)\,\Big]
	\\ = \
	\Tr\Big[\,f(-\n^tX)+f'(-\n^tX)\,R^t\,\Big]
	\ = \
	\Tr\Big[\,f(\n^tX)-f'(\n^tX)\,R^t\,\Big],
\end{multline}
where for the second equality the stability of the trace of a product under cyclic permutations was applied, and in the last equality we made use of the fact that $f$ is an even function, hence $f(-\n^t\!X)= f(\n^t\!X)$ and $f'(-\n^t\!X)= -f'(\n^t\!X)$.

Combining \eqref{E:Trf1} and \eqref{E:Trfn}, we obtain
\begin{multline}\label{E:Trf2}
	\exp\Big[\mathrm{Tr}\big(f(R^t_\grg(X))\big)\Big]
	\ \equiv \
	\sum_{n=0}^\infty\fr 1{n!}\Big(\, \Tr\big[f(\n^tX)\big]-
	\Tr\big[f'(\n^tX)\,R^t\big]\,\Big)^n
	\\ \equiv \
	\sum_{n=0}^\infty\frac{1}{n!}\,\Big(\,
	\big(\Tr\big[f(\n^tX)\big]\big)^n -
	n\big(\Tr\big[f(\n^tX)\big]\big)^{n-1}\Tr\big[f'(\n^tX)R^t\big]\,\Big)	
	\\ = \
	\exp\Big(\Tr\big[f(\n^tX)\big]\Big) \ - \
	\exp\Big(\Tr\big[f(\n^tX)\big]\Big)\cdot \Tr\big[f'(\n^tX)R^t\big]
	\\ = \
	\exp\Big(\Tr\big[f(\n^tX)\big]\Big)\cdot
	\Big(1-\Tr\big[f'(\n^tX)R^t\big]\Big).
\end{multline}

\subsection{The second term in $TL_\grg(g,g_0)$}

Using the notation \eqref{E:def of star} and the fact that the geometric degree of $R^t$ is equal to 2, we obtain
\begin{multline}\label{E:f(R-)}
	f'(R^t-\n^tX)\ \equiv \
	\sum_{n=1}^\infty
	\frac{1}{n!}f^{(n+1)}(0)\,\Big((-\n^tX)^n+H_n(-\n^tX,R^t)\Big)
	\\ = \
	f'(-\n^tX)+ f^{[2]}(-\n^tX)*R^t
	\ = \ -f'(\n^tX)+ f^{[2]}(\n^tX)*R^t,
\end{multline}
where in the last equality we again made use of the evenness of the function $f$, so that, in particular, it satisfies \eqref{E:star pm}.

From \eqref{E:f(R-)} we obtain
\begin{equation}\label{E:trace term}
	\Tr\left[\Theta f'(R^t-\nab^t\!X)\right]
	\ \equiv \
	-\Tr\Big[\Theta f'(\n^tX)\Big] \ + \
	\Tr\Big[\Theta\cdot \left(f^{[2]}(\n^tX)*R^t\right)\Big],
\end{equation}
Note that in comparison with the calculation in \eqref{E:Trfn}, here the appearance of
$\Theta$ prevented us from writing a more simplified form for the second trace. We also note that the first term in the right hand side of \eqref{E:trace term} has geometric degree 1, while the second term has geometric degree 3.

Finally, the cyclic stability of the trace yield
\begin{multline}\label{E:cyclic0}
	\Tr\left[\Theta\cdot \big(f^{[2]}(\n^t\!X)*R^t\big)\right]
	\ = \
	\Tr\left[
	\sum_{n=0}^\infty\fr 1{n!}f^{(n+1)}(0)\,
	 \sum_{q=0}^{n-1}\Theta\cdot(\n^tX)^qR^t(\n^tX)^{n-1-q}
	\right]
	\\ = \
	\Tr\left[
	\sum_{n=0}^\infty\fr 1{n!}f^{(n+1)}(0)\,
		\sum_{q=0}^{n-1}(\n^tX)^{n-1-q}\Theta (\n^tX)^q\cdot R^t
	\right]
	\ = \
	\Tr\left[\left(f^{[2]}(\nab^t\!X)*\Theta\right)R^t\right].
\end{multline}

\subsection{Proof of Proposition~\ref{4transg}}\label{SS:pr4transg}
Substituting first the right hand side of \eqref{E:cyclic0} into \eqref{E:trace term}, and then the result along
with \eqref{E:Trf2}  into \eqref{E:transgression L}, and, finally,  keeping only the terms of geometric degree 3, we obtain \eqref{transg}. \hfill$\square$

\subsection{A different form of a formula for $TL_\grg(g,g_0)_{[3]}$}\label{SS:beautification}
We finish this section by presenting a slightly less explicit, but arguably more aesthetic formula for $TL_\grg(g,g_0)_{[3]}$.

\begin{Corollary}
Under the assumptions of Proposition~\ref{transg} we have
\begin{multline}\label{E:beautification}
	TL_\grg(g,g_0)(X)_{[3]} \\ = \
	\left(\int_0^1 \exp\left(\Tr\left[f(\n^t\!X)\right]\right)\cdot
	\left(1+\Tr\left[\Theta f'(\n^t\!X)\right]\right)\cdot
	\Tr\left[f'(\Theta+\n^t\!X)R^t\right]\,dt\right)_{[3]}.
\end{multline}
\end{Corollary}

\begin{proof}
As in  \eqref{E:f(R-)} we get
\begin{equation}\label{E:f(Theta-)}\notag
	f'(\Theta+\nab^t\!X)\cdot R^t
	\ \equiv \
	f'(\n^tX)R^t\ + \ \big( f^{[2]}(\n^tX)*\Theta\big)\cdot R^t.
\end{equation}
Solving this equation for $\left( f^{[2]}(\n^tX)*\Theta\right)\!R^t$ and taking the trace we thus obtain
\begin{equation}\label{E:cyclic}
		\Tr\left[\big( f^{[2]}(\n^tX)*\Theta\big)\cdot R^t\right]
		 	\\ \equiv \
	\Tr\big[f'(\Theta+\n^t\!X)R^t\big]\ - \ \Tr\big[f'(\n^t\!X)R^t\big].
\end{equation}

Note that $f'(\Theta+\n^t\!X)-f'(\n^t\!X)$ is a sum of forms of geometric degrees $\ge1$. Hence,
\[
	\Tr\big[f'(\n^t\!X)R^t\big]\cdot \Tr\big[\Theta f'(\n^t\!X)\big]
	\ \equiv \
	\Tr\big[f'(\Theta+\n^t\!X)R^t\big]\cdot \Tr\big[\Theta f'(\n^t\!X)\big].
\]
Using this equation and \eqref{E:cyclic} we conclude that
\begin{multline}\label{E:beaut2}
	\Tr\left[\Theta f'(\nab^t\!X)\right]\cdot \Tr\left[f'(\nab^t\!X)R^t\right]
	  \ + \
	  \Tr\big[\big(f^{[2]}(\nab^t\!X)*\Theta\big) R^t
	 \ \equiv \\
	 \Tr\big[f'(\Theta+\n^t\!X)R^t\big]\cdot \Tr\big[\Theta f'(\n^t\!X)\big]
	 \ + \
	 \Tr\big[f'(\Theta+\n^t\!X)R^t\big]\ - \ \Tr\big[f'(\n^t\!X)R^t\big]
	 \\ = \
	 \left(1+\Tr\left[\Theta f'(\n^t\!X)\right]\right)\cdot
	\Tr\left[f'(\Theta+\n^t\!X)R^t\right] \ - \  \Tr\big[f'(\n^t\!X)R^t\big].
\end{multline}
Substituting  \eqref{E:beaut2} into \eqref{transg} and taking into account that the degree 3 component of \linebreak $\Tr\big[f'(\n^t\!X)R^t\big]$ is equal to 0,  we obtain \eqref{E:beautification}.
\end{proof}

\section{ $L_\grg(g_0)$ and $TL_\grg(g,g_0)$ for a class of distinguished metrics in dimension four}\lb{S:examples}

It is of interest to examine whether one can draw any geometric conclusions from
formula \Ref{E:signuntwisted}. In \cite{hit}, Hitchin applied the non-equivariant
version of this formula to the case of the $3$-sphere bounding the $4$-ball. He showed
that the eta invariant is non-positive for conformal structures on the $3$-sphere induced
from conformal classes of complete self-dual Einstein metrics on the
$4$-ball. Positivity of the eta invariant was thus interpreted as an obstruction
for a conformal structure on the $3$-sphere to be induced by such a (globally defined) metric on
the $4$-ball. A key result in his calculation was the vanishing of the transgression of
the Hirzebruch $L$-form of any conformally compact Einstein metric in dimension four.

In this section we compute the integrands appearing in formula \Ref{E:signuntwisted} for
SKR metrics on disk bundles over Riemann surfaces, with boundary the zero set of
a certain Killing potential. The non-equivariant case, for these metrics, was developed
in \cite{mas}.

The class of SKR metrics
includes many K\"ahlerian conformal compactifications of four-dimensional Einstein metrics.
These have a more complicated structure than compactifications of self-dual Einstein metrics,
as in general no tensor in their curvature decomposition vanishes.
Nonetheless, we determine their curvature explicitly. Then, as there is also a circle action by isometries
on these disk bundles, we are able to compute formulas for the integrands in \Ref{E:signuntwisted}.
This allows us to obtain not just a bound, but an exact formula for the equivariant infinitesimal eta invariant.

In spite of this, our formula is difficult to apply in order to deduce the existence of
obstructions similar to Hitchin's. The main difficulty is that the formula for the equivariant
transgression is quite complex, and it is non-trivial to determine whether it
vanishes for arbitrary SKR metrics, or just for those that are conformal compactifications of
Einstein metrics. Still, we hope this formula finds its use in some future study.

Reducible SKR metrics, i.e. SKR metrics which are local K\"ahler products, form a
simpler subclass for which this diffculty does not arise. For these metrics, we show that many
of the curvature components do vanish. As a result, we will see that all terms in
\Ref{E:signuntwisted}, and hence the equivariant infinitesimal eta invariant,
in fact vanish. Thus its non-vanishing, for a conformal structure on an appropriate
circle bundle over a Riemann surface, would prevent the latter from being induced by a
reducible SKR metric on the disk bundle it bounds. A short discussion on whether
this result can be obtained by other means appears in Section~\ref{SS:vanishing}.


\subsection{Summary of basic facts on SKR metrics}

The class of SKR metrics described below was first defined for the purpose
of classifying K\"ahler conformally Einstein metrics on closed manifolds. An SKR metric
is also called, more explicitly, a metric admitting a special K\"ahler-Ricci potential.
Although our main computations are given only in dimension four, in this
subsection we will be discussing SKR metrics
in any dimenion. Our main reference for the following material will be \cite{dm1, dm2}.

\subsubsection{The definition}
Let $(M,J)$ be a complex manifold with a K\"ahler metric $g$.
A {\em Killing potential} $\t$ for $g$ is a smooth function on $M$ such that
$J\nab\t$ is a Killing vector field, infinitesimally generating
isometries of $g$. Such a potential is also a moment map for the K\"ahler form
of $g$. Note that if a Killing potential is nonconstant,
its set of regular points $M'$ is open and dense in $M$.

We use the following notation for distinguished vector
fields on $M$ and distributions on $M'$:
\be\lb{HVvu}
\text{$v:=\nab\t$,\qquad $u:=Jv$,\qquad and\quad $\VV:=\mathrm{span}\{v,u\}$,\qquad $\HH:=\VV^\perp$,}
\end{equation}
where $\HH$ denotes the distribution orthogonal to $\VV$.
\begin{Definition}\lb{SKRdef}
The Killing potential $\t$ above will be called a {\em special K\"ahler-Ricci potential},
and $g$ an {\em SKR metric}, if $\t$ is nonconstant, and at
each regular point of $\t$, the nonzero tangent vectors in
$\HH$ are eigenvectors of both the Ricci endomorphism and the Hessian of $\t$.

We refer to pairs $(g,\t)$ for which these conditions are satisfied as an {\em SKR structure}.
\end{Definition}


For the following paragraph, see Sections $5$, $7$, $10$, $11$ of  \cite{dm1}.
The condition on the Hessian in the definition of an SKR structure
implies that at each regular point of $\t$, the restriction of the
tensor $\nab d\t$ to $\cal{H}$ is a multiple of $g$. Now
on any K\"ahler manifold, the existence of a Killing potential $\t$ is
equivalent to the $J$-invariance of the Hessian $\nab d\t$.
The last two statements together yields the following consequence:
Given an SKR structure, there exists smooth functions $\phi$, $\psi$ on
the regular set of $\t$ which are ``pointwise eigenvalue functions" for
the Hessian\footnote{Similar eigenvalue functions exist also
for the Ricci tensor.} $\nab d\t$, in the sense that
\[
\text{$\nab d\t=\phi g$ on $\HH$,\qquad $\nab d\t=\psi g$ on $\VV$.}
\]
For an SKR metric, $\phi$, $\psi$ and $Q:=g(\nab\t,\nab\t)$ are, locally on the regular set $M'$ of $\t$,
``functions of $\t$", that is, obtained by composing a function defined on the image of $\t$, with $\t$.
As such, they satisfy various relations, for example
\be\lb{rels}
Q=2(\t-\bar{c})\phi,\qquad dQ=2\psi\,d\t, \qquad Qd\phi=2(\psi-\phi)\phi\,d\t,
\end{equation}
where $\bar{c}$ is a constant and the first of these equations only
holds when the function $\phi$ is nowhere zero on $M'$.
It is shown in \cite{dm1} that $\phi$ either vanishes nowhere on $M'$, or vanishes identically.
The relations \Ref{rels} are used in some of the computations in this section. While
all three functions will appear in our calculation, these relations imply that any
one of them determines the other two.

\subsubsection{Local structure}
In Sections $17$, $18$ of \cite{dm1}, it is shown that every regular point of
a special K\"ahler-Ricci potential $\t$ has a neighborhood $U$ biholomorphic to an open
set in a holomorphic line bundle over a K\"ahler\footnote{If the dimension of the manifold is greater than four,
the metric $h$ is also Einstein.} manifold $(N,h)$, with $v$ and $u$ taking the form
of certain linear fields tangent to the fibers.


Theorem $18.1$ in \cite{dm1} gives the local classification of SKR metrics.
It states that on an appropriate neighborhood $U$ as above, the biholomorphic mapping just mentioned
is an isometry of $g$ with a metric on the line bundle having the
following form, with $\VV$ now denoting the vertical distribution of the line bundle and
$\HH$ a horizontal distribution for a Chern connection associated to an appropriate hermitian fiber metric. If $\phi$ is nowhere zero on $M'$,
then
\be\lb{SKR}
\text{$g$\quad is\qquad $2|\t-\bar{c}|\pi^*h$ on $\HH$,\qquad $Q\theta^2/a^2+Q^{-1}d\t^2$ on $\VV$,}
\end{equation}
and\footnote{The form \Ref{SKR} is in fact only equivalent
to the expression appearing in \cite{dm1}.} $g(\HH,\VV)=0$. 
Here $\bar{c}$ is as above, $a\ne 0$ is constant,
$\t$ and $Q$ denote here the push-forwards of these two functions under the above biholomorphism,
$\theta/a$ is the one form dual\footnote{$\theta$ vanishes on $v$ and vector fields perpendicular to $v$ and $u$,
and has value $a$ on $u$.} to (the push-forward of) $u$ and $\pi$ is the projection map
from the total space of the line bundle to $N$. If, on the other hand, $\phi$ is identically zero on $M'$,
\Ref{SKR} will hold except that the factor $2|\t-\bar{c}|$ is replaced by $1$. Since in this latter case
the metric is a local product of the base K\"ahler metric $h$ and the fiber metric,
we will call such metrics reducible SKR metrics, whereas those of the form \Ref{SKR} will be
called irreducible.

In the following, the Levi-Civita connection of $h$ will be denoted by $D$
and its curvature by $R^h$. The K\"ahler forms of $g$ and $h$ will be denoted $\om$, $\om^h$
respectively.

The Levi-Civita connection of an SKR metric $g$ is given, for $v$, $u$ and any
$C^1$ vector fields $w$, $w'$ defined on the regular set $M'$ of $\t$, orthogonal
to and commuting with $v$ and $u$, by
\be\lb{conn}
\begin{aligned}
&\nab_vv=-\nab_uu=\psi v,\quad &&\nab_vu=\nab_uv=\psi u,\quad \nab_vw=\nab_wv=\phi w,\\
&\nab_uw=\nab_wu=\phi Jw,\quad  &&\nab_ww'=D_ww'-(\phi/Q)[g(w,w')v+\om(w,w')u],
\end{aligned}
\end{equation}
(see Sections $7$, $13$ in \cite{dm1}). In dimension greater than $2$, nonzero vector fields $w$, $w'$
of this type always exist. For example, if the above line bundle structure holds
on $M'$, then $w$, $w'$ are obtained as horizontal lifts of vector fields on $N$. That is, lifts
under the projection map $\pi$ of corresponding base vector fields, which are also sections of $\HH$.
In the last relation in \ref{conn} these base vector fields are also denoted $w$, $w'$. Note
that the lift of $Jw$ on the base is $J$ of the lift of $w$ (with $J$ denoting both the
base and the total space complex structures).
Formulas \Ref{conn} will be utilized below to compute curvature components
in a distinguished frame, necessary for the computation of both $L_\grg(g)$ and $TL_\grg(g,g_0)$.

\subsubsection{Global considerations}

The closed manifolds admitting SKR metrics in all dimensions were classified in \cite{dm2, dm3}.
They are biholomorphic to either $CP^m$, or certain $CP^1$-bundles over a K\"ahler base manifold
(which again is also Einstein if the dimension is greater than four). On these manifolds the vector field
$u$ generates a circle action by isometries.

\begin{Definition}\lb{glob}
A {\em compact SKR manifold with boundary} is a  triple $(M,g,\t)$
where M is a compact manifold with boundary and $(g,\t)$ is an SKR structure on
the interior of $M$, which extends smoothly to the boundary $\pM=\tau^{-1}(0)$,
where zero is a regular value of $\t$.
\end{Definition}
A similar classification of compact SKR manifolds with boundary was
sketched in \cite{mas}, without complete proof, described as a consequence of the
classification for closed manifolds. Such a classification shows,  in particular,
that the orbits of $u$ are still closed, and the possible manifolds with boundary
are obtained by cutting $CP^m$ or the above $CP^1$-bundles along the zero set
of $\t$, assuming it is nonempty. Thus, one expects either $\pi:M\to N$ is a closed
disc bundle over a Riemann surface, or $M$ is a closed ball centered at the origin
in a complex vector space.

For our purpose here, which is to compute the formulas for the equivariant Hirzebruch
$L$ form and its transgression, we make the following assumption explicit,
though according the sketch alluded to above, it should follow as one of the cases in the
above classification of SKR manifolds with boundary (up to a biholomorphic isometry).
A {\em fibered SKR manifold}
is a compact SKR manifold with boundary $(M, g, \t)$ (with $g$ also called a fibered SKR metric), such that the following holds.
\be\lb{global}
\begin{aligned}
&\text{On the regular set $M'$ of $\t$, $g$ is given by \Ref{SKR} or its version for reducible}\\[-5pt]
&\text{metrics, for a fibration $(M',g)\to(N,h)$, with $w$, $w'$ of \Ref{conn} defined on $M'$,}\\[-5pt]
&\text{and $u$ generating a circle action by isometries on $M$.}
\end{aligned}
\end{equation}
\begin{Remark}
For an explicit description of the SKR metrics on $CP^m$ and the above $CP^1$-bundles
in terms of data defined on these spaces, we refer to Sections $5$, $6$ in \cite{dm2}.
They include boundary conditions expressed as conditions on $Q$ at critical values of $\t$.
These will not concern us, as we will give formulas for (objects associated with) $L_\grg$,
$TL_\grg$ on the open and dense regular set $M'$, which suffice for integration\footnote{
Another fact that is implicitly used below, and was proven in \cite{dm2}, is that for a
fibered SKR space the value $\bar{c}$ is never in the image of the restriction of
$\t$ to $M'$, so that $\t-\bar{c}$ does not change sign.}.

Note that while we work mostly with metrics satisfying \Ref{SKR}, we will also indicate
how the results are modified for reducible SKR metrics.
\end{Remark}

To conform to the notation for equivariant forms in the rest of this paper, we set $X:=u$,
whereas $\grg$ will denote, from here on, the Lie algebra of the circle group.

\begin{Remark}
As shown in \cite{dm1}, any K\"ahler metric $g$ conformal to an Einstein metric on
a manifold of dimension greater than four is SKR, with $\t$ giving rise to the conformal
factor, so that the Einstein metric is given by $g_E=g/\t^2$. Such conformally Einstein SKR metrics
also exist in dimension four, and motivate our interest in SKR metrics in this dimension. On a compact
SKR manifold with boundary (M,g) satisfying \Ref{global}, with $g/\t^2$ Einstein, $g$ will thus
be a conformal compactification of the Einstein metric, while $\t$, if positive, will be a defining
function for the boundary. This is the only reason for requiring the boundary to be the zero level set
of $\t$. Proposition \ref{TLg} gives a formula for the pull-back to $\pM$ of the degree three component
of $TL_\grg$, and similar formulas hold when the boundary is any other regular level set.
\end{Remark}

\subsection{Curvature}

Curvature, as a matrix valued $2$-form, is given in an orthonormal frame
$\{e_i\}$, $i=1,\ldots, n=2m$ for the tangent bundle with dual coframe $\{e^i\}$ by
$$R_{ij}=R(e_i,e_j)=\sum_{k<l}R_{ijkl}e^k\we e^l,\qquad R_{ijkl}=\lan R(e_i,e_j)e_k,e_l\ran.$$
We now wish to describe this matrix for an SKR metric on a $4$-manifold, with respect to a distinguished frame.
The full curvature tensor $R$ of an SKR metric has not been computed previously, as
only its Ricci tensor was studied.

\subsubsection{The curvature matrix}

We have,
\begin{Proposition}\lb{curv-mtx}
Let $(M,g,\t)$ be a fibered SKR $4$-manifold. Choose an orthonormal frame $\{e_i\}$
so that $e_1=w/|w|$ is the normalization of a horizontal lift $w$, $e_{2}=Je_{1}$, while $e_{3}=u/\sqrt{Q}$ and $e_{4}=-v/\sqrt{Q}$. Setting $e^{ij}:=e^i\we e^j$, the curvature matrix in this frame has the form
\begin{equation}\lb{curv-mx}
[R_{ij}]=\begin{bmatrix}
         0          & \ be^{12}+ce^{34}  & \ r(e^{13}+e^{24})   & r(e^{14}-e^{23}) \\[4pt]
-(be^{12}+ce^{34})  &        0           &  -r(e^{14}-e^{23})   & r(e^{13}+e^{24}) \\[4pt]
-r(e^{13}+e^{24})   & \ r(e^{14}-e^{23}) &         0            & ce^{12}+de^{34}  \\[4pt]
-r(e^{14}-e^{23})   & -r(e^{13}+e^{24}) &  -(ce^{12}+de^{34})  &         0
\end{bmatrix},
\end{equation}
for certain functions\footnote{note that the function $c$ has nothing to do with the constant $\bar{c}$
 previously mentioned.} $b$, $c$, $d$, $r$.
\end{Proposition}

We give two proofs of this result. The first is more algebraic and conceptual, given in terms of the curvature operator,
while the second is more computational, involving individual curvature components. We include the latter since it gives the functions $b$, $c$, $d$, $r$ directly as curvature components, given in terms of $\phi$, $\psi$ and $Q$ and the curvature of $h$. Our computations in relation to the transgression of the equivariant $L$-form refer to particular curvature components.

\subsubsection{First proof}

The first proof is as follows.\vspace{5pt}

\noindent
{\em Proof I.} let
\[
\alpha_1=e^{12}+e^{34},\quad \alpha_2=e^{13}-e^{24},\quad \alpha_3=e^{14}+e^{23},
\]
be the induced (constant length, orthogonal) basis for the self-dual $2$-forms,
and $\beta_i$, $i=1,2,3$ the corresponding basis for the anti-self-dual
$2$-forms, obtained from $\alpha_i$ by changing the sign of the first summand.

Consider the curvature operator $\calR:\Lambda^2(T^*M)\to \Lambda^2(T^*M)$ of $g$,
which is given, with respect to the decomposition into self- and anti-self-dual $2$-forms,
by \[\calR=\begin{bmatrix} W^++\fr s{12} & B\\ C & W^-+\fr s{12}\end{bmatrix},\] where $W^\pm$ are the
anti/self-dual Weyl tensors, $s$ is the scalar curvature and $B$, $C$ will be defined below via the
traceless Ricci tensor.

Now as $g$ is K\"ahler (on a $4$-manifold), $W^+$ is diagonal with respect to the $\al_i$,
where $\al_1$ is, up to normalization, the K\"ahler form,
with eigenvalues $\{s/6,-s/12, -s/12\}$ (see for example Proposition 9.8 of \cite{der}).
Thus $W^++s/12$ has eigenvalues $\{s/4, 0, 0\}$.
We will write f:=s/4.

SKR metrics exhibit the phenomena of duality (see Proposition 5.1 of \cite{mas1}),
so that aside from $g$, the metric $\tilde{g}=g/(\t-\bar{c})^2$ is also
SKR, and in particular K\"ahler with respect to a complex structure compatible
with the opposite orientation. Thus $\tilde{W}^+$ has a similar eigenvalue
structure with two equal eigenvalues. But the conformal covariance of the
anti/self-dual Weyl tensors (considered as operators on $2$-forms), together
with the fact that the complex structures associated to $g$ and $\tilde{g}$
determine opposite orientations, yield the relation $\tilde{W}^+=2(\t-\bar{c})^2W^-$.
Thus $W^-$, and, more importantly, $W^-+s/12$, has its eigenvalues in the form $\{\tilde{f},z, z\}$,
where $\tilde{f}$ and $z$ are determined by the scalar curvatures of $g$,
$\tilde{g}$ and the conformal factor $(\t-\bar{c})^2$.

Let $\ri_0$ be the traceless Ricci tensor of $g$, also regarded as a self-adjoint
traceless endomorphism of the cotangent bundle induced from the traceless Ricci
endomorphism. Define an operator on $2$-forms by
$\ri_0(e^{ij})=\frac 12(\ri_0(e^i)\wedge e^j+e^i\wedge \ri_0(e^j))$.
Letting $\Lambda^\pm(T^*M)$ denote the anti/self-dual $2$-forms,
this operator induces $B:\Lambda^-(T^*M)\to\Lambda^+(T^*M)$ by linearity,
as well as $C$, which is the adjoint of $B$ (see Proposition 3.2.2.1 of \cite{sch}).

Directly from the definition of an SKR metric, the Ricci endomorphism,
and hence the traceless Ricci endomorphism,  has
two eigenvalues, denoted $\lambda_0$, with eigendistribution $\mathrm{span}(e^1, e^2)$ dual
to $\HH$, and $-\lambda_0$, with eigendistribution $\mathrm{span}(e^3, e^4)$ dual to $\VV$.

Thus $\ri_0$, as an operator on $2$-forms, sends $e^{12}$ to $\lambda_0e^{12}$,
$e^{34}$ to $-\lambda_0e^{34}$, while $e^{13}$, $e^{24}$, $e^{14}$ and $e^{23}$
are all sent to zero (since, for example,
$\frac 12(\lambda_0e^1\wedge e^3+e^1\wedge (-\lambda_0e^3))=0$).
Therefore $B\beta_1=-\lambda_0\alpha_1$, while $B$ sends
the other $\beta_i$'s to $0$, and $C$ acts in an analogous manner
on the $\alpha_i$'s.

With these preliminaries, we can now compute the action of the curvature
operator on the $e^{ij}$ to obtain the curvature matrix entries. First,
\[
\begin{aligned}
\calR(2e^{12})&=\calR(\alpha_1-\beta_1)=(W^++s/12)(\alpha_1)-(W^-+s/12)(\beta_1)+C(\alpha_1)-B(\beta_1)\\
&=f\alpha_1-\tilde{f}\beta_1-\lambda_0\beta_1+\lambda_0\alpha_1
=(f+\tilde{f}+2\lambda_0)e^{12}+(f-\tilde{f})e^{34},\\
\calR(2e^{34})&=\calR(\alpha_1+\beta_1)=(W^++s/12)(\alpha_1)+(W^-+s/12)(\beta_1)+C(\alpha_1)+B(\beta_1)\\
&=f\alpha_1+\tilde{f}\beta_1-\lambda_0\beta_1-\lambda_0\alpha_1
=(f-\tilde{f})e^{12}+(f+\tilde{f}-2\lambda_0)e^{34},\\
\end{aligned}
\]
which is in line with \Ref{curv-mx}, with $b=(f+\tilde{f}+2\lambda_0)/2$, $c=(f-\tilde{f})/2$ and
$d=(f+\tilde{f}-2\lambda_0)/2$. Similarly, we compute the other matrix coefficients, in which
many more zeros figure in, giving the simpler form of the upper right quarter of \Ref{curv-mx}:
\[
\begin{aligned}
\calR(2e^{13})&=\calR(\alpha_2-\beta_2)=(W^++s/12)(\alpha_2)-(W^-+s/12)(\beta_2)+C(\alpha_2)-B(\beta_2)\\
&=0-z\beta_2+0-0=z(e^{13}+e^{24}),\\
\calR(2e^{24})&=\calR(-(\alpha_2+\beta_2))=-((W^++s/12)(\alpha_2)+(W^-+s/12)(\beta_2)+C(\alpha_2)+B(\beta_2))\\
&=-(0+z\beta_2+0+0)=-(-z(e^{13}+e^{24}))=z(e^{13}+e^{24})\\
\calR(2e^{14})&=\calR(\alpha_3-\beta_3)=(W^++s/12)(\alpha_3)-(W^-+s/12)(\beta_3)+C(\alpha_3)-B(\beta_3)\\
&=0-z\beta_3+0-0=z(e^{14}-e^{23})\\
\calR(2e^{23})&=\calR(\alpha_3+\beta_3)=(W^++s/12)(\alpha_3)+(W^-+s/12)(\beta_3)+C(\alpha_3)+B(\beta_3)\\
&=0+z\beta_3+0+0=-z(e^{14}-e^{23})\\
\end{aligned}
\]
This fits \Ref{curv-mx} with $r=z/2$. As the curvature matrix is antisymmetric for an orthonormal
basis, this concludes the first proof.

\subsubsection{Curvature components}
The second proof depends on the following lemma, which holds in any dimension, and
relations \Ref{curv-ind} below.
\begin{Lemma}\lb{curv-formu}
Let $g$ be a fibered SKR metric, with $v$, $u$, $w$, $w'$ be as in \eqref{conn}.
Then, on the regular set of $\t$,
\be\lb{curv}
\begin{gathered}
\begin{aligned}
&\lan R(v,u)v,u\ran=-\psi'Q^2,\quad &&\lan R(w,v)w',v\ran=\lan R(w,u)w',u\ran=-\phi'Qg(w,w')/2,\\
&\lan R(w,Jw)u,v\ran=\phi'Qg(w,w),\quad  &&\lan R(w,v)w',u\ran=-\phi'Q\om(w,w')/2,\\
&\lan R(w,Jw)Jw',w'\ran=&&\\[4pt]
\end{aligned}
\end{gathered}
\end{equation}
\vspace{-18.333pt}
\begin{equation*}
\begin{aligned}
\qquad\left|Q/\phi\right|h(R^h(w,Jw)Jw',w')+2(\phi^2/Q)\left(g(w,w')^2+\om(w,w')^2+g(w,w)g(w',w')\right),
\end{aligned}
\end{equation*}
with angle brackets standing for $g$. Additionally, $\lan R(a,b)c,d\ran=0$ when either exactly three of the vector
fields $a,b,c,d$ are taken from the pair $\{w,Jw\}$, for a horizontal lift $w$, or exactly three of them are from the pair $\{v, u\}$.
\end{Lemma}
In the four dimensional case, all nonzero curvature components for the orthonormal frame of Proposition \ref{curv-mtx}
are determined by those above via the curvature symmetries of a K\"ahler metric. Note
also that the third relation above, given just for convenience in checking relations \Ref{curv-ind} below,
is a consequence of the second and fourth via the first Bianchi identity.

The proof of Lemma \ref{curv-formu} is a tedious but straightforward computation based on \Ref{conn} and \Ref{SKR}.
It also employs some known properties of SKR metrics, such as relations \Ref{rels} as well as $[v,u]=0$, and that $Jw$ is a horizontal lift of a base vector field if $w$ is. As the curvature vanishing statements at the end of the lemma are important, we prove only some cases of those.
\begin{proof}
$\lan R(w,u)u,v\ran=\lan([\n_u,\n_w]+\n_{[w,u]})u,v\ran$ and $[w,u]=0$ by \Ref{conn}, which
also shows that $[\n_u,\n_w]u=\n_u(\phi Jw)-\n_w(-\psi v)=(d_u\phi)Jw+\phi(-\phi w)+\psi\phi w$,
because the derivatives of $\phi$, $\psi$ in horizontal directions vanish. Our curvature component
vanishes as all three summands in the last expression are orthogonal to $v$.

Next, $\lan R(w,Jw)w,v\ran=\lan([\n_{Jw},\n_w]+\n_{[w,Jw]})w,v\ran$, and here we compute each term
separately using \Ref{conn}. $\n_{Jw}\n_ww=\n_{Jw}(D_ww-(\phi/Q)g(w,w)v)=D_{Jw}D_ww-(\phi/Q)(g(Jw,D_ww)v+\om(Jw,D_ww)u)
-(\phi/Q)(d_{Jw}(g(w,w))v+g(w,w)\phi Jw)$. Notice that only the two terms with $v$ are not orthogonal to $v$.
Similarly, the only term of $\n_w\n_{Jw}w$ not orthogonal to $v$ is the term $-(\phi/Q)(g(w,D_{Jw}w)v$.
Now
\[
\n_{[Jw,w]}w=\n_{\n_{Jw}w-\n_wJw}w=\n_{D_{Jw}w-(\phi/Q)\om(Jw,w)u-D_{w}Jw-(\phi/Q)\om(w,Jw)u}w
\]
which breaks up into eight terms, with only two not orthogonal to $v$: $(\phi/Q)g(D_{Jw}w,w)v-(\phi/Q)g(D_w{Jw},w)v$.
Of the five terms not orthogonal to $v$ mentioned above, the first and fifth cancel since $h$ is K\"ahler so that $J$
commutes with $D_w$, and since $g$ is hermitian. The second, third and fourth term cancel since
$d_{Jw}(g(w,w))=2(\t-\bar{c})d_{Jw}h(w,w)$ and $d_{Jw}h(w,w)=2h(D_{Jw}w,w)$.
\end{proof}

It now follows that relative to the orthonormal frame of Proposition \ref{curv-mtx}, we have, directly from \eqref{curv}
and the curvature symmetries
\be\lb{curv-ind}
\begin{aligned}
&b:=R_{1212}=-|\phi/Q|R^h_{1212}-4\phi^2/Q,\quad c:=R_{1234}=-\phi',\quad d:=R_{3434}=-\psi',\\[4pt]
&r:=R_{1313}=R_{1324}=-R_{2314}=R_{2424}=R_{1414}=R_{2323}=-\phi'/2.
\end{aligned}
\end{equation}
In the reducible case, formulas \Ref{curv-ind} are valid except that in the first term of $b$ the factor
$|\phi/Q|$ does not appear. Additionally, as $\phi$ vanishes identically in the reducible case, $c=r=0$.

\subsubsection{The second proof}

We now give the second proof.\vspace{5pt}

\noindent
{\em Proof II.} The second proof of Proposition \ref{curv-mtx} now follows from \Ref{curv-ind} along with Lemma
\ref{curv-formu}. For example, to compute $R_{13}$ from its definition, one has to sum terms of the form $R_{13kl}e^{kl}$
with $k<l$. There are $6$ such terms. Those with coefficients $R_{1312}$, $R_{1334}$ vanish by the last clause of this lemma.
Those with coefficient $R_{1314}$, $R_{1323}$ vanish by the fourth (with $w=w'$) and second (with $w'=Jw$) relations in \Ref{curv}. The remaining two terms are exactly $re^{12}$ and $re^{34}$, by the last line of \Ref{curv-ind}.

\subsection{Equivariant curvature}
We now compute the equivariant curvature of a fibered SKR metric.

Let $J_\cal{H}$, $J_\cal{V}$ be the restriction of $J$ to the $J$-invariant distributions $\cal{H}$, $\cal{V}$.
Recall that $X:=u$, the generator of our circle action.
We have $\nab X=\nab u=J\nab v=\phi J_\cal{H}+\psi J_\cal{V}$, as $\nab v$ is the operator corresponding to the Hessian
of $\t$, whose eigenvalues are $\phi$ and $\psi$. Thus
\be\lb{nab-X}
[(\nab X)_{ij}]=\begin{bmatrix}
  0   & \phi &   0   &    0  \\[4pt]
-\phi &   0  &   0   &    0  \\[4pt]
  0   &   0  &   0   &  \psi \\[4pt]
  0   &   0  & -\psi &    0
\end{bmatrix}.
\end{equation}

Setting $\aA=\phi+ be^{12}+ce^{34}$, $\fF=\psi+ce^{12}+de^{34}$,
$\sS=e^{13}+e^{24}$, $\tT=e^{14}-e^{23}$, it follows from \Ref{curv-mx}
and \Ref{nab-X} that
the equivariant curvature matrix takes the form
\[
[R_\grg(X)_{ij}]=[(R-\nab X)_{ij}]=\begin{bmatrix}
  0    &  \aA  & r\sS  & r\tT \\[4pt]
-\aA   &   0   & -r\tT & r\sS \\[4pt]
-r\sS  &  r\tT &   0   &  \fF \\[4pt]
-r\tT  & -r\sS &  -\fF &   0
\end{bmatrix}.
\]

\subsection{Equivariant $L$-form}\lb{eq-L}
In this subsection we compute the equivariant $L$-form $L_\grg(g)$ of a
fibered SKR metric on a $4$-manifold, following a method appearing in \cite{goe}.

First, we compute the eigenvalues of $R_\grg(X)$, taking into account the vanishing, in the
characteristic polynomial, of terms of degree greater than four.
The eigenvalues are given by
\be\lb{eigen}
\begin{gathered}
\{\lam_i\}=\{0,0,i\sqrt{A},-i\sqrt{A}\}\qquad \text{for}\\[4pt]
A=2r^2(\sS^2+\tT^2)+(\aA^2+\fF^2).
\end{gathered}
\end{equation}
The square root of $A$ can be determined explicitly as follows.
Let $e^{ijkl}=e^i\we e^j\we e^k\we e^l$.
Then,
\be\lb{square}(\al+\beta e^{12}+\gamma e^{34}+\delta e^{1234})^2=
\al^2+2\al\beta e^{12}+2\al\gamma e^{34}+2(\al\delta+\beta\gamma)e^{1234}.
\end{equation}
Now $\sS^2+\tT^2=-4e^{1234}$,
so that
\[
A=\phi^2+\psi^2+2(\phi b+\psi c)e^{12}+2(\phi c+\psi d)e^{34}+2(bc+cd-4r^2)e^{1234}.
\]
By equating the coefficients of $A$ with the right hand side of \Ref{square}, and
solving from left to right for $\al$, $\beta$, $\gamma$, $\delta$, we see that
the coefficients of $\sqrt{A}$ are given by
\be\lb{abgd}
\begin{aligned}
&\al=\sqrt{\phi^2+\psi^2},\quad \beta=\frac{b\phi+c\psi}{\sqrt{\phi^2+\psi^2}},\quad \gamma=\frac{c\phi+d\psi}{\sqrt{\phi^2+\psi^2}},\\[4pt]
&\delta=\frac{(cd-4r^2)\phi^2+(bc-4r^2)\psi^2-\phi\psi(bd+c^2)}{(\phi^2+\psi^2)^{3/2}}.
\end{aligned}
\end{equation}

Recall now the $L$-function $\tilde{f}(x)=x/(2\tanh(x/2))$. Let $\bar{f}(x)=x/(2\tan(x/2))$.
We have shown
\begin{Proposition}\lb{P:LXM}
The equivariant $L$-form of a fibered SKR metric $g$ on a $4$-manifold $M$ is given by
\be\lb{LXM}
L_\grg(g)(X)=\mathrm{det}^{1/2}\tilde{f}(R_\grg(X))=
\bar{f}(\sqrt{A}).
\end{equation}
\end{Proposition}
More explicitly, we have
\begin{equation*}
\begin{aligned}
\bar{f}\big(\sqrt{A}\big)&=\bar{f}(\al+\beta e^{12}+\gamma e^{34}+\delta e^{1234})\\
&=\bar{f}(\al)+\bar{f}'(\al)(\beta e^{12}+\gamma e^{34}+\delta e^{1234})+
\bar{f}''(\al)(\beta e^{12}+\gamma e^{34}+\delta e^{1234})^2/2\\
&=\bar{f}(\al)+\bar{f}'(\al)(\beta e^{12}+\gamma e^{34}+\delta e^{1234})
+\bar{f}''(\al)\beta\gamma e^{1234},
\end{aligned}
\end{equation*}
so that
\begin{Corollary}\lb{L-expl}
The explicit form of the degree four component of the equivariant $L$-form of a fibered SKR metric $g$ on a $4$-manifold $M$ with special K\"ahler-Ricci potential $\t$ is
\begin{equation}\lb{L-explicit}
\begin{aligned}
L_\grg(g)(X)_{[4]}=\left(\bar{f}'(\al)\delta+\bar{f}''(\al)\beta\gamma\right)e^{1234},
\end{aligned}
\end{equation}
where $\bar{f}$ is as above and $\al$, $\beta$, $\gamma$, $\delta$ are given in \Ref{abgd},
with $\phi$, $\psi$ the eigenfunctions of $\nab d\t$,
and $r$, $b$, $c$, $d$ the curvature components given by \eqref{curv-ind} (or its version
for a reducible SKR metric),
which ultimately depend on a single function of $\t$ and the base curvature $R^h_{1212}$.
\end{Corollary}
Since $\t$ is a moment map for the K\"ahler form of $g$, in the case where $M$ is a disc bundle,
integration of \Ref{L-explicit} can be reduced to integration over the base Riemann surface via
the Duisteramaat-Heckman theorem, applied to a compact manifold obtained using Lerman's symplectic
cut along $\tau^{-1}(0)$. A variant of this procedure was used \cite{mas}.

Note that the above derivation also proves the general characteristic form
formula $\mathrm{det}^{1/2}\tilde{f}(R_X)=\tilde{f}(0)\tilde{f}(i\sqrt{A})$ for a fibered SKR metric
on a $4$-manifold, whenever $\tilde{f}$ is the germ of an even real-analytic function.

Finally, we note that for reducible fibered SKR metrics, even though $A$ is nontrivial,
$L_\grg(g)(X)_{[4]}=0$. This follows from the corollary, because we have seen that in
this case $\phi$, $c$ and $r$ are zero and hence so are $\beta$ and $\delta$.

\subsection{The formula for the equivariant $L$-form's transgression of a fibered SKR metric}

Before giving our main formula concerning the transgression of the $L$-form on a fibered SKR $4$-manifold,
we note the following general remark.
\begin{Remark}\lb{gp-exp}
Suppose a Lie group $G$ acts by isometries on a manifold with boundary $(M,g)$, and let $\mathbf{n}$ be a 
unit vector field normal to the boundary. Then the normal geodesic flow of $\mathbf{n}$ defines a $G$-equivariant collar.
In fact, the collar is given by $\psi:(y,t)\to \exp_y(t\mathbf{n})$, $y\in\pM$ with $\exp$
the exponential map of $g$. As the group action preserves the boundary, the normal is $G$-invariant.
The $G$-equivariance of $\psi$ with respect to the action \Ref{E:action collar} follows directly from the formula
\be\lb{iso-exp}
	F\circ\exp_y(\mathbf{n})=\exp_{F(y)}\circ DF_y(\mathbf{n}),
\end{equation}
for an isometry $F$ (see \cite{pet}, Section 10.1).

Now let $\mathrm{Exp}(x,w)=\exp_xw$ be the induced map on a neighborhood
of the zero section in $TM$. Restricting this map to the normal bundle, equipped with the norm
induced by the metric, the Gauss Lemma asserts that the $\mathrm{Exp}$ images of constant norm (i.e. constant $t$
for sufficiently small $t$) level sets are orthogonal to the normal geodesics.
If $g$ is SKR on a $4$-manifold with $U(1)$ generator $X$, as $X$ is orthogonal
to $e_4$ even off the boundary,
the group action preserves such a level set, and hence $e_4$, near the boundary.
Formula \Ref{iso-exp}, with $\mathbf{n}=e_4$, applied to a one parameter family of isometries in the flow of
$X$ shows that orbits of $X$ are mapped to other such orbits by the geodesic flow of $e_4$, and hence
$X$ itself is also preserved by this geodesic flow. This fact will be used in Subsection \ref{nabtX}.
\end{Remark}

\begin{Proposition}\lb{TLg}
Let $(M,g,\t)$ be a fibered SKR $4$-manifold with boundary, with $v=\n\t$ outward-pointing
along the boundary and a $U(1)$-action generated by $X=u:=Jv$. Set $U$ to be the $U(1)$-equivariant collar of $M$
obtained using the normal geodesic flow of $-v/|v|$. Suppose $g_0$ is a metric constructed by arbitrarily extending to $M$
the product metric on $U$, whose factors are the restriction of $g$ to $\pM$ and the standard metric on the real line.
Let also $\n^t$  be the linear family of connections \Ref{E:nt} connecting the Levi-Civita connections of $g$ and $g_0$. Then
the pull-back under the boundary inclusion $\imath:\pM\to  M$ of the degree $3$ component of the transgression of the equivariant $L$-form of $g$ for the family $\n^t$, is given in the irreducible case by
\be\lb{final-trns}
\begin{aligned}
\imath^*&TL_\grg(g,g_0)(X)_{[3]}=\bigintsss_0^1
\bigg[\exp\left(2\left(f(i\phi_0)
+f(it\psi_0)\right)\right)\cdot\\
&\cdot\bigg[4\fr{\psi_0}{\sqrt{Q}_0}f'(it\psi_0)\bigg(\big(t^2\fr{\phi_0^2}{Q_0}-R^0_{1212}\big)f'(i\phi_0)
-tR_{1234}f'(it\psi_0)
\bigg)\\
&+2\fr{\phi_0}{\sqrt{Q}_0}\mathlarger{\mathlarger{\sum}}_{m=0}^\infty \bigg(\fr{f^{(2m+2)}(0)}{(2m+1)!}
\bigg(\left((-1)^{m}R^0_{2323}
+(-1)^{m-1}t^2\fr{\phi_0\psi_0}{Q_0}
\right)
\!\mathlarger{\mathlarger{\sum}}_{\mathrm{odd}\ k=1}^{2m-1}\!\!
M_{k,m}(\phi_0,t\psi_0)\\
&+(-1)^{m+1}tR_{2314}\!\!\mathlarger{\mathlarger{\sum}}_{\mathrm{even}\ k=0}^{2m}
\!\!M_{k,m}(\phi_0,t\psi_0)\bigg)\bigg)
-2t\fr{\psi_0}{\sqrt{Q}_0}R_{1234}f''(it\psi_0)
\bigg]e^{123}\bigg]\,dt.
\end{aligned}
\end{equation}
Here $f=\log(x/(2\tanh(x/2)))/2$, $\phi_0$, $\psi_0$ are the eigenfunctions of the
special K\"ahler-Ricci potential $\t$ evaluated at $\t=0$,
$M_{k,m}(\phi_0,t\psi_0):=\phi_0^k(t\psi_0)^{2m-k}+(t\psi_0)^k\phi_0^{2m-k}$, $Q_0=g(\n\t,\n\t)|_{\t=0}$,
$R_{abcd}$, $R^0_{abcd}$ are the curvature components of $g$, $g_0$, respectively,
in the frame $\{e_i\}$ of Proposition \ref{curv-mtx}, also evaluated at $\t=0$, and
$e^{123}=e^1\we e^2\we e^3$ for the dual coframe $\{e^i\}$.

Explicitly, these curvature components are given by
$R_{1234}=-\phi'$, $R_{2314}=\phi'/2$, along with
$R^0_{1212}=2|\bar{c}|R^h_{1212}+3Q_0/(4\bar{c}^2)$
and $R^0_{2323}=-Q_0/(4\bar{c}^2)$.

On the other hand, in the case where $g$ is reducible, we have
\[ \imath^*TL_\grg(g,g_0)_{[3]}(X)=0.
\]
\end{Proposition}

The curvature values of the components $R_{abcd}$ which appear in formula \Ref{final-trns}
are taken from \Ref{curv-ind}. For the curvature components $R^0_{abcd}$ appearing
in this formula, which are just curvature components of the restriction of $g$ to $\pM$,
a standard metric on a circle bundle over a $2$-manifold, we use Lemma $3.2$ of \cite{wat},
along with one more fact about SKR metrics. Namely, the vertical component of the
Lie bracket of two horizontal vector fields $w$, $w'$ is $\mp 2\om^{(h)}(w,w')u$,
which implies that the curvature of the connection form $\theta$ of the circle bundle is
$\pm a\om^h$. This lemma also shows that $R^0_{1313}=R^0_{2323}$, a fact that will be used
later to derive \Ref{k-odA1}, which is part of the computation giving the term involving this
curvature component in \Ref{final-trns}.

Note that  in cases where $R^h_{1212}$ is constant (such as the conformally Einstein
case), the coefficient of $e^{123}$ in the integrand  of \Ref{final-trns} varies only with $t$,
so that after $t$-integration it is constant along the boundary. This is because it is
composed out of functions of $\tau$ evaluated at $\tau=0$.

We prove this proposition in Subsections \ref{prel}-\ref{end}, where, after some preliminaries,
we define and compute the main ingredients of the transgression form, and then use them to
calculate the terms in formula \Ref{transg}.

\subsection{The endomorphism-valued $1$-form $\Theta$}\lb{prel}

In the next three subsections we will give expressions for the ingredients of the equivariant
transgression formula \Ref{transg}, namely $\Theta$, $\n^t\!X$ and $R^t$, or rather their pull-back to $\pM$,
in the case where $(g,\t)$ is a fibered SKR structure on a four manifold with boundary.
We begin with $\Theta$.

\medskip

\subsubsection{The connection $1$-form matrix of $g$}

We work in the orthonormal frame given by $\{e_i\}=\{w/|w|,Jw/|w|,u/\sqrt{Q},-v\sqrt{Q}\}$,
defined in Proposition \ref{curv-mtx} on the regular set of $\t$, along with the dual coframe $\{e^i\}$.
In this frame, the endomorphism-valued one-form $\Theta:=\n^1-\n^0$ is given by the matrix of one forms
\[\Theta_{ij}=\nu_{ij}-\nu^0_{ij},\] where $\nu_{ij}$ is the connection 1-form matrix
of $g$, and $\nu^0_{ij}$ is the connection 1-form matrix of $g_0$. The latter metric is, near $\pM$,
a product of the restriction $g^\pM$ of $g$ to the boundary and a standard metric on an interval.
The $\nu_{ij}$ are obtained directly from relations \eqref{conn} via
\[
\nu_{ij}=\sum_k\nu_{ij}(e_k)e^k,\qquad \nu_{ij}(e_k)=g(\nab_{e_k}e_i,e_j).
\]
Calculating this gives, with
\be\lb{kl}
k=\phi/\sqrt{Q}, \quad \ell=\psi/\sqrt{Q}
\end{equation}
and $\mathbf{f}:=g(D_ww,Jw)e^1+g(D_{Jw}w,Jw)e^2+ke^3$,
\be\lb{nu}
[\nu_{ij}]=\begin{bmatrix}
      0          &   \mathbf{f} &   ke^2   &  ke^1  \\[4pt]
   -\mathbf{f}   &    0         &   -ke^1  &  ke^2  \\[4pt]
   -ke^2         &    ke^1      &   0      &  \ell e^3 \\[4pt]
   -ke^1         &   -ke^2      &  -\ell e^3   &    0
\end{bmatrix}.
\end{equation}

\medskip

\subsubsection{The pull-back of $\Theta$}Let $\imath:\pM\to M$ be the boundary embedding.
To obtain the entries of $\imath^*\Theta_{ij}$, one need not calculate the connection
one-form matrix of $g_0$. Instead, note first that
\[
\sum_j\Theta_{ij}(a)e_j=\Theta(a)e_i=\n_a e_i-\n^0_a e_i=\n_a e_i-\n^\pM_a e_i=\Pi(a,e_i)e_4,
\quad i < 4, \quad a\in T\pM,
\]
where $\n^\pM$ is the Levi-Civita connection of $g^\pM$, and $\Pi$ is the second fundamental form of the boundary.
It follows that for $i,j\le 3$, $\imath^*\Theta_{ij}=0$.

Second, let $\{e_i^0\}$ be a $g_0$-orthonormal
frame\footnote{Note the distinct meanings of $e^j_i$, $j\geq 1$ and $e^0_i$.}
coinciding with $\{e_i\}$ on $\pM$. The
entries $\nu^0_{ij}=g_0(\n^0_{e_k}e^0_i,e^0_j)e^k$ form an antisymmetric matrix.
Since the metric $g_0$ is product near the boundary, it follows that
$\imath^*\nu^0_{ij}$ must vanish if either $i$ or $j$ is $4$. We thus see from \Ref{nu} that
\be\lb{thet}[(\imath^*\Theta)_{ij}]=\begin{bmatrix}
      0   &   0     &   0   & \pP  \\[4pt]
      0   &   0     &   0   & \qQ  \\[4pt]
      0   &   0     &   0   & \rR \\[4pt]
   -\pP   &   -\qQ  &  -\rR &    0
\end{bmatrix},
\end{equation}
where
\be\lb{pqr}
\text{$\pP=ke^1$, $\qQ=ke^2$ and $\rR=\ell e^3$, all evaluated at
$\t=0$.}
\end{equation}
Equivalently, denoting $e^i_j:=e^i\otimes e_j$
\be\lb{ithet}
\imath^*\Theta=\sum_{i=1}^2ke^i\otimes(e^i_4-e^4_i)+\ell e^3\otimes(e^3_4-e^4_3).
\end{equation}

\subsection{The endomorphism $\nab^t\!X$}\lb{nabtX}

We now derive an expression for the $t$-dependent endomorphism-valued function $\n^t\!X$,
for the linear family $\n^t$ given by \Ref{E:nt}, associated to a fibered SKR metric $g$
with $X$ the vector field corresponding to an infinitesimal $U(1)$-generator.


\subsubsection{The endomorphism $\n^0\!X$}

On the boundary $\pM=\{\t=0\}$, the vector field $X=u$ has constant length with respect
to $g^\pM$, as $Q=g(X,X)$ is a function of $\t$ (see lines before \Ref{rels}).
As $g_0$ is product near the boundary, and the normal geodesic flow preserves $X$ (see Remark \ref{gp-exp}),
it follows that $X$ also has constant $g_0$-length near the
boundary. Namely, its length is $\sqrt{Q_0}$ where $Q_0=Q|_{\{\t=0\}}$. Therefore,
one can choose one of the vector fields in the $g_0$-orthonormal frame mentioned in the previous subsection
(and defined near the boundary) to be $e_3^0=X/\sqrt{Q_0}$. We thus compute
\begin{equation}\lb{nab0X}
\n^0\!X=\n^0(\sqrt{Q_0}e_3^0)=\sqrt{Q_0}\,\nu^0_{3j}\otimes e^0_j=\sqrt{Q_0}\,(\nu^0_{31}\otimes e^0_1+\nu^0_{32}\otimes e^0_2),
\end{equation}
where the last equality follows as $[\nu^0_{ij}]$ is antisymmetric, and we have seen in the previous
subsection that $\nu^0_{i4}=0$.

\subsubsection{Equality of connection $1$-form matrix entries}

Next, we claim that $\Theta_{3j}=0$ for $j=1,2$  along the boundary. First, if $S$ is the
shape operator for the boundary, given for vector fields $a$, $b$ tangent to the boundary by  $g^\pM(Sa,b)=\Pi(a,b)$, where $\Pi$ is the second fundamental form of the boundary, then $\Theta(a)e_4=-Sa$, because $\n^0_{a}e_4=0$ and $g(\n_a e_4,b)=-g(e_4,\n_a b)$.
Furthermore, it follows that $\Theta(e_4)=-S$, as one can see by extending $a\in T\pM$
away from the boundary using the flow of $e_4$, and also because $e_4$ is geodesic with
respect to both $g$ and $g_0$.
Hence for $i$, $j<4$
\[
\Theta_{ij}(e_4)=g^\pM(\sum_k\Theta_{ik}(e_4)e_k, e_j)=g^\pM(\Theta(e_4)e_i,e_j)=g^\pM(-Se_i,e_j)=-\Pi(e_i,e_j).
\]
But for $i=3$ and $j=1$ or $j=2$, the covariant derivative formulas \Ref{conn} imply that $\Pi(e_i,e_j)$,
which is determined by the component of $\nab_{e_i}e_j$ normal to the boundary, is zero.
Together with the previously shown relations $\imath^*\Theta_{ij}=0$,  $i$, $j<4$, we obtain
the claim.

It follows
\be\lb{nus}
\nu^0_{31}=\nu_{31},\qquad \nu^0_{32}=\nu_{32}\qquad \mathrm{on}\ \pM.
\end{equation}

\subsubsection{The endomorphism $\n X$}

Next, we compare \Ref{nab0X} to $\n X$, which we compute here independently, ignoring \Ref{nab-X}:
\begin{equation}\lb{nabX}
\begin{aligned}
\n X&=\n (\sqrt{Q}e_3)=d(\sqrt{Q})e_3+\sqrt{Q}\nu_{3j}\otimes e_j\\
&=(1/(2\sqrt{Q}))Q'(-\sqrt{Q})e^4_3
+\sqrt{Q}(\nu_{31}\otimes e_1+\nu_{32}\otimes e_2+\ell e^3_4)\\
&=\psi(e^3_4-e^4_3)+\sqrt{Q}(\nu_{31}\otimes e_1+\nu_{32}\otimes e_2),
\end{aligned}
\end{equation}
where in the last step we have used the second equation in \eqref{rels}.
Note, of course that the last term in \Ref{nabX} is just $\phi(e^1_2-e^2_1)$,
in accordance with \Ref{nab-X}.

\subsubsection{Conclusion}
By \Ref{nus}, {\em on the boundary},
the second term in the last line of \Ref{nabX} is identical to the right hand side of \Ref{nab0X}.
Putting \Ref{nab0X} and \Ref{nabX} together, we thus see that
$\n^t\!X=(1-t)\n^0\!X+t\n X$ is given on $\pM$ by
%
\be\lb{ntx}[(\nab^t\!X)_{ij}]=
\begin{bmatrix}
      0   &   \phi  &   0     & 0  \\[4pt]
  -\phi   &   0     &   0     & 0  \\[4pt]
      0   &   0     &   0     & t\psi \\[4pt]
      0   &   0     &  -t\psi &    0
\end{bmatrix},
\end{equation}
so that \be\lb{ntx1} \n^t X=\phi J_\cal{H}+t\psi J_\cal{V}. \end{equation}
Here and further on,
for notational ease, we write $\phi$, $\psi$ rather than their values $\phi_0$, $\psi_0$ at $\t=0$.

\subsection{The curvature of $\n^t$}

We now derive the expression for the curvature $R^t$ for the linear family $\n^t$ associated to a fibered SKR metric
$g$.

\subsubsection{The decomposition}

To analyze the curvature of $\n^t=\n^0+t\Theta$, we follow Moroianu \cite{mor}, who considers
the decomposition
\be\lb{curv-decom}
R^t=(\n^t)^2=R^0+td^{\n^0}\!\Theta+t^2\Theta^2.
\end{equation}
We are interested in the pull-back to the boundary of these terms,
which we denote
\be\lb{pull-curv}
A^1:=\imath^*R^0, \quad A^2:=\imath^*d^{\n^0}\!\Theta, \quad A^3:=\imath^*\Theta^2.
\end{equation}
Of these, $A^3$ is simplest to compute since we have $\imath^*\Theta$. Next, since
$g_0$ is a product metric near $\pM$,
\be\lb{A1}
A^1=R^{\del M}=\sum_{\substack{a,b,c,d=1 \\a<b,\, c<d}}^3e^{ab}\otimes[R^0_{abcd}(e^c_d-e^d_c)].
\end{equation}

\subsubsection{Formula for $A^2$}

For $A^2$ we note the following, which is similar to \cite{mor}. First, the integral curves of
$v$ are pre-geodesics of $g$. This is obvious from the Hessian clause in definition \ref{SKRdef},
and discussed in \cite{dm2}. Therefore, as was already alluded to in Remark \ref{gp-exp}, the
integral curves of the normalization $e_4$ of $v$ are geodesics of $g$ normal to the boundary,
and via this flow one can define a $U(1)$-equivariant collar in $M$.
These integral curves are also geodesics of the product metric $g_0$ on the collar,
and $e_4$ is parallel with respect to $g_0$. This last fact, together
with formula \Ref{ithet} for $\imath^*\Theta$ shows that $g(d^{\nab^0}\Theta(e_i,e_j)e_k,e_l)=0$ if $i,j,k,l<4$.
Such vanishing is also easily shown if $k=l=4$.
Next, for $i,j,k<4$,
\be\lb{codaz}
g(d^{\nab^0}\Theta(e_i,e_j)e_k,e_4)=R_{ijk4},
\end{equation}
by the Codazzi-Mainardi equation for $\pM$.
Therefore, we have
\be\lb{A2}
A^2=\sum_{\substack{a,b,c<4\\a<b}} e^{ab}\otimes[R_{abc4}(e^c_4-e^4_c)].
\end{equation}

\subsubsection{Further details}

We supply some details on Moroianu's observation \Ref{codaz}, to make clear how a relation comes about
between the Levi-Civita connection of $g_0$ and the curvature of $g$.
We already mentioned the following relations between the second fundamental form $\Pi$, the shape
operator $S$ of $\pM$ and $\Theta$, namely
\[
\Theta(a)b=\Pi(a,b)e_4, \quad \Theta(a)e_4=-Sa,\quad \Theta(e_4)=-S, \quad a,b\in T\pM.
\]
A further somewhat tedious but straightforward verification then shows that for $\{i,j,k\}\subset\{1,2,3\}$, $g((d^{\n^0}\!\Theta)(e_i,e_j)e_k,e_4)=g((d^{\n^0}\!S(e_i,e_j),e_k)$.

To show the relation to the curvature of $g$, we compute
\begin{equation*}
\begin{aligned}
(d^{\n^0}&\!S)(a,b)
=\nabla^0_a(S(b))-\nabla^0_b(S(a))-S([a,b])\\
&=\nabla_a(S(b))-\Pi(a,S(b))e_4-[\nabla_b(S(a))-\Pi(b,S(a))e_4]-S([a,b])\\
&=\nabla_a(S(b))-S(\nabla_ab)-[\nabla_b(S(a))-S(\nabla_ba)]
-g^\pM(S(a),S(b))e_4+g^\pM(S(b),S(a))e_4\\
&=(\nabla_aS)(b)-(\nabla_bS)(a).
\end{aligned}
\end{equation*}
This, together with $g^\pM((\nabla_aS)b,c)=g^\pM((\nabla_a\tilde{\Pi})(b,c),e_4)$
for $\tilde{\Pi}(a,b)=\Pi(a,b)e_4$ (see for instance the proof of
Cor. 34 in chapter 4 of \cite{on}) and the Codazzi-Minardi equation
\[
g^\pM((\nabla_a\tilde{\Pi})(b,c),e_4)-g^\pM((\nabla_b\tilde{\Pi})(a,c),e_4)=g^\pM(R(a,b)c,e_4)
\]
yields \Ref{codaz}.

\subsection{The transgression form: the most complex term}

In the next two subsections we will compute the pull-backs to the boundary of the four trace terms appearing
in formula \eqref{transg}, and then put them together to calculate $TL_\grg$ on a fibered SKR $4$-manifold
using Proposition \ref{4transg}. We compute here in detail the most complex of these, and in the next subsection
simply give the formulas for the other three, as they involve similar but easier calculations. In order to
simplify the notations, we will write $\phi$, $\psi$, $Q$ in place of $\phi_0$, $\psi_0$, $Q_0$ throughout
all calculations.

\subsubsection{Considerations unrelated to curvature}

The most complex of the trace terms in \Ref{transg} is the last one, involving
the non-commutative analogue of the second derivative of $f$. Given that $f$ is even, this term is
\be\lb{mixed-R-T}
\mathrm{Tr}\left[(f^{[2]}(\nab^t X)*\Theta)R^t\right]=\sum_{m=0}^\infty\sum_{k=0}^{2m}\fr {f^{(2m+2)}(0)}{(2m+1)!}\mathrm{Tr}\left[(\nab^t\!X)^{2m-k}\Theta (\nab^t\!X)^{k}R^t\right].
\end{equation}
Note that the two powers in which $\nab^t\!X$ appears in each summand are either
both even or both odd. Letting $I_\cal{H}$, $I_\cal{V}$ denote the projection operators onto $\cal{H}$,
$\cal{V}$, respectively, and $J_\cal{H}$, $J_\cal{V}$ the composition of $J$ with these operators,
\eqref{ntx1} gives
\be\lb{powers}
(\nab^t\!X)^{k}=
\begin{cases}
(-1)^{k/2}\phi^{k}I_\cal{H}+(-1)^{k/2}(t\psi)^{k}I_\cal{V},\quad \text{$k$ even,}\\
(-1)^{(k-1)/2}\phi^{k}J_\cal{H}+(-1)^{(k-1)/2}(t\psi)^{k}J_\cal{V},\quad \text{$k$ odd.}
\end{cases}
\end{equation}
Thus, using \eqref{thet}, for $k$ even
\begin{multline}\lb{all-but-Rt}
[((\nab^t\!X)^{2m-k}\imath^*\Theta (\nab^t\!X)^{k})_{ij}]=\\
\begin{bmatrix}
0 & 0 & 0 &(-1)^m\phi^{2m-k}(t\psi)^k\mathbf{p}\\
0 & 0 & 0 &(-1)^m\phi^{2m-k}(t\psi)^k\mathbf{q}\\
0 & 0 & 0 &(-1)^m(t\psi)^{2m}\mathbf{r}\\
(-1)^{m+1}(t\psi)^{2m-k}\phi^k\mathbf{p} & (-1)^{m+1}(t\psi)^{2m-k}\phi^k\mathbf{q}
& (-1)^{m+1}(t\psi)^{2m}\mathbf{r} & 0
\end{bmatrix},
\end{multline}
whereas for $k$ odd
\begin{multline}\lb{all-but-Rt-2}
[((\nab^t\!X)^{2m-k}\imath^*\Theta (\nab^t\!X)^{k})_{ij}]=\\
\begin{bmatrix}
0 & 0 & (-1)^m\phi^{2m-k}(t\psi)^k\mathbf{q} & 0\\
0 & 0 & (-1)^{m-1}\phi^{2m-k}(t\psi)^k\mathbf{p} & 0\\
(-1)^{m-1}(t\psi)^{2m-k}\phi^k\mathbf{q} & (-1)^{m}(t\psi)^{2m-k}\phi^k\mathbf{p}
& 0 & (-1)^{m}(t\psi)^{2m}\mathbf{r}\\
0 & 0 &(-1)^{m-1}(t\psi)^{2m}\mathbf{r} & 0\\
\end{bmatrix}.
\end{multline}
To complete the calculation of the trace part in each summand of \eqref{mixed-R-T}, we need to multiply
these matrices by the one corresponding to $\imath^*R^t$ and take the trace. This we do by splitting into cases,
according to the decomposition \Ref{curv-decom}, that is, the terms $A^i$ of \Ref{pull-curv}. Together with the
parity of $k$, this gives six cases to consider. We analyze each of them below.

\subsubsection{$A^1$ and $A^3$ for $k$ even}

Following Moroianu \cite{mor} we set
\[
\imath^*\Theta=\Theta'-\Theta''
\]
where $\Theta'=\sum_{i=1}^2k e^i\otimes e^i_4+\ell e^3\otimes e^3_4$ and
$\Theta''$ is the transpose of $\Theta'$. Then, as $R^t$
and $\nab^t\!X$ are skew-symmetric, the pull-back under $\imath$
of each trace term in \eqref{mixed-R-T} can be written as
\be\lb{th-prime}
-\mathrm{Tr}\left[(\nab^t\!X)^k\Theta'' (\nab^t\!X)^{2m-k}\imath^*R^t\right]
-\mathrm{Tr}\left[(\nab^t\!X)^{2m-k}\Theta'' (\nab^t\!X)^{k}\imath^*R^t\right],
\end{equation}
Now the endomorphism components of
$A^1:=\imath^*R^0=R^\pM$ and $A^3:=\imath^*\Theta^2$ map $T_pM$ to $T_p\pM$.
For the latter this can be seen as
\be\lb{thet2}
[(\imath^*\Theta^2)_{ij}]=
\begin{bmatrix}
0                        &-\mathbf{p}\we\mathbf{q} & -\mathbf{p}\we\mathbf{r} & 0\\
-\mathbf{q}\we\mathbf{p} &            0            & -\mathbf{q}\we\mathbf{r} & 0\\
-\mathbf{r}\we\mathbf{p} &-\mathbf{r}\we\mathbf{q} &               0          & 0\\
0                        &            0            &               0          & 0
\end{bmatrix}.
\end{equation}
On the other hand, the endomorphism component of $\Theta''$ vanishes when restricted to $T\pM$.
When $k$ is even, the matrices of powers of $\n^t X$ are diagonal and so map $T\pM$ to itself.
Combining the last three statements we conclude that when $k$ is even, \eqref{th-prime} vanishes
if its curvature term is replaced by either $A^1$ or $A^3$. Thus
\be\lb{k-evA13}
\text{for $k$ even,\qquad
$\mathrm{Tr}\left[(\nab^t\!X)^{2m-k}\imath^*\Theta (\nab^t\!X)^{k}A^i\right]=0$\qquad if\ \ $i=1,3$.}
\end{equation}
This can also be seen directly by calculating with \eqref{all-but-Rt} to get an off-diagonal matrix.

\subsubsection{$A^3$ for $k$ odd}

A direct matrix product calculation using \eqref{all-but-Rt-2} and \Ref{thet2} gives
\be\lb{k-odA3}
\text{for $k$ odd,\qquad
$\mathrm{Tr}\left[(\nab^t\!X)^{2m-k}\imath^*\Theta (\nab^t\!X)^{k}A^3\right]=
(-1)^{m-1}2(\phi^{2m-k}(t\psi)^k+(t\psi)^{2m-k}\phi^k)
\mathbf{p}\we\mathbf{q}\we\mathbf{r}$.}
\end{equation}

\subsubsection{$A^1$ for $k$ odd}

Next we consider the case where $k$ is odd and the trace is taken with $A^1$, which is given by
\Ref{A1}:
\be\lb{A1-sec}
A^1=R^{\del M}=\sum_{\substack{a,b,c,d=1 \\a<b,\, c<d}}^3e^{ab}\otimes[R^0_{abcd}(e^c_d-e^d_c)].
\end{equation}
For each fixed pair $\{a,b\}$ we consider the matrix corresponding to the term in square brackets in the
summand of \Ref{A1-sec}, with $\{c,d\}$ equal to one of the ordered pairs $\{1,2\}$, $\{1,3\}$, $\{2,3\}$.
We take the product of the matrix \Ref{all-but-Rt-2} with this matrix. The trace of the resulting matrix
is a contribution to the trace appearing in a summand of \Ref{mixed-R-T}.
The computation gives zero trace for the pair $\{1,2\}$, while
the traces for the other two pairs are
\[
(-1)^{m-1}(\phi^{2m-k}(t\psi)^k+(t\psi)^{2m-k}\phi^k)R^0_{ab13}\mathbf{q},\qquad
(-1)^{m}(\phi^{2m-k}(t\psi)^k+(t\psi)^{2m-k}\phi^k)R^0_{ab23}\mathbf{p}.
\]
Recall now that the matrices for all the endomorphisms in these calculations
consist of entries which are differential forms, which are wedged upon matrix multiplication.
Hence, given that $\mathbf{q}$ ($\mathbf{p}$) is a multiple of $e^2$ ($e^1$),
a nonzero contribution will only come when the two curvature coefficients
are, respectively, $R^0_{1313}$ and $R^0_{2323}$. But as mentioned earlier these two
are equal (see \cite{wat}).
Thus, since $\mathbf{q}\we e^{13}=-\mathbf{p}\we e^{23}$,
the above two terms, each wedged with the appropriate $e^{ab}$,
are in fact equal.
We thus finally arrive at
\be\lb{k-odA1}
\text{for $k$ odd,\quad
$\mathrm{Tr}\left[(\nab^t\!X)^{2m-k}\imath^*\Theta (\nab^t\!X)^{k}A^1\right]=
(-1)^{m}2(\phi^{2m-k}(t\psi)^k+(t\psi)^{2m-k}\phi^k)R^0_{2323}
\mathbf{p}\we e^{23}$.}
\end{equation}

\subsubsection{$A^2$ for any $k$}

It remains to consider $A^2$, which is given by \Ref{A2}:
\be\lb{A2-sec}
A^2=\sum_{\substack{a,b,c<4\\a<b}} e^{ab}\otimes[R_{abc4}(e^c_4-e^4_c)].
\end{equation}
We proceed as in the previous case, fixing indices $\{a,b\}$ and calculating the
trace of the composition of \eqref{all-but-Rt} or \eqref{all-but-Rt-2}
with the matrix corresponding to the term in the square bracket of \Ref{A2-sec}, separately for index $c$
equal to $1$, $2$ and $3$. Wedging the result with the appropriate $e^{ab}$
(only one of which gives a nonzero answer) and summing the three contributions yields
\begin{equation}\lb{k-evA2}
\begin{aligned}
\text{for $k$ even,}\quad
\mathrm{Tr}&\left[(\nab^t\!X)^{2m-k}\imath^*\Theta (\nab^t\!X)^{k}A^2\right]\\
&=(-1)^{m+1}(\phi^{2m-k}(t\psi)^k+(t\psi)^{2m-k}\phi^k)R_{2314}
\mathbf{p}\we e^{23}\\
&+(-1)^{m+1}(\phi^{2m-k}(t\psi)^k
+(t\psi)^{2m-k}\phi^k)R_{1324}
\mathbf{q}\we e^{13}\\
&+(-1)^{m+1}2(t\psi)^{2m}R_{1234}
\mathbf{r}\we e^{12}\\
&=(-1)^{m+1}2(\phi^{2m-k}(t\psi)^k+(t\psi)^{2m-k}\phi^k)R_{2314}
\mathbf{p}\we e^{23}\\
&+(-1)^{m+1}2(t\psi)^{2m}R_{1234}
\mathbf{r}\we e^{12},
\end{aligned}
\end{equation}
as the first two summands are equal by \eqref{curv-ind}. When $k$ is odd,
the traces corresponding to index $c=1$ and $c=2$ vanish, and one
arrives at
\be\lb{k-odA2}
\text{for $k$ odd,\quad
$\mathrm{Tr}\left[(\nab^t\!X)^{2m-k}\imath^*\Theta (\nab^t\!X)^{k}A^2\right]
=(-1)^{m-1}2(t\psi)^{2m}R_{1234}
\mathbf{r}\we e^{12}$.}
\end{equation}
Notice that this is the same expression as the last term in \Ref{k-evA2}.

\subsubsection{Conclusion}

Putting \eqref{k-evA13}, \Ref{k-odA3}, \Ref{k-odA1}, \Ref{k-evA2}, \Ref{k-odA2} together,
while recalling our notation $M_{k,m}(a,b)=a^kb^{2m-k}+b^ka^{2m-k}$, we have
\be\lb{biggy}
\begin{aligned}
  \sum_{k=0}^{2m}\imath^*\mathrm{Tr}&\left[(\nab^t\!X)^{2m-k}\Theta (\nab^t\!X)^{k}R^t\right]\\
&=\sum_{k=0}^{2m}\mathrm{Tr}\left[(\nab^t\!X)^{2m-k}\imath^*\Theta (\nab^t\!X)^{k}(A^1+tA^2+t^2A^3)\right]\\
&=\sum_{k\ \mathrm{odd}=1}^{2m-1}\left[M_{k,m}(\phi,t\psi)\left((-1)^{m}2R^0_{2323}
\mathbf{p}\we e^{23}\right.\right.\left.\left.+(-1)^{m-1}2t^2
\mathbf{p}\we\mathbf{q}\we\mathbf{r}\right)\right]\\
&+\sum_{k\ \mathrm{even}=0}^{2m}\left[M_{k,m}(\phi,t\psi)(-1)^{m+1}2tR_{2314}
\mathbf{p}\we e^{23}\right]\\
&+t(2m+1)(-1)^{m-1}2(t\psi)^{2m}R_{1234}
\mathbf{r}\we e^{12},
\end{aligned}
\end{equation}
Then $\imath^*\mathrm{Tr}\left[(f^{[2]}(\nab^t X)*\Theta)R^t\right]$ is obtained
by multiplying this expression by $f^{(2m+2)}(0)/(2m+1)!$ and summing
over $m=0\ldots\infty$.

\subsection{The transgression form: simpler terms}\lb{trans-simp}
It remains to calculate the $\imath$-pullbacks of $\mathrm{Tr}\left[f'(\nab^t\!X) R^t\right]$,
$\mathrm{Tr}\left[\Theta f'(\nab^t\!X)\right]$ and $\exp(\mathrm{Tr}\left[f(\nab^t\!X)\right])$,
which are the remaining terms in formula \Ref{transg} for the equivariant transgression.

For the first of these we proceed as before except that we use (the odd case of)
\eqref{powers} in the matrix multiplications rather than \eqref{all-but-Rt}, \eqref{all-but-Rt-2}.
Decomposing $\imath^*R^t$ as in the previous subsection, we add the traces of the matrix product
with each of the $A^i$, to arrive at
\begin{multline*}
\imath^*\mathrm{Tr}\left[f'(\nab^t\!X)R^t\right]
=\sum_{m=0}^\infty\fr {f^{(2m+2)}(0)}{(2m+1)!}\imath^*\mathrm{Tr}\left[(\nab^t\!X)^{2m+1}R^t\right]\\
=\sum_{m=0}^\infty\fr {f^{(2m+2)}(0)}{(2m+1)!}\left[t^22(-1)^{m}\phi^{2m+1}\mathbf{p}\we\mathbf{q}
+\sum_{\substack{a,b=1\\ a<b}}^3e^{ab}\, 2(-1)^{m+1}\phi^{2m+1}R_{ab12}^0\right.\\
\left.+t\sum_{\substack{a,b<4\\ a<b}}e^{ab}\, 2(-1)^{m+1}(t\psi)^{2m+1}R_{ab34}\right].
\end{multline*}

The remaining two terms are easier, as they do not involve curvature. We have,
\begin{multline*}
\imath^*\mathrm{Tr}\left[\imath^*\Theta f'(\nab^t\!X)\right]
=\sum_{m=0}^\infty\fr {f^{(2m+2)}(0)}{(2m+1)!}\mathrm{Tr}\left[\Theta (\nab^t\!X)^{2m+1}\right]
=\sum_{m=0}^\infty\fr {f^{(2m+2)}(0)}{(2m+1)!}2(-1)^{m+1}(t\psi)^{2m+1}\mathbf{r}
\end{multline*}
and, as $f$ is even,
\begin{multline*}
\mathrm{Tr}\left[f(\nab^t\!X)\right]
=\sum_{m=0}^\infty\fr {f^{(2m)}(0)}{(2m)!}\mathrm{Tr}\left[(\nab^t\!X)^{2m}\right]\\
=\sum_{m=0}^\infty\fr {f^{(2m)}(0)}{(2m)!}(-1)^m2(\phi^{2m}+(t\psi)^{2m})
=2\left(f(i\phi)+f(it\psi)\right).
\end{multline*}

\subsection{Assembling the constituents of the equivariant transgression formula in the irreducible case}\lb{fin}
We now gather all the work of the last two subsections and obtain formula \Ref{final-trns}
for the pull-back to the boundary of the degree $3$ component of the
equivariant transgression of an irreducible fibered SKR $4$-manifold.

We recall here the integrand in formula \eqref{transg}
for the transgression of the equivariant $L$-form
in dimension four:
\[
\exp\left(\mathrm{Tr}\left[f(\nab^t\!X)\right]\right)\!
\left(\mathrm{Tr}\left[\Theta f'(\nab^t\!X)\right]\,\mathrm{Tr}\left[f'(\nab^t\!X)R^t\right]+
\mathrm{Tr}\left[\left(f^{[2]}(\nab^t\!X)*\Theta\right) R^t\right]\right)
\]
Set $a_m=f^{(2m+2)}(0)/(2m+1)!$. From \Ref{biggy} and the four formulas of Subsection \ref{trans-simp}
we arrive at the following expression, giving the pull-back to the boundary of the degree $3$ component of
the transgression of the equivariant $L$-form of an irreducible SKR metric on a fibered SKR $4$-manifold
with boundary $\{\tau=0\}$:
\begin{equation*}
\begin{aligned}
\imath^*TL_\grg(g,g_0)(X)_{[3]}&=\bigintsss_0^1\bigg[\exp\left(2\left(f(i\phi)
+f(it\psi)\right)\right)\cdot\\
&\cdot\bigg[\mathlarger{\mathlarger{\sum}}_{m,n=0}^\infty\!\! a_ma_n\bigg(
2(-1)^{m+1}(t\psi)^{2m+1}\cdot
\Big(t^22(-1)^{n}\phi^{2n+1}\mathbf{r}\we\mathbf{p}\we\mathbf{q}\\
&+2(-1)^{n+1}\phi^{2n+1}R_{1212}^0\,\mathbf{r}\we e^{12}
+t2(-1)^{n+1}(t\psi)^{2n+1}R_{1234}\,\mathbf{r}\we e^{12}\Big)\bigg)\\
&+\mathlarger{\mathlarger{\sum}}_{m=0}^\infty a_m\bigg(\!
\mathlarger{\mathlarger{\sum}}_{\mathrm{odd}\ k=1}^{2m-1}\!\!
\Big(M_{k,m}(\phi,t\psi)\left((-1)^{m}2R^0_{2323}
\mathbf{p}\we e^{23}\right.\left.+(-1)^{m-1}2t^2
\mathbf{p}\we\mathbf{q}\we\mathbf{r}\right)\Big)\\
&+\mathlarger{\mathlarger{\sum}}_{\mathrm{even}\ k=0}^{2m}
\!\!\Big(M_{k,m}(\phi,t\psi)(-1)^{m+1}2tR_{2314}
\mathbf{p}\we e^{23}\Big)\\
&+t(2m+1)(-1)^{m-1}2(t\psi)^{2m}R_{1234}
\mathbf{r}\we e^{12}
\bigg)\bigg]\bigg]\,dt.
\end{aligned}
\end{equation*}

To obtain formula \Ref{final-trns} from this, we recall
that as $f$ is even, we have $f'(ix)=i\mathrm{Im}f'(ix)$.
Using the power series for $f'(ix)$ and $f''(ix)$,
along with the substitutions
\Ref{pqr} and \Ref{kl} for $\pP$, $\qQ$ and $\rR$
and $\phi_0$, $\psi_0$, $Q_0$ for $\phi$, $\psi$, $Q$, respectively,
we verify \Ref{final-trns} after some simple algebra.

\subsection{Vanishing of the equivariant transgression in the reducible case}\label{SS:vanishing}\lb{end}

We show here the vanishing of the  pull-back to the boundary of the degree $3$ component of the
equivariant transgression of a reducible fibered SKR $4$-manifold.

Recall that for a reducible SKR metric, the metric expression
differs from the irreducible case \eqref{SKR} only with regard to the factor
multiplying the base metric pull-back $\pi^*h$: it is $1$ rather
than $2|\t-\bar{c}|$. Equations \Ref{conn} for the Levi-Civita connection are still
valid. The only resulting difference in the
curvature components of \Ref{curv-ind} is in $R_{1212}$,
which does not appear in \Ref{final-trns}.
Thus, for reducible metrics, one can formally rederive formula \Ref{final-trns}
by exactly the same procedure with no resulting changes (although
the explicit formulas for the curvature $R^0$ of $g^0$ in Proposition \ref{TLg}
will be different). However, in actuality the calculations and the result
are much simpler, as on the regular set of $\t$,
$\phi=0$ identically in the reducible case, so that $R_{1234}$, $R_{2314}$
also vanish. Formula \Ref{final-trns} in the reducible case thus takes the form
\be
\begin{aligned}
\imath^*TL_\grg(g,g_0)(X)_{[3]}&=-\bigintsss_0^1
\bigg[\exp\left(2\left(f(0)
+f(it\psi_0)\right)\right)\cdot
\bigg(4\fr{\psi_0}{\sqrt{Q_0}}f'(it\psi_0)R^0_{1212}f'(0)
\bigg)e^{123}\bigg]\,dt.
\end{aligned}
\end{equation}
However, our $f$ is also even, so that $f'(0)=0$
and thus in the reudcible case, i.e when the SKR metric $g$
is a local product of K\"ahler factors, with $\cal{H}$, $\cal{V}$
of \Ref{HVvu} tangent to these factors,
\[
\imath^*TL_\grg(g,g_0)(X)_{[3]}=0.
\]
Note that such a metric is generally not a ``product metric" in the
sense used for $g_0$.

This concludes the proof of Proposition \ref{TLg}.
Note that the vanishing of this pulled-back component of the equivariant transgression
form in the reducible SKR case, and equation \Ref{final-trns} in the irreducible SKR
case, are valid not just for the Hirzebruch
$L$-form, but for any germ of an analytic even function $f$.

Given this vanishing result, and the one for the degree four component of
$L_\grg(g)(X)$  mentioned at the end of Subsection \ref{eq-L}, we see
from \Ref{E:signuntwisted} that in the reducible fibered
SKR case, the infinitesimal equivariant eta invariant is minus the signature.
However, for reducible fibered SKR metrics the manifold $M$ is a flat disc bundle over
a Riemann surface, so that the signature, and therefore this eta invariant, are also zero.
Summarizing
\begin{Proposition}\lb{van-eta}
For a reducible SKR metric on the total space of a disc bundle $M$ over a compact oriented surface,
let $D^+_\ppM$ be the restriction of the odd signature operator to the even forms. Then
\[
\eta_X(D^+_\ppM)=0.
\]
\end{Proposition}

It is not clear to us whether this result can be obtained via other means.
In the non-equivariant case, for example, one method for proving the vanishing of the eta
invariant is to show that the spectrum of $D^+_\ppM$ is symmetric about the origin. Often, this involves finding an orientation-reversing isometry, and deducing from its existence this spectral symmetry. If the SKR manifold $M$ is a global K\"ahler product, such an isometry can indeed be constructed. However, reducible SKR metrics also arise as local K\"ahler products which are not global, induced on flat disk bundles. As these restrict to local, but non-global products on the boundary circle bundle, it is not obvious to us how to obtain such an isometry.

The second difficulty in finding an alternative proof of our vanishing result is that symmetry of the spectrum is not sufficient for the vanishing of the equivariant eta-invariant. At least for Donnelley's equivariant eta invariant, one needs the restriction of $D^+_\ppM$ to each isotypical  component of the space of sections (i.e. a subspace of sections on which the circle acts with a given weight) to have a symmetric spectrum. To prove this one would need to construct an isometry which commutes with the circle action. It is not clear to us how to construct such an isometry even in the global product case.

\appendix

\section{Deformation of the Dirac bundle compatible with a family of metrics}\label{S:pr of deformation}

In this appendix we present a proof of Proposition~\ref{P:deformation Dirac}.

\subsection{Restriction of the Clifford action to $\pM\times\{u\}$}\label{SS:octu}
Recall that we denote by $E$ the restriction of $\E$ to $\pM$ and identify the restriction of $\E$ to $U=\pM\times(-\infty,0]\subset M$ with the product $E\times(-\infty,0]$. As in \eqref{E:octu}, for $u\in (-\infty,0]$, $y\in \pM$, $\xi\in T^*_{(y,u)}M$, we define a family of maps $\oc_u(\xi):E_y\to E_y$ by
\begin{equation}\label{E:ocu}
	c(\xi)\cdot(e,u)\ = \ \big(\, \oc_u(\xi)\cdot e,u\,\big),
	\qquad e\in E_y.
\end{equation}
Then $\oc_u(\xi)^2= -g^M(\xi,\xi)$ for all $y\in \pM,\ \xi\in T^*_{(y,u)}M$.

Our  first goal is to construct a family $\oc_{t,u}$ ($t\in \RR, \ u\in (-\infty,0]$) such that  $\oc_{t,u}(\xi)^2= -g^M_t(\xi,\xi)$ for all $y\in \pM,\ \xi\in T^*_{(y,u)}M$ and such that the action $c_t$ defined by  \eqref{E:octu} satisfies the conditions of Proposition~\ref{P:deformation Dirac}.

\subsection{Parallel transport on $T^*M$ associated to the metric $\gM_t$}\label{SS:parallel transport}
We denote by $\n^{LC}_t$ the Levi-Civita connection on $T^*M$ associated to the metric $\gM_t$. For $y\in\pM$ and $u_1,\ u_2\in (-\infty,0]$, let
\[
	\Phi_{t;y}^{u_1,u_2}:\, T^*_{(y,u_1)}M\ \to \ T^*_{(y,u_2)}M
\]
denote the parallel transport along the segment $\{y\}\times[u_1,u_2]$ with respect to this connection. This is an isometry, i.e
\begin{equation}\label{E:Phinorm}
	\big\|\,\Phi_{t;y}^{u_1,u_2}(\xi)\,\big\|_t \ = \ \|\xi\|_t,
	\qquad \xi\in T^*_{(y,u_1)}M,
\end{equation}
where
\[
	\|\xi\|_t\ := \ \sqrt{g^M_t(\xi,\xi)}.
\]

\subsection{The family of Clifford actions on $\pM\times(-\infty,0]$}\label{SS:consruction ct}
Let $r:\RR\to \RR$ be a smooth non-decreasing function such that
\[
	r(u)\ = \ \begin{cases}
		0, \quad &\text{for all} \ \ u\,\le\, -2;\\
		1, \quad &\text{for all} \ \ u\,\ge\, -1.
	\end{cases}
\]
Set
\[
	\phi(t,u)\ := \ t^{r(u)}, \qquad 0\le{}t\le1, \ \ u\le0.
\]
Then
\begin{equation}\label{E:phitu}
	\phi(t,u)\ = \
	\begin{cases}
	   t, \quad &\text{if}\ \ u\ge -1;\\
	   1, \quad &\text{if}\ \ u\le -2.
	\end{cases}	
\end{equation}

For $u\le0$, $y\in \pM$, and $\xi\in T^*_{(y,u)}M$, set
 \begin{equation}\label{E:construction ct}
	\oc_{t;u}(\xi)\ := \ \oc_0\big(\Phi^{u,0}_{\phi(t,u);y}(\xi)\big),
\end{equation}
and define the action $c_t:T^*\Big(\pM\times(-\infty,0]\Big)\to \End(\E)$ by \eqref{E:octu}.

\begin{lemma}\label{L:octu2}
For all $0\le{}t\le1$, $u\le0$, $y\in \pM$, and $\xi\in T^*_{(y,u)}M$ we have
\begin{equation}\label{E:octu2}
	\oc_{t,u}(\xi)^2\ = \ -\,\|\xi\|^2_t.
\end{equation}
\end{lemma}

\begin{proof}
First, we consider the case $u\ge-1$. Then $\phi(t,u)=t$ and
\begin{equation}\label{E:octu2 u>-1}
	\oc_{t;u}(\xi)^2\ := \ \oc_0\big(\Phi^{u,0}_{t;y}(\xi)\big)^2
	\ = \ -\,g^M\big(\Phi^{u,0}_{t;y}(\xi),\Phi^{u,0}_{t;y}(\xi)\big)
	\ = \ -\,\big\|\Phi^{u,0}_{t;y}(\xi)\big\|_t^2,
\end{equation}
where in the last equality we used that (cf. Subsection~\ref{SS:properties gMa}(v))
$g^M_t\big|_{\pM}= g^M\big|_\pM$ for all $t$. The equality \eqref{E:octu2} follows now from \eqref{E:Phinorm}.

Consider now the case when $u<-1$. Recall from Subsection~\ref{SS:properties gMa}(iii) that the restriction of $g^M_\tau$ to $\pM\times(-\infty,-1)$ is independent of $\tau\ge0$ and is equal to the restriction of $g^M$.
In particular,
\begin{equation}\label{E:xi2tau2}
		\|\xi\|_{\phi(t,u)} \ = \ \|\xi\|_t\ = \ \|\xi\|_1, \qquad\text{for all}\quad 0\le{}t\le1, \ \ u<-1.
\end{equation}
Hence, as in \eqref{E:octu2 u>-1}, we obtain
\begin{multline}\notag
	\oc_{t;u}(\xi)^2\ := \ \oc_0\big(\Phi^{u,0}_{\phi(t,u);y}(\xi)\big)^2
	\ = \ -\,\big\|\Phi^{u,0}_{\phi(t,u);y}(\xi)\big\|^2_1
	\\ = \ -\,\big\|\Phi^{u,0}_{\phi(t,u);y}(\xi)\big\|^2_{\phi(t,u)}
	\ \overset{\text{by} \ \eqref{E:Phinorm}}= \
	-\,\|\xi\|^2_{\phi(t,u)} \ \overset{\text{by} \ \eqref{E:xi2tau2}}= \
	-\,\|\xi\|^2_t.
\end{multline}
\end{proof}

\begin{lemma}\label{L:ct(Phi)}
Fix $0\le{}t\le1$, $u_0\le0$, and $y\in \pM$. Assume that either $u_0\le-2$ or $t=1$. Then for all  $\xi\in T^*_{(y,u_0)}M$, we have
\begin{equation}\label{E:ct(Phi)}
	c_t(\xi)\ = \ c(\xi).
\end{equation}
\end{lemma}

\begin{proof}
By \eqref{E:phitu}, the conditions of the lemma imply that $\phi(t,u_0)=1$. Thus
\[
	\oc_{t,u_0}(\xi) \ = \ \oc_0\big(\Phi_{1;y}^{u_0,0}(\xi)\big).
\]
For $u\le0$ set
\[
	\xi(u)\ := \ \Phi_{1;y}^{u_0,u}(\xi) \ \in \ T^*_{(y,u)}M.
\]
Then $\xi(0)= \Phi^{u_0,0}_{1;y}(\xi)$, \/ $\xi(u_0)=\xi$. Recall that $g^M_1= g^M$ and $\n^{LC}_1= \n^{LC}$. Hence,
\begin{equation}\label{E:nabla xi(t)}\notag
	\n^{LC}_{\frac{\p}{\p u}}\, \xi(u) \ = \ 0.
\end{equation}

Recall from \eqref{E:flat(e,u)} that for every $e\in E$ we have $\n^\E_{\frac{\p}{\p u}}(e,u)=0$.  Using \eqref{E:Clifford connection}, we obtain
\begin{multline}\label{E:transport on M}\notag
	\frac{d}{d u}\Big(\, \oc_{u}\big(\xi(u)\big)\cdot e,u\,\Big)\ = \
 	\frac{d}{d u}\Big(\, c\big(\xi(u)\big)\cdot(e,u)\,\Big)\\ = \
 	c\big(\n^{LC}_{\frac{\p}{\p u}}\xi(u)\big)\cdot (e,u)\ + \
 	c\big(\xi(u)\big)\cdot\n^\E_{\frac{\p}{\p u}}(e,u)\ = \ 0.
 \end{multline}
Hence,
\[
	\oc_{t,u_0}\big(\xi\big)\cdot e \ = \
	\oc_{0}\big(\xi(0)\big)\cdot e
	\ = \ \oc_{u_0}\big(\xi(u_0)\big)\cdot e \ = \
	\oc_{u_0}(\xi)\cdot e.
\]
\end{proof}

\subsection{The family of Clifford actions on $M$}\label{SS:ct on M}
Lemma~\ref{L:octu2} implies that for all $0\le{}t\le1$ the map  $c_t:T^*\Big(\pM\times(-\infty,0]\Big)\to \End(\E)$ defined by \eqref{E:octu} defines a Clifford module structure on $\E$ compatible with the metric $g^M_t$. Lemma~\ref{L:ct(Phi)} implies that the restriction of this action to the cylinder $\pM\times(-\infty,-2)$ coincides with the original Clifford action $c(\xi)$. We now  extend $c_t$ to the whole manifold $M$, by setting $c_t(\xi)= c(\xi)$ for all $\xi\in T^*(M\backslash{}U)$. This is a smooth family of $G$-invariant Clifford actions, and for each $0\le{}t\le1$ the action $c_t$ is compatible with the metric $g^M_t$. In particular, $c_t$ satisfies condition (iv) of Proposition~\ref{P:deformation Dirac}. By Lemma~\ref{L:ct(Phi)}, the action $c_1$ is equal to the original action $c$, thus $c_t$ satisfies condition (i) of Proposition~\ref{P:deformation Dirac}.

It also follows directly from the construction of $c_t$ that the restriction of $c_t$ to the boundary is independent of $t$ and coincides with the restriction of $c$. This proves  part (ii) of Proposition~\ref{P:deformation Dirac}.

For $t<2/3$ the metric $g^M_t$ is a product on $\pM\times(-2/3+t,0]$, cf. Subsection~\ref{SS:properties gMa}(iv). Hence, it follows immediately from the construction of $c_t$ that $c_t$ is a product on $\pM\times(-2/3+t,0]$ (cf. Definition~\ref{D:product Dirac}(ii) for the definition of a product Clifford action).

\subsection{A Clifford connection on $\E$}\label{SS:Clifford connection}
To finish the proof of Proposition~\ref{P:deformation Dirac} it now remains to construct  a family $\n^\E_t$ of Clifford connections on $\E$ compatible with the Clifford action $c_t$, which  is a product near the boundary for $t=1$.

Corollary~3.41 of \cite{bgv} states that for every Clifford action on $\E$ there exists a compatible Clifford connection. This is proven by first constructing such a connection locally and then patching the local Clifford connections together by means of a partition of unity.  The same construction shows that given a smooth family of Clifford actions $c_t$, there exists a smooth family of Hermitian connections $\widehat{\n}^\E_t$, such that for each $t$ the connection $\widehat{\n}^\E_t$ is compatible with $c_t$. In this case we say that the family $\widehat{\n}^\E_t$ is compatible with the family $c_t$.

Let $\chi:M\to [0,1]$ be a smooth function such that
\[
	\chi(x)\ = \
	\begin{cases}
		1, \qquad &\text{for} \quad
		x\in M\backslash{}U\sqcup\big(\,\pM\times(-\infty,-3]\,\big);\\
		0 \qquad &\text{for}\quad x\in \pM\times[-2,0],
	\end{cases}
\]
and set
\[
	\tiln^\E_t\ := \ \chi\,\n^\E\ + \ (1-\chi)\,\widehat{\n}^E_t.
\]
The restrictions of $\tiln^\E_t$ to the cylinder $\pM\times[-2,0]$ is equal to $\widehat{\n}^\E_t$ and, hence, is compatible with $c_t$. Since $c_t=c$ on $M\backslash{}U\sqcup\big(\pM\times(-\infty,-2)\big)$, the restrictions of both connections $\n^\E$ and $\widehat{\n}^\E_t$ to $M\backslash{}U\sqcup\big(\pM\times(-\infty,-2)\big)$ are compatible with $c_t$. Hence, so is their convex linear combination $\tiln^\E_t$. We conclude that $\tiln^\E_t$ is compatible with $c_t$ and its restriction to $M\backslash{}U\sqcup\big(\pM\times(-\infty,-3]\big)$ is equal to the original connection $\n^\E$.

Let $\tilde{\n}^\E_t$ be an arbitrary family of connections compatible with $c_t$ whose restriction to $M\backslash{}U\sqcup\big(\pM\times(-\infty,-3]\big)$ is equal to  $\n^\E$.  By averaging over the action of $G$ we can assume that $\tilde{\n}^\E_t$ is a $G$-invariant connection. Notice that  by construction, $\tiln^\E_t$ satisfies condition (v) of Propositions~\ref{P:deformation Dirac}, but does not necessarily satisfy the rest of the conditions of  this Proposition. Our next step is to deform $\tiln^\E_t$ to a new family of connections $\n^\E_t$ which does satisfy the conditions of Propositions~\ref{P:deformation Dirac}.

Let $\tiln^E$ be a $G$-invariant connection on $E= \E\big|_\pM$ compatible with the restriction of the Clifford action $c$ to the boundary. Let $\tiln^{E\times(-\infty,0]}$ be the connection on $\E\big|_U$ induced by $\tiln^E$. Since for $t<2/3$ the Clifford action $c_t$ is a product on the cylinder $\pM\times(-2/3+t,0]$, we conclude that the restriction of $\tiln^{E\times(-\infty,0]}$ to this cylinder is compatible with $c_t$.

Let $s_t:\RR\to \RR$ be as in Subsection~\ref{SS:deformation of the metric}. Let $\n^\E_t$ denote the family of connections on $\E$, whose restriction to $M\backslash{}U$ is equal to $\tiln^\E_t\big|_{M\backslash{}U}=\n^\E$ and whose restriction to $U= \pM\times(-\infty,0]$ is given by
\[
	\n^\E_t \ := \ s_{t+\frac13}\,\tilde{\n}^\E_t\ + \
	(1-s_{t+\frac13})\, \tiln^{E\times(-\infty,0]}.
\]
If $t\ge2/3$, then $\n^\E_t=\tiln^\E_t$ and, hence, is compatible with $c_t$. If $t<2/3$ then the restriction of $\n^\E_t$ to the cylinder $\pM\times(-\infty, -2/3+t]$ coincides with $\tiln^\E_t$ and is compatible with $c_t$. The restriction of $\n^\E_t$ to  $\pM\times(-2/3+t,0]$ is a convex combination of connections $\tiln^\E$ and $\tiln^{E\times(-\infty,0]}$, both of which are compatible with the restriction of $c_t$ to this cylinder. Hence, $\n^\E_t$ is compatible with $c_t$.

For all $0\le{}t\le1$ the restriction of $\n^\E_t$ to $M\backslash{}U\sqcup\big(\pM\times(-\infty,-3]\big)$ coincides with $\n^\E$.  For $t<1/3$ the restriction of the connection $\n^\E_t$ to the cylinder $\pM\times(-1/3+t,0]$ is equal to the product connection $\tiln^{E\times(-\infty,0]}$. Hence, the connection $\n^\E_t$ satisfies the conditions of Proposition~\ref{P:deformation Dirac}. The proposition is proven.
\hfill$\square$



\end{document}